\documentclass{article}
\usepackage[utf8]{inputenc}
\usepackage{amsmath,amsthm,amsfonts,amssymb,bm, mathdots}
\usepackage{xcolor}
\usepackage{geometry}
\usepackage{tikz,tikz-cd}
\usepackage{ytableau}
\usepackage{authblk}
\usepackage{enumitem}
\usepackage{indentfirst}
\usepackage{url}
\usepackage{fdsymbol}
\usepackage{adjustbox}

\theoremstyle{plain}
\newtheorem{thm}{Theorem}[section]
\newtheorem{prop}[thm]{Proposition}
\newtheorem{cor}[thm]{Corollary}
\newtheorem{lem}[thm]{Lemma}
\newtheorem{example}[thm]{Example}
\newtheorem{definition}[thm]{Definition}
\newtheorem{remark}[thm]{Remark}

\newcommand{\lm}{{\bm{\lambda}/\bm{\mu}}}
\newcommand{\lmp}{{\bm{\lambda}'/\bm{\mu}'}}
\newcommand{\0}{\mathbf{0}}
\newcommand{\1}{\mathbf{1}}
\newcommand{\I}{\bm{I}}
\newcommand{\J}{\bm{J}}
\newcommand{\K}{\bm{K}}
\renewcommand{\L}{\bm{L}}
\newcommand{\g}{g}
\newcommand{\mdy}{\text{deg}_Y}
\newcommand{\latW}{W}
\newcommand{\latP}{P}
\newcommand{\latS}{S}

\newcommand{\A}{\mathbf{A}}
\newcommand{\B}{\mathbf{B}}
\newcommand{\C}{\mathbf{C}}
\newcommand{\D}{\mathbf{D}}

\newcommand{\Z}{\mathbb{Z}}

\newcommand{\CC}{\mathbb{C}}

\newcommand{\mc}[1]{{\mathcal #1}}

\newcommand{\alb}{{\textcolor{blue}{a}} < {\textcolor{red}{b}}}

\DeclareMathOperator{\spp}{sp}
\DeclareMathOperator{\cosp}{cosp}

\DeclareMathOperator{\spin}{spin}
\DeclareMathOperator{\inv}{inv}
\DeclareMathOperator{\coinv}{coinv}

\DeclareMathOperator{\SSYT}{SSYT}
\DeclareMathOperator{\SSSYT}{SSSYT}
\DeclareMathOperator{\SRT}{SRT}
\DeclareMathOperator{\SSRT}{SSRT}
\DeclareMathOperator{\HRS}{HRS}
\DeclareMathOperator{\VRS}{VRS}
\DeclareMathOperator{\weight}{weight}

\DeclareMathOperator{\Inv}{Inv}

\def\diag{ \begin{tikzpicture} \draw[dashed] (-.12,-.12) -- (.42, .42); \end{tikzpicture} }

\newcommand{\nullification}[1]{}

\title{A Vertex Model for Supersymmetric LLT Polynomials}
\author{Andrew Gitlin}
\affil{Department of Mathematics, UC Berkeley USA\\
{\small\ttfamily andrew\_gitlin@berkeley.edu}}
\author{David Keating}
\affil{Department of Mathematics, UW Madison USA\\
{\small\ttfamily dkeating3@wisc.edu}}

\date{}

\begin{document}

\maketitle

\begin{center} \it
The polynomials LLT \\
This time with supersymmetry. \\
The purple and white \\
Are such a delight \\
`Til one must prove the YBE!
\end{center}

\abstract{We describe a Yang-Baxter integrable vertex model, which can be realized as a degeneration of a vertex model introduced by Aggarwal, Borodin, and Wheeler.  From this vertex model, we construct a certain class of partition functions that we show are essentially equal to the super ribbon functions of Lam.  Using the vertex model formalism, we give proofs of many properties of these polynomials, namely a Cauchy identity and generalizations of known identities for supersymmetric Schur polynomials.}

\section{Introduction} 

\indent LLT polynomials were originally introduced by Lascoux, Leclerc, and Thibon (for whom the polynomials are eponymously named) in \cite{LLT}.  They are a class of symmetric polynomials, and can be seen as a $t$-deformation of products of (skew) Schur functions.  The original motivation for LLT polynomials was to study certain plethysm coefficients, and they were defined via a relationship with the Fock space representation of a quantum affine Lie algebra.  The original definition expresses the LLT polynomials as a sum over semistandard $k$-ribbon tableaux, weighted with a spin statistic which arises naturally in this representation \cite{LLT,iij13}.  Bylund and Haiman discovered an alternative way to model LLT polynomials. They constructed a family of symmetric polynomials indexed by $k$-tuples of skew Young diagrams, weighted with an inversion statistic (Remark \ref{rmk:hhl-llt}). The Bylund–Haiman model is described in more detail in \cite{HHLRU}.  The relationship between the two definitions uses the Littlewood quotient map (sometimes called the Stanton-White correspondence) \cite{stanton1985schensted}, which sends semistandard $k$-ribbon tableaux to $k$-tuples of semistandard Young tableaux. 

\indent Supersymmetric LLT polynomials $\mc{G}_{\lambda/\mu}^{(k)}(X_n;Y_m;t)$ were introduced in \cite{Lam} (in which they are called super ribbon functions).  As the name suggests, these polynomials are supersymmetric in the $X$ and $Y$ variables (see Definition \ref{def:supersymmetric}) and specialize to LLT polynomials when $m=0$ (see Remark \ref{rmk:lam-llt}).  The supersymmetric LLT polynomials have several other interesting specializations in addition to the LLT polynomials, including the metaplectic symmetric functions introduced in \cite{BBBG} and the supersymmetric Schur polynomials, which can be realized as characters of certain simple modules of the Lie superalgebra $\mathfrak{gl}(n|m)$ \cite{BereleRegev}.

\indent The goal of the present work is to study supersymmetric LLT polynomials, from the perspective of integrable vertex models.  While the study of integrable systems is a classical subject (see \cite{baxter2016exactly, reshetikhin2010lectures} for example), they have recently enjoyed an advent into the world of (non)symmetric polynomials \cite{brubaker2011schur, WHEELER2016543, BorodinWheeler, tsilevich}. Also known as vertex models, ice models, or multiline queues, these models have been generalized to colored vertex models \cite{borodin2018coloured,brubaker2019colored,brubaker2019colored2, buciumas2020colored, corteel2018multiline, garbali2020modified} and polyqueue tableaux \cite{corteel2018multiline, ayyer2020stationary}.  

\indent It has recently been shown \cite{ABW,CGKM, CYZZ} that the LLT polynomials can be realized as a certain class of partition functions constructed from an integrable vertex model.  In \cite{BBBG}, the authors constructed a vertex model whose partition functions are the metaplectic symmetric functions, a certain specialization of the supersymmetric LLT polynomials.  In this paper, we generalize these results by showing the that there is an integrable vertex model whose partition functions are the supersymmetric LLT polynomials.  Our vertex model is adapted from the work of \cite{CGKM}, and can also be realized as a degeneration of a colored vertex model introduced in \cite{ABW}. Recently, in \cite{CFYZZ21} the authors independently constructed an integrable vertex model whose partition functions are the supersymmetric spin LLT polynomials generalizing the vertex model of \cite{CYZZ}. They prove a Cauchy identity analogous to Thm. \ref{thm:main2} here. See remark \ref{rem:CYZZ} for a discussion of the relationship between the different vertex models.

The layout of our paper is as follows. In Section \ref{sec:ribbon}, we discuss relevant background. We define coinversion LLT polynomials and spin LLT polynomials. We describe how to relate the tuples of semistandard Young tableaux to semistandard ribbon tableaux through the Littlewood quotient map, and we extend this map to the case of super tableaux. 

In Section, \ref{sec:VertexModel}, we introduce the vertex models we will use through the paper. We show that they are integrable in the sense that they satisfy a Yang-Baxter equation, and we define the relevant partition functions that give rise to the supersymmetric LLT polynomials $\mathcal{L}_{\lm}^S$. 

In Section \ref{sec:identities}, we prove a variety of properties of the $\mathcal{L}_{\lm}^S$. The main result is 
\begin{thm} \label{thm:main3}
The polynomials $\mc{L}^S_\lm(X_n;Y_m;t)$ satisfy the following properties.
\begin{enumerate} 
    \item (Symmetry, Lemma \ref{lem:permuterows}) The polynomials $\mc{L}^S_\lm(X_n;Y_m;t)$ are symmetric in the $X$ and $Y$ variables.
    \item (Cancellation, Lemma \ref{lem:cancellation}) $\mc{L}_\lm^S(X_{n-1},r;Y_{m-1},-r;t) = \mc{L}_\lm^S(X_{n-1};Y_{m-1};t)$
    \item (Homogeneity, Lem. \ref{lem:homogeneous}) The polynomial $\mc{L}^S_\lm(X_n;Y_m;t)$ is homogeneous in the $X$ and $Y$ variables of degree $|\lm| = |\bm{\lambda}| - |\bm{\mu}|$.
    \item (Restriction, Lemma \ref{lem:restriction})
    \[ \begin{aligned} 
    &\mc{L}_\lm^S(X_{n-1},0;Y_m;t) = \mc{L}_\lm^S(X_{n-1};Y_m;t), \\
    &\mc{L}_\lm^S(X_n;Y_{m-1},0;t) = \mc{L}_\lm^S(X_n;Y_{m-1};t)
    \end{aligned} \]
    \item (Factorization, Lemma \ref{lem:factorization}) If there exist $\bm{\tau}$ and $\bm{\eta}$ such that
    \[ \lambda^{(i)} = (m+\tau^{(i)}_1,\ldots,m+\tau^{(i)}_n,\eta^{(i)}_1,\ldots,\eta^{(i)}_s) \]
    for all $i$, then 
    \[ \begin{aligned}
    \mc{L}_{\bm{\lambda}}^S(X_n;Y_m;t) = \mc{L}_{\bm{\tau}}(X_n;t) 
    \cdot t^{g(\bm{\eta})} \mc{L}_{\bm{\eta}'}(Y_m;t^{-1}) 
    \cdot \prod_{l=0}^{k-1} \prod_{i=1}^n \prod_{j=1}^m (t^l x_i+y_j).
    \end{aligned} \]
\end{enumerate}
\end{thm}
\noindent Note the first two properties together imply that the polynomials $\mc{L}_{\lm}^S(X_n;Y_m;t)$ are supersymmetric.  In the case where $\bm{\lambda} = (\lambda)$ is a 1-tuple of partitions, one can show that the supersymmetric LLT polynomial $\mc{L}^S_{\bm{\lambda}}(X_n;Y_m;t)$ is exactly the supersymmetric Schur polynomial $s_\lambda(X_n;Y_m)$.  In fact, taking $\bm{\mu} = \0$ and $k=1$ in Theorem \ref{thm:main3}, these properties uniquely characterize the supersymmetric Schur polynomials (see \cite[Section 2.1.2]{Moens} and \cite[Example I.3.23]{macdonald1998symmetric}).  However, we suspect (but do not prove) that the properties in Theorem \ref{thm:main3} do not uniquely characterize the supersymmetric LLT polynomials, even in the case $\bm{\mu} = \0$.

In Section \ref{sec:relate-ls-to-g}, we relate the coinversion supersymmetric LLT polynomials to the ribbon supersymmetric LLT polynomials of Lam. The main result is
\begin{thm}[Proposition \ref{prop:YBEpurplewhite} + Prop. \ref{prop-L-equals-G}] \label{thm:main1}
Suppose the $k$-tuple of skew shapes $\lm$ is the $k$-quotient of the skew shape $\lambda/\mu$.  There is a Yang-Baxter integrable vertex model whose partition function $\mc{L}^S_\lm(X_n;Y_m;t)$ is equal to 
\[ \mc{L}^S_\lm(X_n;Y_m;t) = t^{\square} \mathcal{G}_{\lambda/\mu}^{(k)}(X_n;Y_m;t^{1/2}) \]
for some half-integer $\square \in \frac{1}{2}\mathbb{Z}$, where $\mathcal{G}_{\lambda/\mu}^{(k)}(X_n;Y_m;t)$ is the super ribbon LLT polynomial.
\end{thm}

Finally, in Section \ref{sec:Cauchy}, we show the supersymmetric LLT polynomials satisfy a Cauchy identity. The main result is
\begin{thm} \label{thm:main2}
Let $\bm{\mu}$ and $\bm{\nu}$ be tuples of partitions each with infinitely many parts only finitely many of which are nonzero. Fix positive integers $n,m,p,q$. Then
\begin{equation}
\begin{aligned}
 & \sum_{\bm{\lambda}} t^{d(\bm{\mu},\bm{\lambda})} \mathcal{L}^S_{\bm{\nu}/\bm{\lambda}}(X_n,Y_m;t) \mathcal{L}^S_{\bm{\mu}/\bm{\lambda}}(W_p,Z_q;t) =\\&
  \Omega(X_n,Y_m,W_p,Z_q;t) \sum_{\bm{\lambda}} t^{d(\bm{\lambda},\bm{\nu})} \mathcal{L}^S_{\bm{\lambda}/\bm{\mu}}(X_n,Y_m;t)\mathcal{L}^S_{\bm{\lambda}/\bm{\nu}}(W_p,Z_q;t)
 \end{aligned}
\end{equation}
where
\[
\Omega(X_n,Y_m,W_p,Z_q;t) = \prod_{l=0}^{k-1} \prod_{i,i'=1}^n\prod_{j,j'=1}^m\prod_{\alpha,\alpha'=1}^p\prod_{\beta,\beta'=1}^q \frac{(1-x_iw_\alpha t^l)(1-y_{j'}z_{\beta'}t^l)}{(1+y_{j}w_{\alpha'}t^l)(1+x_{i'}z_{\beta}t^l)}.
\]
\end{thm}

\subsection*{Acknowledgements.} 
The authors would like to thank Sylvie Corteel for many helpful discussions during the course of this work. The authors were partially supported by NSF DMS-1600447 and NSF DMS-2054482.

\section{LLT polynomials, super ribbon functions, and the Littlewood quotient map}\label{sec:ribbon}

\indent This section provides necessary background information and establishes some notation for the rest of this paper.  First we venture into the world of tableaux on tuples of skew shapes, and we define coinversion LLT polynomials.  Then we venture into the world of ribbon tableaux, and we define super ribbon function.  Finally, we connect these two worlds via the Littlewood quotient map.

\subsection{Tuples of skew shapes and coinversion LLT polynomials}

\indent We begin by discussing semistandard Young tableaux on tuples of skew shapes, which we use to give one formulation of LLT polynomials.

\indent Fix a nonnegative integer $p$.  A \textbf{partition} with $p$ parts is a weakly decreasing sequence $\lambda = (\lambda_1 \geq \cdots \geq \lambda_p \geq 0)$ of $p$ nonnegative integers.  Note that here we consider our partitions to have a fixed number of parts but allow for the possibility of parts equalling zero; later we will also consider partitions with an infinite number of parts, only finitely many of which are nonzero.  We let the length $\ell(\lambda)$ be the number of nonzero parts of $\lambda$, and we let the size $|\lambda|$ be the sum $\lambda_1 + \ldots + \lambda_p$ of its parts.  We associate to $\lambda$ its Young (or Ferrers) diagram $D(\lambda) \subseteq \Z \times \Z$, given as
\[ D(\lambda) = \{(i,j) \mid 1 \leq i \leq \ell(\lambda), \; 1 \leq j \leq \lambda_i \}. \]
We refer to the elements of $D(\lambda)$ as cells.  We draw our diagrams in French notation in the first quadrant, so that the first row is on the bottom and the first column is on the left, such as below.
\[ \lambda = (4,2,1), \qquad D(\lambda) = 
\ytableausetup{aligntableaux=center}
\begin{ytableau} \\ & \\ & & \bullet & \end{ytableau} \]
The cell labelled above has coordinates (1,3).  The \textbf{content} of a cell $u = (i,j)$ in row $i$ and column $j$ of any Young diagram is $c(u) = j-i$.  In what follows, we will use $\lambda$ and $D(\lambda)$ interchangeably, when it will not cause confusion. 

If $\lambda$ and $\mu$ are partitions such that $D(\lambda) \supseteq D(\mu)$, then the \textbf{skew shape} $\lambda / \mu$ is the set of cells in $D(\lambda)$ that are not in $D(\mu)$.  We draw the diagram of $\lambda / \mu$ by coloring the cells in $D(\mu)$ gray, such as below.
\[ \lambda = (4,2,1), \mu = (2,1), \qquad D(\lambda / \mu) = 
\ytableausetup{aligntableaux=center}
\begin{ytableau} \\ *(lightgray) & \\ *(lightgray) & *(lightgray) & & \end{ytableau} \]
The size $|\lambda/\mu|$ of $\lambda / \mu$ is $|\lambda| - |\mu|$.

A \textbf{semistandard Young tableau} of shape $\lambda/\mu$ is a filling of each cell of $D(\lambda/\mu)$ with a positive integer such that the rows are weakly increasing and the columns are strictly increasing. We call the set of all such fillings $\SSYT(\lambda/\mu)$.  Let $\lm = (\lambda^{(1)}/\mu^{(1)}, \ldots, \lambda^{(k)}/\mu^{(k)})$ be a tuple of skew partitions.  We define its size to be
\[ |\lm| = |\lambda^{(1)}/\mu^{(1)}| + \ldots + |\lambda^{(k)}/\mu^{(k)}| \]
and we define a semistandard Young tableau of shape $\lm$ to be a semistandard Young tableau on each $\lambda^{(j)}/\mu^{(j)}$, that is,
\[ \SSYT(\lm) = \SSYT(\lambda^{(1)}/\mu^{(1)}) \times \cdots \times \SSYT(\lambda^{(k)}/\mu^{(k)}). \]
We can picture this as placing the Young diagrams diagonally ``on content lines" with the first shape in the South-West direction and the last shape in the North-East direction. 

\begin{example}
	Let $\lm = ((3,1), (2,2,2)/(1,1,1), (1), (2,1)/(2))$. The top row labels the contents of each line.
	\ytableausetup{nosmalltableaux}
	\ytableausetup{nobaseline}
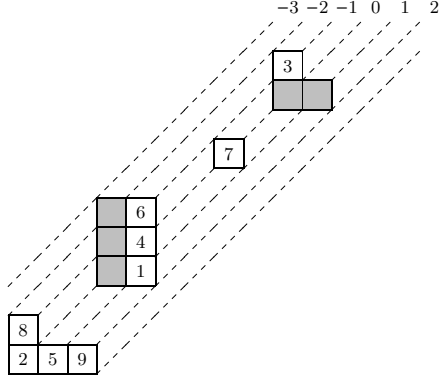
	\begin{center}
	\resizebox{.4\textwidth}{!}{
	\begin{ytableau}
		\none & \none & \none & \none & \none & \none & \none & \none & \none & \none[-3] & \none[-2] & \none[-1] & \none[0] & \none[1] & \none[2] \\			
		\none & \none & \none & \none & \none & \none & \none & \none &\none[\diag] &\none[\diag] &\none[\diag] &\none[\diag] & \none[\diag] & \none[\diag] \\
		\none & \none & \none & \none & \none & \none & \none &\none[\diag] &\none[\diag] & 3 &\none[\diag] &\none[\diag] & \none[\diag] & \none[\diag] \\
		\none & \none & \none & \none & \none &\none &\none[\diag] &\none[\diag] & \none[\diag] & *(lightgray) & *(lightgray) &\none[\diag] & \none[\diag] \\
		\none & \none & \none & \none &\none &\none[\diag] &\none[\diag] &\none[\diag] &\none[\diag] &\none[\diag] &\none[\diag] &\none[\diag] & \none \\
		\none & \none & \none &\none &\none[\diag] &\none[\diag] &\none[\diag] & 7 &\none[\diag] &\none[\diag] &\none[\diag] &\none \\
		\none & \none &\none &\none[\diag] &\none[\diag] &\none[\diag] &\none[\diag] &\none[\diag] &\none[\diag] &\none[\diag] &\none &\none \\
		\none & \none &\none[\diag] & *(lightgray) & 6 &\none[\diag] &\none[\diag] &\none[\diag] &\none[\diag] &\none &\none & \none \\
		\none & \none[\diag] & \none[\diag] & *(lightgray) & 4 &\none[\diag] &\none[\diag] &\none[\diag] &\none &\none & \none & \none \\
		\none[\diag] & \none[\diag] & \none[\diag] & *(lightgray) & 1 &\none[\diag] &\none[\diag] &\none &\none & \none & \none & \none \\
		\none[\diag] & \none[\diag] & \none[\diag] &\none[\diag] &\none[\diag] &\none[\diag] &\none &\none & \none & \none & \none & \none \\
		8 &\none[\diag] &\none[\diag] &\none[\diag] &\none[\diag] &\none &\none & \none & \none & \none & \none & \none \\
		2 & 5 & 9 &\none[\diag] &\none &\none & \none & \none & \none & \none & \none & \none \\
	\end{ytableau}	
	}
	\end{center}
	\label{exampleLLT}
\end{example}

We will use tuples of skew partitions to index our formulation of LLT polynomials.  Given a tuple $\lm$ of skew partitions, we say three cells $u, v, w \in \Z \times \Z$ form a \textbf{triple} if
\begin{enumerate}
\item $v \in \lm$;
\item they are situated as below
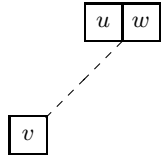
\begin{equation*} \label{triple}
\ytableausetup{nobaseline}
\begin{ytableau}
\none & \none & \none &\none & u & w  \\
\none & \none & \none & \none & \none[\diag]  \\
\none & \none & \none  & \none[\diag] \\
\none & \none & v \\
\end{ytableau}
\end{equation*}
namely with $v$ and $w$ on the same content line and $w$ in a later shape, and $u$ on a content line one smaller, in the same row as $w$; and 
\item if $u$ and $w$ are in row $r$ of $\lambda^{(j)}/\mu^{(j)}$, then $u$ and $w$ must be between the cells $(r, \mu^{(j)}_r)$ and $(r, \lambda^{(j)}_r+1)$, inclusive.
\end{enumerate}
It is important to note that while $v$ must be a cell in $\lm$, this is not true of $u$ and $w$.  If not in $\lm$, $u$ must be at the end of some row in $\bm{\mu}$, and $w$ must be the cell directly to the right of the end of some row in $\bm{\lambda}$.

\begin{definition}
Let $\lm$ be a tuple of skew partitions and let $T \in \SSYT(\lm)$. Let $a, b, c$ be the entries in the cells of a triple $u, v, w$, where we set $a = 0$ and $c=\infty$ if the respective cell is not in $\lm$. Given the triple of entries
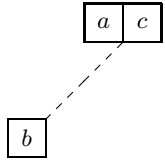
\ytableausetup{nobaseline}
\[
\begin{ytableau}
\none & \none & \none &\none & a & c  \\
\none & \none & \none & \none & \none[\diag]  \\
\none & \none & \none  & \none[\diag] \\
\none & \none & b \\
\end{ytableau}
\]
we say this is a \textbf{coinversion (inversion) triple} if $a \leq b \leq c$ ($b < a \leq c$ or $a \leq c < b$).
\label{inv-coinv-triple}
\end{definition}

\noindent There are 7 coinversion triples in Example \ref{exampleLLT} above: (0, 2, 4), (0, 2, 7), (3,4,$\infty$), (0,4,7), (4,5,$\infty$), (1,9,$\infty$), and (0,9,$\infty$).  

With these definitions in place, we are finally able to define the coinversion LLT polynomials. 
\begin{definition} \label{def:coinv-LLT}
The \textbf{coinversion LLT polynomial} associated to a tuple $\lm$ of skew partitions is the generating function
\[ \mc{L}_{\lm}(X; t) = \sum_{T \in \SSYT(\lm)} t^{\coinv(T)} x^T \]
where $\coinv(T)$ is the number of coinversion triples in $T$.
\end{definition}

\noindent See Appendix \ref{sec:other-LLT} for other formulations of LLT polynomials that appear in the literature.

\subsection{Ribbon tableaux and the Littlewood quotient map}

Next we discuss semistandard ribbon tableaux. We define a bijection, called the Littlewood quotient map, relating them and tuples of semistandard Young tableaux.

Fix a positive integer $k$.  A \textbf{$k$-ribbon} is a skew shape of size $k$ that is connected and does not contain any $2 \times 2$ square.  The head (tail) of a $k$-ribbon is the SE-most (NW-most) cell in its Young diagram.  A \textbf{horizontal (vertical) $k$-ribbon strip} of shape $\lambda/\mu$ is a tiling of $\lambda/\mu$ by $k$-ribbons such that the head (tail) of each ribbon is adjacent to the southern (western) boundary of the shape.  We let $\HRS_k(\lambda/\mu)$ ($\VRS_k(\lambda/\mu)$) denote the set of horizontal (vertical) $k$-ribbon strips of shape $\lambda/\mu$.   

\indent Throughout this paper, we will omit $k$ when it is clear from context.  For example, we will use ``ribbon" and ``$k$-ribbon" interchangeably.

\begin{definition} \label{def:ssrt}
A \textbf{semistandard $k$-ribbon tableau} of shape $\lambda/\mu$ is a tiling of $\lambda/\mu$ by $k$-ribbons and a labelling of the $k$-ribbons by positive integers such that, for all $i$,
\begin{enumerate}
\item removing all ribbons labelled $j$ for $j > i$ gives a valid skew shape $\lambda_{\leq i} / \mu$, and
\item the subtableau of ribbons labelled $i$ form a horizontal $k$-ribbon strip of shape $\lambda_{\leq i} / \lambda_{\leq i-1}$.
\end{enumerate}
We let $\SSRT_k(\lambda/\mu)$ denote the set of semistandard $k$-ribbon tableau of shape $\lambda/\mu$.
\end{definition}  

\indent Following the exposition of \cite[Section 3]{Pfannerer}, we now define the Littlewood $k$-quotient map.  This map was introduced in \cite{littlewood}; another (perhaps clearer) formulation, as well as a proof that the map is a bijection, is given in \cite{stanton1985schensted}.

We first define the \textbf{$k$-quotient map}, which is a function
\[ \begin{aligned}
\{ \text{skew partitions $\lambda/\mu$} \} \rightarrow \{ \text{$k$-tuples $\lm = (\lambda^{(0)}/\mu^{(0)},\ldots,\lambda^{(k-1)}/\mu^{(k-1)})$ of skew partitions} \}. 
\end{aligned} \]
This function can be defined graphically as follows.  Given a partition $\lambda$, we associate a finite sequence $(a_0,\ldots,a_{r-1})$ of East and South steps, called the \textbf{Maya diagram} of $\lambda$, by following the North-East boundary of $\lambda$ from North-West to South-East.

\begin{example}
The Maya diagram of (4,3,2,2,1) is ESESSESES.
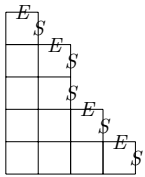
\[ \resizebox{2cm}{!}{
\begin{tikzpicture}[baseline=(current bounding box.center)]
\draw (0,0) grid (2,4); \draw (0,4) grid (1,5); \draw (2,0) grid (3,2); \draw (3,0) grid (4,1);
\node[scale=2] at (0.5,5) {E}; \node[scale=2] at (1,4.5) {S}; \node[scale=2] at (1.5,4) {E}; \node[scale=2] at (2,3.5) {S}; \node[scale=2] at (2,2.5) {S}; \node[scale=2] at (2.5,2) {E}; \node[scale=2] at (3,1.5) {S}; \node[scale=2] at (3.5,1) {E}; \node[scale=2] at (4,0.5) {S}; 
\end{tikzpicture} } \]
\end{example}

\begin{remark} \label{rmk:maya-extend}
 
Observe that postpending finitely many E's to a Maya diagram does not change the corresponding partition.  Thus we can take the length $r$ of the Maya diagram $(a_0,\ldots,a_{r-1})$ to be a multiple of $k$.
\end{remark}

\noindent Let $\lambda$ be a partition with Maya diagram $(a_0,\ldots,a_{r-1})$.  By the preceding remark, we may assume $s = r/k$ is an integer.  We define the $k$-quotient of $\lambda$ to be 
\[ \bm{\lambda} = (\lambda^{(0)},\ldots,\lambda^{(k-1)}) \]
where, for each $i$, $\lambda^{(i)}$ is the partition corresponding to the Maya diagram
\[ (a_i,a_{k+i},\ldots,a_{(s-1)k+i}). \]
We define the $k$-quotient of a skew partition $\lambda/\mu$ to be the $k$-tuple $\lm$ of skew partitions, where $\bm{\lambda}$ and $\bm{\mu}$ are the $k$-quotients of $\lambda$ and $\mu$ respectively.  Here we require $\lambda$ and $\mu$ to have the same number of parts, postpending parts equalling 0 to $\mu$ if necessary.

\begin{example} \label{ex:quotient}
The 3-quotient of (4,3,2,2,1) is ((1,1),(0,0),(2)).
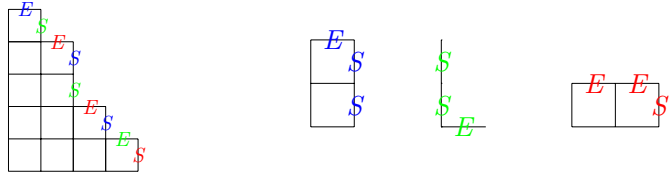
\[ \resizebox{2cm}{!}{
\begin{tikzpicture}[baseline=(current bounding box.center)]
\draw (0,0) grid (2,4); \draw (0,4) grid (1,5); \draw (2,0) grid (3,2); \draw (3,0) grid (4,1);
\node[scale=2,blue] at (0.5,5) {E}; \node[scale=2,green] at (1,4.5) {S}; \node[scale=2,red] at (1.5,4) {E}; \node[scale=2,blue] at (2,3.5) {S}; \node[scale=2,green] at (2,2.5) {S}; \node[scale=2,red] at (2.5,2) {E}; \node[scale=2,blue] at (3,1.5) {S}; \node[scale=2,green] at (3.5,1) {E}; \node[scale=2,red] at (4,0.5) {S}; 
\end{tikzpicture} }
\hspace{2cm}
\resizebox{5cm}{!}{
\begin{tikzpicture}[baseline=(current bounding box.center)]
\draw (0,0) grid (1,2);
\node[blue,scale=2] at (0.5,2) {E}; \node[blue,scale=2] at (1,1.5) {S}; \node[blue,scale=2] at (1,0.5) {S};
\draw (3+0,0) grid (3+0,2); \draw (3+0,0) grid (3+1,0);
\node[green,scale=2] at (3+0,1.5) {S}; \node[green,scale=2] at (3+0,0.5) {S}; \node[green,scale=2] at (3+0.5,0) {E};
\draw (6+0,0) grid (6+2,1);
\node[red,scale=2] at (6+0.5,1) {E}; \node[red,scale=2] at (6+1.5,1) {E}; \node[red,scale=2] at (6+2,0.5) {S};
\end{tikzpicture} }
\]
\end{example}

We are now ready to define the Littlewood $k$-quotient map.

\begin{definition} \label{def:littlewood-quotient-map}
Let $\lm$ be the $k$-quotient of $\lambda/\mu$.  The \textbf{Littlewood $k$-quotient map} is a bijection
\[ \SSRT_k(\lambda/\mu) \rightarrow  \SSYT(\lm) \]
defined as follows.  Fix $T \in \SSRT_k(\lambda/\mu)$.  For each $i$, we put an $i$ into each cell of the $k$-quotient of $\lambda_{\leq i}/\lambda_{\leq i-1}$ (which lies inside $\lm$).  In this fashion, we place positive integers into the cells of $\lm$, resulting in $\bm{T} = (T^{(0)},\ldots,T^{(k-1)}) \in \SSYT(\lm)$.
\end{definition}

\begin{example} In Example \ref{ex:quotient}, we found that the 3-quotient of $\lambda = (4,3,2,2,1)$ was $\bm{\lambda} = ((1,1),(0,0),(2))$.  One can compute that
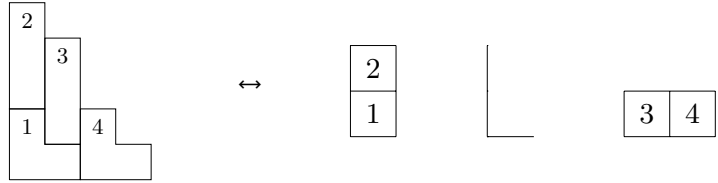
\[ 
\resizebox{2cm}{!}{
\begin{tikzpicture}[baseline=(current bounding box.center)]
\draw (0,0) -- (2,0) -- (2,1) -- (1,1) -- (1,2) -- (0,2) -- (0,0);
\draw (2,0) -- (4,0) -- (4,1) -- (3,1) -- (3,2) -- (2,2) -- (2,0);
\draw (0,2) rectangle (1,5);
\draw (1,1) rectangle (2,4);
\node[scale=2] at (0.5,1.5) {$1$}; \node[scale=2] at (0.5,4.5) {$2$};
\node[scale=2] at (1.5,3.5) {$3$}; \node[scale=2] at (2.5,1.5) {$4$};
\end{tikzpicture} }
\hspace{1cm} \leftrightarrow \hspace{1cm}
\resizebox{5cm}{!}{
\begin{tikzpicture}[baseline=(current bounding box.center)]
\draw (0,0) grid (1,2);
\draw (3+0,0) grid (3+0,2); \draw (3+0,0) grid (3+1,0);
\draw (6+0,0) grid (6+2,1);
\node[scale=2] at (0.5,0.5) {$1$}; \node[scale=2] at (0.5,1.5) {$2$};
\node[scale=2] at (6+0.5,0.5) {$3$}; \node[scale=2] at (6+1.5,0.5) {$4$};
\end{tikzpicture} }
\]
via the Littlewood 3-quotient map.  For example, when $i=3$, one can compute the $k$-quotient of $\lambda_{\leq i}/\lambda_{\leq i-1}$ as follows.
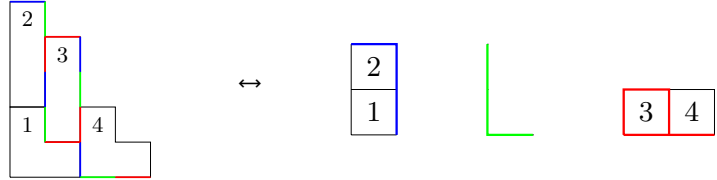
\[ 
\resizebox{2cm}{!}{
\begin{tikzpicture}[baseline=(current bounding box.center)]
\draw (0,0) -- (2,0) -- (2,1) -- (1,1) -- (1,2) -- (0,2) -- (0,0);
\draw (2,0) -- (4,0) -- (4,1) -- (3,1) -- (3,2) -- (2,2) -- (2,0);
\draw (0,2) rectangle (1,5);
\draw (1,1) rectangle (2,4);
\node[scale=2] at (0.5,1.5) {$1$}; \node[scale=2] at (0.5,4.5) {$2$};
\node[scale=2] at (1.5,3.5) {$3$}; \node[scale=2] at (2.5,1.5) {$4$};
 
\draw[blue, line width = 0.5mm] (0,5) -- (1,5); \draw[green, line width = 0.5mm] (1,5) -- (1,4); \draw[red, line width = 0.5mm] (1,4)--(2,4); \draw[blue, line width = 0.5mm] (2,4)--(2,3); \draw[green, line width = 0.5mm] (2,3)--(2,2); \draw[red, line width = 0.5mm] (2,2)--(2,1); \draw[blue, line width = 0.5mm] (2,1)--(2,0); \draw[green, line width = 0.5mm] (2,0)--(3,0); \draw[red, line width = 0.5mm] (3,0)--(4,0); \draw[red, line width = 0.5mm] (1,4)--(1,3); \draw[blue, line width = 0.5mm] (1,3)--(1,2); \draw[green, line width = 0.5mm] (1,2)--(1,1); \draw[red, line width = 0.5mm] (1,1)--(2,1);
\end{tikzpicture} }
\hspace{1cm} \leftrightarrow \hspace{1cm}
\resizebox{5cm}{!}{
\begin{tikzpicture}[baseline=(current bounding box.center)]
\draw (0,0) grid (1,2);
\draw[blue, line width = 0.5mm] (0,2)--(1,2)--(1,0);
\draw (3+0,0) grid (3+0,2); \draw (3+0,0) grid (3+1,0);
\draw[green, line width = 0.5mm] (3+0,2)--(3+0,0)--(3+1,0);
\draw (6+0,0) grid (6+2,1);
\draw[red, line width = 0.5mm] (6+0,0)--(6+1,0)--(6+1,1)--(6+0,1)--(6+0,0); \draw[red, line width = 0.5mm] (6+1,0)--(6+2,0);
\node[scale=2] at (0.5,0.5) {$1$}; \node[scale=2] at (0.5,1.5) {$2$};
\node[scale=2] at (6+0.5,0.5) {$3$}; \node[scale=2] at (6+1.5,0.5) {$4$};
\end{tikzpicture} }
\]
Here we have drawn the Maya diagrams of both $\lambda_{\leq 3}$ and $\lambda_{\leq 2}$.  From this, we see that that a box at coordinates $(1,1)$ is added in $T^{(2)}$, and we fill it with a 3.
\end{example}

\subsection{Extending the Littlewood quotient map to super tableaux}

\indent Finally, we extend the Littlewood quotient map to a bijection between super ribbon tableaux (Definition \ref{def:srt}) and semistandard super Young tableaux (Definition \ref{def:sssyt}).  This bijection plays a pivotal role in the rest of the paper, because it allows us to relate our partition functions (Definition \ref{def:latS}) to the super ribbon functions (Definition \ref{def:super-ribbon-function}) in Proposition \ref{prop-L-equals-G}.

\indent Throughout this subsection, let $\mc{A} = \{1 < 2 < \cdots \}$ and $\mc{A'} = \{ 1' < 2' < \cdots \}$.  Also fix a total order on $\mc{A} \cup \mc{A'}$ that is compatible with the natural orders on $\mc{A}$ and $\mc{A'}$.  

\begin{definition} \label{def:srt}
A \textbf{super $k$-ribbon tableau} of shape $\lambda/\mu$ is a tiling of $\lambda/\mu$ by $k$-ribbons and a labelling of the $k$-ribbons by the alphabet $\mc{A} \cup \mc{A'}$ such that 
\begin{enumerate}
\item for $i \in \mc{A} \cup \mc{A'}$, removing all ribbons labelled $j$ for $j > i$ gives a valid skew shape $\lambda_{\leq i} / \mu$;
\item for $i \in \mc{A}$, the subtableau of ribbons labelled $i$ form a horizontal $k$-ribbon strip; and
\item for $i' \in \mc{A'}$, the subtableau of ribbons labelled $i'$ form a vertical $k$-ribbon strip.
\end{enumerate}
We let $\SRT_k(\lambda/\mu)$ denote the set of super $k$-ribbon tableau of shape $\lambda/\mu$.
\end{definition}

\noindent Note that a SRT in the alphabet $\mc{A}$ of shape $\lambda/\mu$ is the same as a SSRT of shape $\lambda/\mu$.  Moreover, there is a bijection between SRT in the alphabet $\mc{A'}$ of shape $\lambda/\mu$ and SSRT of shape $\lambda'/\mu'$, given by conjugation (and unpriming the labels).

The height $h(R)$ of a ribbon $R$ is the number of rows it contains.  The \textbf{spin} of a super ribbon tableau $T$ is
\[ \spin(T) = \sum_R (h(R)-1) \]
where the sum is taken over all ribbons $R$ in $T$.

\begin{definition} \cite[Definition 44]{Lam} \label{def:super-ribbon-function}
The \textbf{super $k$-ribbon function} associated to a skew partition $\lambda/\mu$ is the generating function
\[ \mc{G}^{(k)}_{\lambda/\mu}(X;Y;t) = \sum_{T \in \SRT_k(\lambda/\mu)} t^{\spin(T)} x^{\weight(T)} y^{\weight'(T)}. \]
\end{definition}

\begin{example}
We use the ordering $1 < 2 < \ldots < 1' < 2' < \ldots$ on $\mc{A} \cup \mc{A'}$.  Let $k = 3$ and $\lambda/\mu = (8,7,6,6,6,4,1)/(2)$.  The super ribbon tableau
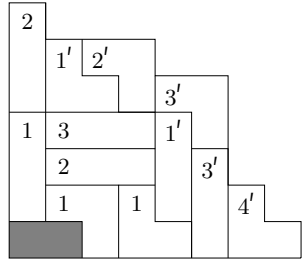
\[ \resizebox{4cm}{!}{
\begin{tikzpicture}[baseline=(current bounding box.center)]
\fill[gray] (0,0) rectangle (2,1);
\draw[thin] (8,0) -- (0,0) -- (0,7) -- (1,7) -- (1,6) -- (4,6) -- (4,5) -- (6,5) -- (6,2) -- (7,2) -- (7,1) -- (8,1) -- (8,0);
\draw[thin] (2,0) -- (2,1) -- (0,1);
\draw[thin] (1,1) -- (1,7);
\draw[thin] (0,4) -- (5,4) -- (5,0);
\draw[thin] (1,2) -- (3,2) -- (3,0);
\draw[thin] (1,3) -- (4,3) -- (4,1) -- (5,1) -- (5,0);
\draw[thin] (3,2) -- (4,2);
\draw[thin] (2,6) -- (2,5) -- (3,5) -- (3,4) -- (4,4) -- (4,3);
\draw[thin] (5,3) -- (6,3) -- (6,0);
\draw[thin] (4,5) -- (4,4);
\node[scale=2] at (0.5,3.5) {$1$}; \node[scale=2] at (0.5,6.5) {$2$};
\node[scale=2] at (1.5,1.5) {$1$}; \node[scale=2] at (1.5,2.5) {$2$}; \node[scale=2] at (1.5,3.5) {$3$}; \node[scale=2] at (1.5,5.5) {$1'$};
\node[scale=2] at (2.5,5.5) {$2'$}; \node[scale=2] at (3.5,1.5) {$1$};
\node[scale = 2] at (4.5,3.5) {$1'$}; \node[scale = 2] at (4.5,4.5) {$3'$};
\node[scale = 2] at (5.5,2.5) {$3'$}; \node[scale = 2] at (6.5,1.5) {$4'$};
\end{tikzpicture}
} \]
has spin 14 and contributes 
\[ t^{14} x_1^3 x_2^2 x_3^1 y_1^2 y_2^1 y_3^2 y_4^1 \]
to $\mc{G}^{(k)}_{\lambda/\mu}(X;Y;t)$.
\end{example}

\begin{definition} \label{def:sssyt}
A \textbf{semistandard super Young tableau} of shape $\lambda/\mu$ is a filling of each cell of $D(\lambda)$ with an element of $\mc{A} \cup \mc{A'}$ such that
\begin{enumerate}
    \item the rows and the columns are weakly increasing, 
    \item the entries in $\mc{A}$ are strictly increasing along columns, and
    \item the entries in $\mc{A'}$ are strictly increasing along rows.
\end{enumerate}
We let $\SSSYT(\lambda/\mu)$ denote the set of semistandard super Young tableaux of shape $\lambda/\mu$.  Given a tuple $\lm = (\lambda^{(1)}/\mu^{(1)}, \ldots, \lambda^{(k)}/\mu^{(k)})$ of skew partitions, a semistandard super Young tableau of shape $\lm$ is a semistandard super Young tableau on each $\lambda^{(j)}/\mu^{(j)}$, that is,
\[ \SSSYT(\lm) = \SSSYT(\lambda^{(1)}/\mu^{(1)}) \times \cdots \times \SSSYT(\lambda^{(k)}/\mu^{(k)}). \]
\end{definition}

\noindent Note that a SSSYT in the alphabet $\mc{A}$ of shape $\lambda/\mu$ is the same as a SSYT of shape $\lambda/\mu$.  Moreover, there is a bijection between SSSYT in the alphabet $\mc{A'}$ of shape $\lambda/\mu$ and SSYT of shape $\lambda'/\mu'$, given by conjugation (and unpriming the labels). 

\indent We are now ready to extend the Littlewood $k$-quotient map.

\begin{definition}
The \textbf{(extended) Littlewood $k$-quotient map} is a bijection
\[ \SRT_k(\lambda/\mu) \rightarrow \SSSYT(\lm) \]
where $\lm$ is the $k$-quotient of $\lambda/\mu$.  We simply take Definition \ref{def:littlewood-quotient-map} and extend the set of labels: for each $i \in \mc{A} \cup \mc{A'}$, we put an $i$ into each cell of the $k$-quotient of $\lambda_{\leq i}/ \lambda_{\leq i-1}$.
\end{definition}

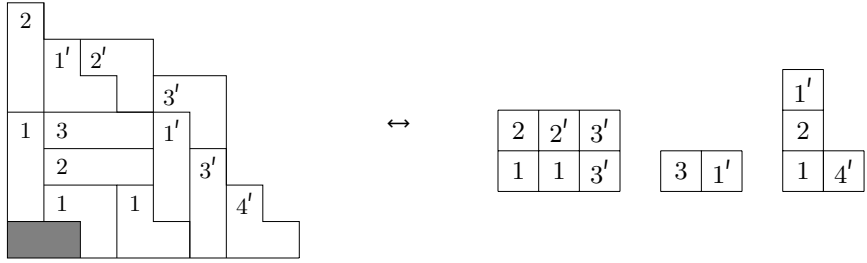
\begin{example} \label{ex:extended-lqm}
\[ \resizebox{4cm}{!}{
\begin{tikzpicture}[baseline=(current bounding box.center)]
\fill[gray] (0,0) rectangle (2,1);
\draw[thin] (8,0) -- (0,0) -- (0,7) -- (1,7) -- (1,6) -- (4,6) -- (4,5) -- (6,5) -- (6,2) -- (7,2) -- (7,1) -- (8,1) -- (8,0);
\draw[thin] (2,0) -- (2,1) -- (0,1);
\draw[thin] (1,1) -- (1,7);
\draw[thin] (0,4) -- (5,4) -- (5,0);
\draw[thin] (1,2) -- (3,2) -- (3,0);
\draw[thin] (1,3) -- (4,3) -- (4,1) -- (5,1) -- (5,0);
\draw[thin] (3,2) -- (4,2);
\draw[thin] (2,6) -- (2,5) -- (3,5) -- (3,4) -- (4,4) -- (4,3);
\draw[thin] (5,3) -- (6,3) -- (6,0);
\draw[thin] (4,5) -- (4,4);
\node[scale=2] at (0.5,3.5) {$1$}; \node[scale=2] at (0.5,6.5) {$2$};
\node[scale=2] at (1.5,1.5) {$1$}; \node[scale=2] at (1.5,2.5) {$2$}; \node[scale=2] at (1.5,3.5) {$3$}; \node[scale=2] at (1.5,5.5) {$1'$};
\node[scale=2] at (2.5,5.5) {$2'$}; \node[scale=2] at (3.5,1.5) {$1$};
\node[scale = 2] at (4.5,3.5) {$1'$}; \node[scale = 2] at (4.5,4.5) {$3'$};
\node[scale = 2] at (5.5,2.5) {$3'$}; \node[scale = 2] at (6.5,1.5) {$4'$};
\end{tikzpicture}
}
\hspace{1cm} \leftrightarrow \hspace{1cm}
\resizebox{5cm}{!}{
\begin{tikzpicture}[baseline=(current bounding box.center)]
\draw (0,0) grid (3,2);
\draw (4+0,0) grid (4+2,1);
\draw (7+0,0) grid (7+1,3); \draw (7+1,0) grid (7+2,1);
\node[scale=2] at (0.5,0.5) {$1$}; \node[scale=2] at (0.5,1.5) {$2$}; \node[scale=2] at (1.5,0.5) {$1$}; \node[scale=2] at (1.5,1.5) {$2'$}; \node[scale=2] at (2.5,0.5) {$3'$}; \node[scale=2] at (2.5,1.5) {$3'$};
\node[scale=2] at (4+0.5,0.5) {$3$}; \node[scale=2] at (4+1.5,0.5) {$1'$};
\node[scale=2] at (7+0.5,0.5) {$1$}; \node[scale=2] at (7+1.5,0.5) {$4'$}; \node[scale=2] at (7+0.5,1.5) {$2$}; \node[scale=2] at (7+0.5,2.5) {$1'$};
\end{tikzpicture} }
\]
\end{example}

The following facts can be useful in computing the Littlewood $k$-quotient map in examples.
\begin{lem} \label{lem:lqm-properties}
Suppose $T \leftrightarrow \bm{T}$ via the Littlewood $k$-quotient map.
\begin{enumerate}
    \item A ribbon in $T$ labelled $i$ corresponds to a cell labelled $i$ in $\bm{T}$, so the number of ribbons in $T$ labelled $i$ equals the number of cells labelled $i$ in $\bm{T}$.
    \item Two ribbons $R,R'$ in $T$ whose tails $u,u'$ have the same content modulo $k$ correspond to two cells $v,v'$ in the same shape in $\bm{T}$.  Moreover, in this case, 
    \[ \frac{c(u)-c(u')}{k} = c(v)-c(v'). \]
\end{enumerate}
\end{lem}

Both the non-extended and the extended Littlewood $k$-quotient maps have these properties.  The proof for the non-extended map follows from the proof for the extended map, which is given in Appendix \ref{sec:lqm-properties-proof}.

\pagebreak

\section{Vertex models, Yang-Baxter equations, and partition functions} \label{sec:VertexModel}

\indent In this section, we introduce several types of vertex models. We show that these models are integrable in the sense that they satisfy several Yang-Baxter equations.  Then we use these vertex models to construct certain families of partition functions.  Studying these families of partition functions, in particular the family in Definition \ref{def:latS}, will be the main focus of the rest of the paper.

\indent We begin with some notation.  For a vector $\I = (I_1,\ldots,I_k) \in \mathbb{R}^k$, we define \[ |\I| = \sum_{m=1}^k I_i. \]
For vectors $\I = (I_1,\ldots,I_k), \J = (J_1,\ldots,J_k) \in \mathbb{R}^k$, we define
\[ \varphi(\I,\J) = \sum_{1 \leq i < j \leq k} I_iJ_j. \] 
For variables $x$ and $t$ and an integer $n \geq 0$, we define the $t$-Pochhammer symbol
\[ (x;t)_n = \prod_{m=0}^{n-1} (1-xt^m). \]
\noindent We are now ready to define our vertices algebraically.  We will define five vertices - the $L$-matrix, the $L'$-matrix, the $R$-matrix, the $R'$-matrix, and the $R''$-matrix.  The $L$ and $L'$ matrices are families of functions $(\{0,1\}^k)^4 \rightarrow \CC[x,t]$, one for each integer $k \geq 0$; the $R$, $R'$, and $R''$ matrices are families of functions $(\{0,1\}^k)^4 \rightarrow \CC(x,y,t)$, one for each integer $k \geq 0$.  In other words, each vertex associates a \textbf{weight} (either a polynomial in $x,t$ or a rational function in $x,y,t$) to every 4-tuple of vectors in $\{0,1\}^k$ for each integer $k \geq 0$.
\begin{center}
\resizebox{.85\textwidth}{!}{
\begin{tabular}{||c|c||}
\hline
Type of vertex & Algebraic definition \\ \hline \hline
$L$ & $\displaystyle L_x^{(k)}(\I,\J,\K,\L) = \textbf{1}_{\I+\J=\K+\L} \prod_{i=1}^k \textbf{1}_{I_i+J_i \neq 2} \cdot x^{|\L|} t^{\varphi(\L,\I+\J)}$ \\ \hline
$L'$ & $\displaystyle L_x^{'(k)}(\I,\J,\K,\L) = \textbf{1}_{\I+\J=\K+\L} \prod_{i=1}^k \textbf{1}_{K_i \geq J_i} \cdot x^{|\L|} t^{\varphi(\L,\K-\J)}$ \\ \hline
$R$ & $\displaystyle R_{y/x}^{(k)}(\I,\J,\K,\L) = \textbf{1}_{\I+\J=\K+\L} \prod_{i=1}^k \textbf{1}_{J_i \geq K_i} \cdot (-1)^{|\J|-|\K|} (y/x)^{|\J|} (x/y;t)_{|\J|-|\K|} t^{\varphi(\J,\K-\J)}$ \\ \hline
$R'$ &  $\displaystyle R_{y/x}^{'(k)}(\I,\J,\K,\L) = \textbf{1}_{\I+\J=\K+\L} \prod_{i=1}^k \textbf{1}_{I_i + J_i \neq 2} \cdot (x/y)^{|\L|} (-x/y;t)_{|\K|+|\L|}^{-1} t^{\varphi(\L,\K+\L)}$ \\ \hline
$R''$ & $\displaystyle R_{x/y}^{''(k)}(\I,\J,\K,\L) = \textbf{1}_{\I+\J=\K+\L} \prod_{i=1}^k \textbf{1}_{K_i \geq J_i} \cdot (x/y)^{|\L|} (x/y;t)_{|\K|-|\J|} t^{\varphi(\L,\K-\J)}$ \\
\hline
\end{tabular}
}
\end{center}

\indent However, it is often useful to think of a vertex graphically.  We can draw a vertex as a face with four incident edges, each labelled by an element of $\{0,1\}^k$.  A face takes one of two forms, 
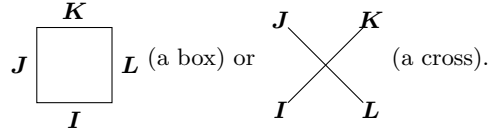
\[ \begin{tikzpicture}[baseline=(current bounding box.center)]\draw[thin] (0,0) rectangle (1,1); \node[left] at (0,0.5) {$\J$}; \node[right] at (1,0.5) {$\L$};\node[above] at (0.5,1) {$\K$};\node[below] at (0.5,0) {$\I$}; \end{tikzpicture} \text{(a box) or } \begin{tikzpicture}[baseline=(current bounding box.center)] \draw[thin] (0,1)--(1,0);\draw[thin] (0,0)--(1,1); \node at (-0.1,-0.1) {$\I$}; \node at (-0.1,1.1) {$\J$}; \node at (1.1,1.1) {$\K$}; \node at (1.1,-0.1) {$\L$};\end{tikzpicture} \text{(a cross).} \]
The edge labels describe colored paths moving through the face SW-to-NE (for a box) or left-to-right (for a cross).  If an edge has the label $\I = (I_1,\ldots,I_k) \in \{0,1\}^k$, then for each $i \in [k]$, a path of color $i$ is incident at the edge if and only if $I_i = 1$.  For example, with $k=2$ (letting blue be color 1 and red be color 2), the path configuration associated to the edge labels
\[ \I = (0,1), \J = (1,0), \K = (0,1), \L = (1,0) \]
is 
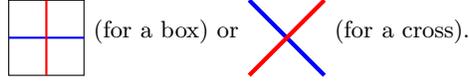
\[
\begin{tikzpicture}[baseline=(current bounding box.center)] \draw[thin] (0,0) rectangle (1,1); \draw[blue,thick] (0,0.5)--(1,0.5); \draw[red, thick] (0.5,0)--(0.5,1); \end{tikzpicture} 
\text{ (for a box) or }
\begin{tikzpicture}[baseline=(current bounding box.center)] \draw[blue, ultra thick] (0,1)--(1,0); \draw[red, ultra thick] (0,0)--(1,1); \end{tikzpicture}
\text{ (for a cross).}
\]
The factor of $\textbf{1}_{\I+\J=\K+\L}$ that appears in the algebraic definitions of all five vertices imposes a \textbf{path conservation} restriction: in order for a vertex to have a nonzero weight, the paths entering the vertex and the paths exiting the vertex must be the same.  To define the vertex weights graphically, we start by defining the weights in the case $k=1$.
\begin{center}
\resizebox{0.85\textwidth}{!}{
\begin{tabular}{||c|c||}
\hline
Type of vertex & One-color definition \\ \hline\hline
$L$ & \resizebox{0.8\textwidth}{!}{
\begin{tabular}{cccccc}
   \begin{tikzpicture}[baseline=(current bounding box.center)]\draw[thin] (0,0) rectangle (1,1); \node[left] at (0,0.5) {j}; \node[right] at (1,0.5) {l};\node[above] at (0.5,1) {k};\node[below] at (0.5,0) {i};\node at (0.5,0.5) {$x$}; \end{tikzpicture}: & \begin{tikzpicture}[baseline=(current bounding box.center)]\draw[thin] (0,0) rectangle (1,1);\end{tikzpicture} & \begin{tikzpicture}[baseline=(current bounding box.center)]\draw[thin] (0,0) rectangle (1,1);\draw[red, thick] (0.5,0)--(0.5,0.5)--(1,0.5);\end{tikzpicture} & \begin{tikzpicture}[baseline=(current bounding box.center)] \draw[thin] (0,0) rectangle (1,1); \draw[red, thick] (0,0.5)--(1,0.5); \end{tikzpicture} & \begin{tikzpicture}[baseline=(current bounding box.center)] \draw[thin] (0,0) rectangle (1,1); \draw[red, thick] (0.5,0)--(0.5,1); \end{tikzpicture} & \begin{tikzpicture}[baseline=(current bounding box.center)] \draw[thin] (0,0) rectangle (1,1); \draw[red, thick] (0,0.5)--(0.5,0.5)--(0.5,1); \end{tikzpicture} \\
$L_x^{(1)}(i,j,k,l)$: &  $1$ & $x$ & $x$ & $1$ & $1$ \\
\end{tabular}
} \\ \hline
$L'$ & \resizebox{0.8\textwidth}{!}{
\begin{tabular}{cccccc}
   \begin{tikzpicture}[baseline=(current bounding box.center)]\draw[thin,fill=violet] (0,0) rectangle (1,1); \node[left] at (0,0.5) {j}; \node[right] at (1,0.5) {l};\node[above] at (0.5,1) {k};\node[below] at (0.5,0) {i};\node at (0.5,0.5) {$x$};\end{tikzpicture}: & \begin{tikzpicture}[baseline=(current bounding box.center)] \draw[thin,fill=violet] (0,0) rectangle (1,1); \draw[red, thick] (0.5,0)--(0.5,1); \end{tikzpicture} & \begin{tikzpicture}[baseline=(current bounding box.center)] \draw[thin,fill=violet] (0,0) rectangle (1,1); \draw[red, thick] (0,0.5)--(0.5,0.5)--(0.5,1); \end{tikzpicture} & \begin{tikzpicture}[baseline=(current bounding box.center)] \draw[thin,fill=violet] (0,0) rectangle (1,1); \draw[red, thick] (0,0.5)--(1,0.5); \draw[red, thick] (0.5,0)--(0.5,1); \end{tikzpicture} & \begin{tikzpicture}[baseline=(current bounding box.center)]\draw[thin,fill=violet] (0,0) rectangle (1,1); \end{tikzpicture} & \begin{tikzpicture}[baseline=(current bounding box.center)]\draw[thin,fill=violet] (0,0) rectangle (1,1);\draw[red, thick] (0.5,0)--(0.5,0.5)--(1,0.5);\end{tikzpicture} \\
$L_x^{'(1)}(i,j,k,l)$: &  $1$ & $1$ & $x$ & $1$ & $x$ \\
\end{tabular}
} \\ \hline
$R$ & \resizebox{0.8\textwidth}{!}{
\begin{tabular}{cccccc}
\begin{tikzpicture}[baseline=(current bounding box.center)] \draw[thin] (0,1)--(1,0);\draw[thin] (0,0)--(1,1); \node at (-0.1,-0.1) {$i$}; \node at (-0.5,-0.1) {$y$}; \node at (-0.1,1.1) {$j$}; \node at (-0.5,1.1) {$x$}; \node at (1.1,1.1) {$k$}; \node at (1.1,-0.1) {$l$};\end{tikzpicture}: & \begin{tikzpicture}[baseline=(current bounding box.center)] \draw[red, ultra thick] (0,1)--(1,0);\draw[thin] (0,0)--(1,1); \end{tikzpicture} & \begin{tikzpicture}[baseline=(current bounding box.center)] \draw[red, ultra thick] (0,1)--(0.5,0.5)--(1,1);\draw[thin] (0,0)--(0.5,0.5)--(1,0); \end{tikzpicture} & \begin{tikzpicture}[baseline=(current bounding box.center)] \draw[thin] (0,1)--(0.5,0.5)--(1,1);\draw[red, ultra thick] (0,0)--(0.5,0.5)--(1,0); \end{tikzpicture} & \begin{tikzpicture}[baseline=(current bounding box.center)] \draw[red, ultra thick] (0,1)--(1,0);\draw[red, ultra thick] (0,0)--(1,1); \end{tikzpicture} & \begin{tikzpicture}[baseline=(current bounding box.center)] \draw[thin] (0,1)--(1,0);\draw[thin] (0,0)--(1,1); \end{tikzpicture} \\
$R_{y/x}^{(1)}(i,j,k,l)$:&  $1-y/x$ & $y/x$ & $1$ & $y/x$ & $1$ 
\end{tabular}
} \\ \hline
$R'$ & \resizebox{0.8\textwidth}{!}{
\begin{tabular}{cccccc}
\begin{tikzpicture}[baseline=(current bounding box.center)] \fill[yellow] (0,0) rectangle (1,1); \draw[thin] (0,1)--(1,0);\draw[thin] (0,0)--(1,1); \node at (-0.1,-0.1) {$i$}; \node at (-0.5,-0.1) {$y$}; \node at (-0.1,1.1) {$j$}; \node at (-0.5,1.1) {$x$}; \node at (1.1,1.1) {$k$}; \node at (1.1,-0.1) {$l$};\end{tikzpicture}: & \begin{tikzpicture}[baseline=(current bounding box.center)] \fill[yellow] (0,0) rectangle (1,1); \draw[red, ultra thick] (0,1)--(1,0);\draw[thin] (0,0)--(1,1); \end{tikzpicture} & \begin{tikzpicture}[baseline=(current bounding box.center)] \fill[yellow] (0,0) rectangle (1,1); \draw[red, ultra thick] (0,1)--(0.5,0.5)--(1,1);\draw[thin] (0,0)--(0.5,0.5)--(1,0); \end{tikzpicture} & \begin{tikzpicture}[baseline=(current bounding box.center)] \fill[yellow] (0,0) rectangle (1,1); \draw[thin] (0,1)--(0.5,0.5)--(1,1);\draw[red, ultra thick] (0,0)--(0.5,0.5)--(1,0); \end{tikzpicture} & \begin{tikzpicture}[baseline=(current bounding box.center)] \fill[yellow] (0,0) rectangle (1,1); \draw[thin] (0,1)--(1,0);\draw[red, ultra thick] (0,0)--(1,1); \end{tikzpicture} & \begin{tikzpicture}[baseline=(current bounding box.center)] \fill[yellow] (0,0) rectangle (1,1); \draw[thin] (0,1)--(1,0);\draw[thin] (0,0)--(1,1); \end{tikzpicture} \\
$R_{y/x}^{'(1)}(i,j,k,l)$:&  $\frac{1}{1+y/x}$ & $\frac{y/x}{1+y/x}$ & $\frac{1}{1+y/x}$ & $\frac{y/x}{1+y/x}$ & $1$ 
\end{tabular}
} \\ \hline
$R''$ & \resizebox{0.8\textwidth}{!}{
\begin{tabular}{cccccc}
\begin{tikzpicture}[baseline=(current bounding box.center)] \fill[orange] (0,0) rectangle (1,1); \draw[thin] (0,1)--(1,0);\draw[thin] (0,0)--(1,1); \node at (-0.1,-0.1) {$i$}; \node at (-0.5,-0.1) {$y$}; \node at (-0.1,1.1) {$j$}; \node at (-0.5,1.1) {$x$}; \node at (1.1,1.1) {$k$}; \node at (1.1,-0.1) {$l$};\end{tikzpicture}: & \begin{tikzpicture}[baseline=(current bounding box.center)] \fill[orange] (0,0) rectangle (1,1); \draw[red, ultra thick] (0,0)--(1,1);\draw[thin] (0,1)--(1,0); \end{tikzpicture} & \begin{tikzpicture}[baseline=(current bounding box.center)] \fill[orange] (0,0) rectangle (1,1); \draw[red, ultra thick] (0,1)--(0.5,0.5)--(1,1);\draw[thin] (0,0)--(0.5,0.5)--(1,0); \end{tikzpicture} & \begin{tikzpicture}[baseline=(current bounding box.center)] \fill[orange] (0,0) rectangle (1,1); \draw[thin] (0,1)--(0.5,0.5)--(1,1);\draw[red, ultra thick] (0,0)--(0.5,0.5)--(1,0); \end{tikzpicture} & \begin{tikzpicture}[baseline=(current bounding box.center)] \fill[orange] (0,0) rectangle (1,1); \draw[red, ultra thick] (0,1)--(1,0);\draw[red, ultra thick] (0,0)--(1,1); \end{tikzpicture} & \begin{tikzpicture}[baseline=(current bounding box.center)] \fill[orange] (0,0) rectangle (1,1); \draw[thin] (0,1)--(1,0);\draw[thin] (0,0)--(1,1); \end{tikzpicture} \\
$R_{x/y}^{''(1)}(i,j,k,l)$:&  $1-x/y$ & $1$ & $x/y$ & $x/y$ & $1$ 
\end{tabular}
}
\\ \hline
\end{tabular}
}
\end{center}

The $k$-color weights are then defined in terms of the one-color weights.  
\begin{center}
\resizebox{0.85\textwidth}{!}{
\begin{tabular}{||c|c||}
\hline
Type of vertex & $k$-color definition \\ \hline \hline
$L$ & $\displaystyle L_x^{(k)}(\I,\J,\K,\L) =  \prod_{i=1}^k L_{xt^{\delta_i}}^{(1)}(I_i,J_i,K_i,L_i)$ where $\delta_i = \text{\# colors greater than $i$ that are present}$ \\ \hline
$L'$ & $L_x^{'(k)}(\I,\J,\K,\L) = \prod_{i=1}^k L_{xt^{\delta'_i}}^{'(1)}(I_i,J_i,K_i,L_i)$ where $\delta'_i = \text{\# colors greater than $i$ of the form}$ \resizebox{0.04\textwidth}{!}{\begin{tikzpicture}[baseline=(current bounding box.center)] \draw[thin,fill=violet] (0,0) rectangle (1,1); \draw[red, thick] (0.5,0)--(0.5,1); \end{tikzpicture}} \\ \hline
$R$ & $\displaystyle R_{y/x}^{(k)}(\I,\J,\K,\L) =  \prod_{i=1}^k R_{y/(xt^{\epsilon_i})}^{(1)}(I_i,J_i,K_i,L_i)$ where $\epsilon_i = \text{\# colors greater than $i$ of the form}$ \resizebox{0.04\textwidth}{!}{\begin{tikzpicture}[baseline=(current bounding box.center)] \draw[red, ultra thick] (0,1)--(1,0);\draw[thin] (0,0)--(1,1); \end{tikzpicture}} \\ \hline
$R'$ &  $\displaystyle R_{y/x}^{'(k)}(\I,\J,\K,\L) = \prod_{i=1}^k R_{y/(xt^{\epsilon'_i})}^{'(1)}(I_i,J_i,K_i,L_i)$ where $\epsilon'_i = \text{\# colors greater than $i$ that are present}$
\\ \hline
$R''$ & $\displaystyle R_{x/y}^{''(k)}(\I,\J,\K,\L) = \prod_{i=1}^k R_{t^{\epsilon''_i}x/y}^{''(1)}(I_i,J_i,K_i,L_i)$ where $\epsilon''_i = \text{\# colors greater than $i$ of the form}$ \resizebox{0.04\textwidth}{!}{\begin{tikzpicture}[baseline=(current bounding box.center)] \fill[orange] (0,0) rectangle (1,1); \draw[red, ultra thick] (0,0)--(1,1);\draw[thin] (0,1)--(1,0); \end{tikzpicture}} \\ \hline
\end{tabular}
}
\end{center}

We leave it as an exercise for the reader to check that the algebraic and graphical definitions are equivalent.

We remark that the $L$-matrix and the $R$-matrix appeared in \cite{CGKM}.  Moreover, by the following lemma, the weights of all five vertices can be realized as degenerations of vertex weights that appear in \cite{ABW}.

\begin{lem} \label{lem:degen-abw}
We adopt the notation of \cite{ABW}, except we use $t$ in place of $q$.  In particular, we let $W_z(\A,\B;\C,\D | r,s)$ be the vertex weights from \cite[Definition 5.1.1]{ABW} with $t$ in place of $q$.  Then
\[ \begin{aligned}
\resizebox{1.8cm}{!}{ \begin{tikzpicture}[baseline=(current bounding box.center)] 
\draw[step=1.0,black,thin] (0,0) grid (1,1); 
\node at (0.5,0.5) {$x$}; \node[below] at (0.5,0) {$\A$}; \node[left] at (0,0.5) {$\B$}; \node[above] at (0.5,1) {$\C$}; \node[right] at (1,0.5) {$\D$};
\end{tikzpicture} } &= \lim_{\alpha \rightarrow 0} (-\alpha)^d \lim_{\beta \rightarrow 0} \beta^{-2d} W_{x/\alpha}(\A,\B;\C,\D | (x/\alpha)^{1/2},\beta), \\
\resizebox{1.8cm}{!}{ \begin{tikzpicture}[baseline=(current bounding box.center)] 
\draw[step=1.0,black,thin,fill=violet] (0,0) rectangle (1,1); 
\node at (0.5,0.5) {$y$}; \node[below] at (0.5,0) {$\A$}; \node[left] at (0,0.5) {$\B$}; \node[above] at (0.5,1) {$\C$}; \node[right] at (1,0.5) {$\D$};
\end{tikzpicture} } &= \lim_{Y \rightarrow 0} Y^{-d} \lim_{S \rightarrow 0} W_1(\A,\B;\C,\D | Sy^{-1/2},SY^{1/2}), \\
 \resizebox{1.8cm}{!}{
\begin{tikzpicture}[baseline=(current bounding box.center)] 
\draw[] (-1,0.5) -- (0,1.5); \draw[] (-1,1.5) -- (0,0.5);
\node at (-1.5,0.5) {$y$}; \node at (-1.5,1.5) {$x$};
\node[left] at (-1,0.5) {$\A$}; \node[left] at (-1,1.5) {$\B$}; \node[right] at (0,1.5) {$\C$}; \node[right] at (0,0.5) {$\D$}; 
\end{tikzpicture} } &= \lim_{\alpha \rightarrow 0} W_{x/y}(\A,\B;\C,\D | (x/\alpha)^{1/2},(y/\alpha)^{1/2}), \\
\resizebox{1.8cm}{!}{
\begin{tikzpicture}[baseline=(current bounding box.center)] 
\fill[yellow] (-1,0.5) rectangle (0,1.5); \draw[] (-1,0.5) -- (0,1.5); \draw[] (-1,1.5) -- (0,0.5);
\node at (-1.5,0.5) {$y$}; \node at (-1.5,1.5) {$x$};
\node[left] at (-1,0.5) {$\A$}; \node[left] at (-1,1.5) {$\B$}; \node[right] at (0,1.5) {$\C$}; \node[right] at (0,0.5) {$\D$}; 
\end{tikzpicture} } &= \lim_{\alpha \rightarrow 0} W_{x/\alpha}(\A,\B;\C,\D | (x/\alpha)^{1/2},(-y/\alpha)^{-1/2}), \\
\resizebox{1.8cm}{!}{
\begin{tikzpicture}[baseline=(current bounding box.center)] 
\fill[orange] (-1,0.5) rectangle (0,1.5); \draw[] (-1,0.5) -- (0,1.5); \draw[] (-1,1.5) -- (0,0.5);
\node at (-1.5,0.5) {$y$}; \node at (-1.5,1.5) {$x$};
\node[left] at (-1,0.5) {$\A$}; \node[left] at (-1,1.5) {$\B$}; \node[right] at (0,1.5) {$\C$}; \node[right] at (0,0.5) {$\D$}; 
\end{tikzpicture} } &= \lim_{S \rightarrow 0} W_1(\A,\B;\C,\D | Sx^{-1/2},Sy^{-1/2}).
\end{aligned} \]
\end{lem}

\begin{proof}
For the $L$-matrix, this follows from \cite[Cor. 8.3.6]{ABW}.  For the $L'$-matrix, this follows from \cite[Cor. 8.3.4]{ABW}.  For the $R$-matrix, this follows from \cite[Cor. 8.3.1]{ABW} (after substituting $s = (y/\alpha)^{1/2}$).  For the $R'$-matrix, this follows from \cite[Cor. 8.3.8]{ABW} (after substituting $s = (-y/\alpha)^{-1/2}$).  For the $R''$-matrix, this follows from \cite[Cor. 8.3.3]{ABW}.
\end{proof}

It turns out that these vertices satisfy three Yang-Baxter equations.  All three follow from \cite[Prop. 5.1.4]{ABW} plus Lemma \ref{lem:degen-abw}; for interested readers, we give a detailed derivation of Propositions \ref{prop:YBEpurplewhite} and \ref{prop:YBEpurple} in Appendix \ref{sec:abw-ybe-limits}.  Proposition \ref{prop:YBEwhite} is proven in a different way in \cite[Theorem 4.1]{CGKM}.

\begin{prop} \label{prop:YBEwhite}

The $L$ and $R$ matrices satisfy the Yang-Baxter equation
\[
\sum_{\text{interior paths}}
w\left(
\resizebox{1.8cm}{!}{
  \begin{tikzpicture}[baseline=(current bounding box.center)] 
 \draw[] (-1,0.5) -- (0,1.5); \draw[] (-1,1.5) -- (0,0.5); 
 \draw[step=1.0,black,thin] (0,0) grid (1,2); 
 \node at (0.5,0.5) {$x$}; \node at (0.5,1.5) {$y$};
 \node[left] at (-1,1.5) {$\J_1$}; \node[left] at (-1,0.5) {$\I_1$}; \node[below] at (0.5,0) {$\K_1$};
 \node[right] at (1,1.5) {$\I_3$}; \node[right] at (1,0.5) {$\J_3$}; \node[above] at (0.5,2) {$\K_3$};
 \end{tikzpicture} 
 }
 \right)
=
\sum_{\text{interior paths}}
w\left(
\resizebox{1.8cm}{!}{
 \begin{tikzpicture}[baseline=(current bounding box.center)] 
 \draw[] (2,0.5) -- (1,1.5); \draw[] (2,1.5) -- (1,0.5); 
 \draw[step=1.0,black,thin] (0,0) grid (1,2); 
 \node at (0.5,0.5) {$y$}; \node at (0.5,1.5) {$x$};
 \node[left] at (0,1.5) {$\J_1$}; \node[left] at (0,0.5) {$\I_1$}; \node[below] at (0.5,0) {$\K_1$};
 \node[right] at (2,0.5) {$\J_3$}; \node[right] at (2,1.5) {$\I_3$}; \node[above] at (0.5,2) {$\K_3$};
 \end{tikzpicture}
 }
 \right)
\]
for any choice of boundary condition $\I_1,\J_1,\K_1,\I_3,\J_3,\K_3$.

\end{prop}

\begin{prop} \label{prop:YBEpurplewhite}

The $L$, $L'$, and $R'$ matrices satisfy the Yang-Baxter equation
\[
\sum_{\text{interior paths}}
w\left(
\resizebox{1.8cm}{!}{
  \begin{tikzpicture}[baseline=(current bounding box.center)] 
  \draw[thin,fill=violet] (0,1) rectangle (1,2); \fill[yellow] (-1,0.5) rectangle (0,1.5);
 \draw[] (-1,0.5) -- (0,1.5); \draw[] (-1,1.5) -- (0,0.5); 
 \draw[step=1.0,black,thin] (0,0) grid (1,2); 
 \node at (0.5,0.5) {$x$}; \node at (0.5,1.5) {$y$};
 \node[left] at (-1,1.5) {$\J_1$}; \node[left] at (-1,0.5) {$\I_1$}; \node[below] at (0.5,0) {$\K_1$};
 \node[right] at (1,1.5) {$\I_3$}; \node[right] at (1,0.5) {$\J_3$}; \node[above] at (0.5,2) {$\K_3$};
 \end{tikzpicture} 
 }
 \right)
=
\sum_{\text{interior paths}}
w\left(
\resizebox{1.8cm}{!}{
 \begin{tikzpicture}[baseline=(current bounding box.center)] 
 \draw[thin,fill=violet] (0,0) rectangle (1,1); \fill[yellow] (1,0.5) rectangle (2,1.5);
 \draw[] (2,0.5) -- (1,1.5); \draw[] (2,1.5) -- (1,0.5); 
 \draw[step=1.0,black,thin] (0,0) grid (1,2); 
 \node at (0.5,0.5) {$y$}; \node at (0.5,1.5) {$x$};
 \node[left] at (0,1.5) {$\J_1$}; \node[left] at (0,0.5) {$\I_1$}; \node[below] at (0.5,0) {$\K_1$};
 \node[right] at (2,0.5) {$\J_3$}; \node[right] at (2,1.5) {$\I_3$}; \node[above] at (0.5,2) {$\K_3$};
 \end{tikzpicture}
 }
 \right)
\]
for any choice of boundary condition $\I_1,\J_1,\K_1,\I_3,\J_3,\K_3$.

\end{prop}

\begin{prop} \label{prop:YBEpurple}

The $L'$ and $R''$ matrices satisfy the Yang-Baxter equation
\[
\sum_{\text{interior paths}}
w\left(
\resizebox{1.8cm}{!}{
  \begin{tikzpicture}[baseline=(current bounding box.center)] 
  \draw[thin,fill=violet] (0,1) rectangle (1,2); \draw[thin,fill=violet] (0,0) rectangle (1,1); \fill[orange] (-1,0.5) rectangle (0,1.5);
 \draw[] (-1,0.5) -- (0,1.5); \draw[] (-1,1.5) -- (0,0.5); 
 \draw[step=1.0,black,thin] (0,0) grid (1,2); 
 \node at (0.5,0.5) {$x$}; \node at (0.5,1.5) {$y$};
 \node[left] at (-1,1.5) {$\J_1$}; \node[left] at (-1,0.5) {$\I_1$}; \node[below] at (0.5,0) {$\K_1$};
 \node[right] at (1,1.5) {$\I_3$}; \node[right] at (1,0.5) {$\J_3$}; \node[above] at (0.5,2) {$\K_3$};
 \end{tikzpicture} 
 }
 \right)
=
\sum_{\text{interior paths}}
w\left(
\resizebox{1.8cm}{!}{
 \begin{tikzpicture}[baseline=(current bounding box.center)] 
 \draw[thin,fill=violet] (0,0) rectangle (1,1); \draw[thin,fill=violet] (0,1) rectangle (1,2); \fill[orange] (1,0.5) rectangle (2,1.5);
 \draw[] (2,0.5) -- (1,1.5); \draw[] (2,1.5) -- (1,0.5); 
 \draw[step=1.0,black,thin] (0,0) grid (1,2); 
 \node at (0.5,0.5) {$y$}; \node at (0.5,1.5) {$x$};
 \node[left] at (0,1.5) {$\J_1$}; \node[left] at (0,0.5) {$\I_1$}; \node[below] at (0.5,0) {$\K_1$};
 \node[right] at (2,0.5) {$\J_3$}; \node[right] at (2,1.5) {$\I_3$}; \node[above] at (0.5,2) {$\K_3$};
 \end{tikzpicture}
 }
 \right)
\]
for any choice of boundary condition $\I_1,\J_1,\K_1,\I_3,\J_3,\K_3$.

\end{prop}

\indent We use these vertices to construct three classes of partition functions, the first of which was studied in \cite{CGKM}.  Given a lattice $L$, the associated \textbf{partition function} is
\[ \sum_{C \in LC(L)} \weight(C) \]
where $LC(L)$ is the set of valid configurations on the lattice $L$.  In what follows, let 
\[ \lm = (\lambda^{(1)}/\mu^{(1)},\ldots,\lambda^{(k)}/\mu^{(k)}) \]
be a $k$-tuple of skew partitions, each having $p$ parts.

\begin{thm} \cite[Theorem 3.4]{CGKM} \label{thm:cgkm-lattice} The coinversion LLT polynomial $\mc{L}_{\lm}(X_n;t)$ is the partition function associated to the lattice
\[ 
\latW_n(\lm) :=
\resizebox{2cm}{!}{
\begin{tikzpicture}[baseline=(current bounding box.center)] 
\draw (0,0) grid (3,3);
\node[left] at (0,1.5) {$\vdots$}; \node[below] at (1.5,3) {$\ldots$};
\node[below] at (1.5,0) {$\bm{\mu}$}; \node[above] at (1.5,3) {$\bm{\lambda}$};
\node[left] at (0,0.5) {$x_1$}; \node[left] at (0,2.5) {$x_n$}; 
\end{tikzpicture} }.
\] 
\end{thm}

\begin{definition} \label{def:latP} We define $\mc{L}^P_{\lm}(X_n;t)$ to be the partition function associated to the lattice
\[ 
\latP_n(\lm) := 
\resizebox{2cm}{!}{
\begin{tikzpicture}[baseline=(current bounding box.center)] 
\draw[fill=violet] (0,0) rectangle (3,3); \draw (0,0) grid (3,3);
\node[left] at (0,1.5) {$\vdots$}; \node[below] at (1.5,3) {$\ldots$};
\node[below] at (1.5,0) {$\bm{\mu}$}; \node[above] at (1.5,3) {$\bm{\lambda}$};
\node[left] at (0,0.5) {$x_1$}; \node[left] at (0,2.5) {$x_n$}; 
\end{tikzpicture} } .
\] 
\end{definition}

\begin{definition} \label{def:latS} We define $\mc{L}^S_{\lm}(X_n;Y_m;t)$ to be the partition function associated to the lattice
\[ 
\latS_{n,m}(\lm) := 
\resizebox{2cm}{!}{
\begin{tikzpicture}[baseline=(current bounding box.center)] 
\draw[fill=violet] (0,3) rectangle (3,6); \draw (0,0) grid (3,6);
\node[left] at (0,1.5) {$\vdots$}; \node[left] at (0,4.5) {$\vdots$}; \node[below] at (1.5,6) {$\ldots$};
\node[below] at (1.5,0) {$\bm{\mu}$}; \node[above] at (1.5,6) {$\bm{\lambda}$};
\node[left] at (0,0.5) {$x_1$}; \node[left] at (0,2.5) {$x_n$}; \node[left] at (0,3.5) {$y_1$}; \node[left] at (0,5.5) {$y_m$}; 
\end{tikzpicture} } .
\] 
\end{definition}

Often, when it is clear from context, we will abuse notation and let the drawing of the lattice be equal to the partition function of the vertex model on the lattice.

\pagebreak

\section{Identities of Supersymmetric LLT Polynomials} \label{sec:identities}

\indent The goal of this section is to establish various properties of the partition functions $\mc{L}^S_\lm$ of Definition \ref{def:latS}.  These include four properties (summarized in Theorem \ref{thm:main3}) which generalize four properties that uniquely characterize the supersymmetric Schur polynomials.

\indent Let $P_p^{(k)}$ be the set of $k$-tuples of partitions, each having $p$ parts. Given $\bm{\lambda},\bm{\mu} \in P_p^{(k)}$ and a nonnegative integer $n$, there is a bijection
\[ \psi : LC(\latW_n(\lm)) \rightarrow LC(\latP_n(\lmp)) \]
where $\latW_n(\lm)$ is defined in Theorem \ref{thm:cgkm-lattice} and $\latP_n(\lmp)$ is defined in Definition \ref{def:latP}.  Explicitly, $\psi(C)$ is obtained from $C$ by inverting the vertical parts of the paths, reflecting the lattice over its left edge, changing each color $i$ to color $k-i$, and making the vertices purple.  For example,
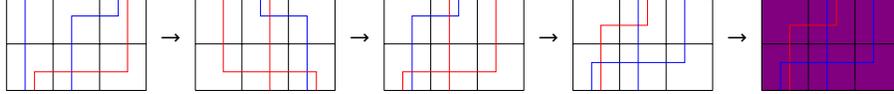
\begin{center}
\resizebox{2cm}{!}{
\begin{tikzpicture}[baseline=(current bounding box.center)]
\draw (0,0) grid (3,2);
\draw[blue] (0.4,0)--(0.4,2); \draw[blue] (1.4,0)--(1.4,1.6)--(2.4,1.6)--(2.4,2);
\draw[red] (0.6,0)--(0.6,0.4)--(2.6,0.4)--(2.6,2);
\end{tikzpicture} } 
$\rightarrow$
\resizebox{2cm}{!}{
\begin{tikzpicture}[baseline=(current bounding box.center)]
\draw (0,0) grid (3,2);
\draw[blue] (1.4,2)--(1.4,1.6)--(2.4,1.6)--(2.4,0);
\draw[red] (0.6,2)--(0.6,0.4)--(2.6,0.4)--(2.6,0); \draw[red] (1.6,0)--(1.6,2);
\end{tikzpicture} } 
$\rightarrow$
\resizebox{2cm}{!}{
\begin{tikzpicture}[baseline=(current bounding box.center)]
\draw (0,0) grid (3,2);
\draw[blue] (3-1.4,2)--(3-1.4,1.6)--(3-2.4,1.6)--(3-2.4,0);
\draw[red] (3-0.6,2)--(3-0.6,0.4)--(3-2.6,0.4)--(3-2.6,0); \draw[red] (3-1.6,0)--(3-1.6,2);
\end{tikzpicture} } 
$\rightarrow$
\resizebox{2cm}{!}{
\begin{tikzpicture}[baseline=(current bounding box.center)]
\draw (0,0) grid (3,2);
\draw[red] (0.6,0)--(0.6,1.4)--(1.6,1.4)--(1.6,2);
\draw[blue] (0.4,0)--(0.4,0.6)--(2.4,0.6)--(2.4,2); \draw[blue] (1.4,0)--(1.4,2);
\end{tikzpicture} } 
$\rightarrow$
\resizebox{2cm}{!}{
\begin{tikzpicture}[baseline=(current bounding box.center)]
\draw[fill=violet] (0,0) rectangle (3,2); \draw (0,0) grid (3,2);
\draw[red] (0.6,0)--(0.6,1.4)--(1.6,1.4)--(1.6,2);
\draw[blue] (0.4,0)--(0.4,0.6)--(2.4,0.6)--(2.4,2); \draw[blue] (1.4,0)--(1.4,2);
\end{tikzpicture} } 
\end{center}
where we have done each step in order (and where blue is color 1 and red is color 2).

\indent Fix $\bm{\lambda},\bm{\mu} \in P_p^{(k)}$.  Also fix a sufficiently large number of columns; specifically, the number of columns must be larger than each $p+\lambda^{(i)}_1$ and $p+\mu^{(i)}_1$.  If $\lm$ is a horizontal strip, then there is a unique configuration $C_{\lm}$ on a single white row with top boundary $\bm{\lambda}$, bottom boundary $\bm{\mu}$, and empty left/right boundaries.  Similarly, if $\lm$ is a vertical strip, then there is a unique configuration on a single purple row with top boundary $\bm{\lambda}$, bottom boundary $\bm{\mu}$, and empty left/right boundaries.

\begin{lem} \label{lem:dream}
There exists a function $\g : P_p^{(k)} \rightarrow \Z_{\geq 0}$ so that
\[ \mc{L}_{\lm}(x;t) = t^{\g(\bm{\lambda})-\g(\bm{\mu})} \mc{L}^P_{\lmp}(x;t^{-1}) \]
for all $\bm{\lambda},\bm{\mu} \in P_p^{(k)}$ such that $\lm$ is a horizontal strip.
\end{lem}

\begin{proof} 
Fix $\bm{\lambda},\bm{\mu} \in P_p^{(k)}$ such that $\lm$ is a horizontal strip.  Let $C = C_{\lm}$.  Note that $C$ is the unique configuration on a single white row with top boundary $\bm{\lambda}$, bottom boundary $\bm{\mu}$, and empty left/right boundaries. Moreover, $\psi(C)$ is the unique configuration on a single purple row with top boundary $\bm{\lambda}'$, bottom boundary $\bm{\mu}'$, and empty left/right boundaries.  Thus
\[ \begin{aligned}
&\mc{L}_{\lm}(x;t) = \weight(C) = x^\alpha t^\beta, \\
&\mc{L}^P_{\lm}(x;t) = \weight(\psi(C)) = x^\gamma t^\delta
\end{aligned} \]
for some nonnegative integers $\alpha,\beta,\gamma,\delta$.  Note
\[ \begin{aligned}
\alpha &= \# \{ (i,j) \mid \text{the $i$-th smallest color exits right in the $b$-th leftmost box in $C$} \} \\
&=\# \{ (i,j) \mid \text{the $i$-th largest color exits right in the $b$-th rightmost box in $\psi(C)$} \} = \gamma.
\end{aligned} \]
Therefore 
\[ \mc{L}_{\lm}(x;t) = x^\alpha t^\beta = t^{\beta+\delta} x^\gamma t^{-\delta} = t^{\beta+\delta} \mc{L}^P_{\lm}(x;t^{-1}). \]
Also note
\[ \begin{aligned}
\beta + \delta &= \# \{ (i,j,b) \mid \text{$i < j$, in box $b$ of $C$ color $i$ exits right and color $j$ is present} \} \\
&\hspace{0.5cm} + \# \{ (i',j',b') \mid \text{$i' > j'$, in box $b'$ of $\psi(C)$ color $i'$ is vertical and color $j'$ exits right} \} \\
&= \# \{ (i,j,b) \mid \text{$i < j$, in box $b$ of $C$ color $i$ exits right and color $j$ is present} \} \\
&\hspace{0.5cm} + \# \{ (i,j,b) \mid \text{$i < j$, in box $b$ of $C$ color $i$ is absent and color $j$ enters left} \} \\
&= \sum_b \sum_{i < j} \1_\text{$b$ is ``good" for $i$ and $j$}
\end{aligned} \]
where we say a box $b$ is ``good" for the colors $i < j$ if either color $i$ exits right and color $j$ is present, or color $i$ is absent and color $j$ enters left.
Therefore 
\[ \mc{L}_{\lm}(x;t) = t^{\tilde{\g}(\lm)} \mc{L}^P_{\lmp}(x;t^{-1}) \]
where we have defined
\[ \tilde{\g}(\bm{\lambda/\mu}) := \sum_{\substack{\text{boxes $b$} \\ \text{of $C_\lm$}}} \sum_{\substack{\text{colors} \\ i < j}} \1_\text{$b$ is ``good" for $i$ and $j$}. \]

We can recursively define the desired function $\g$ by the rule
\[ \resizebox{0.85\textwidth}{!}{
$
\g(\bm{\lambda}) = \left \{ \begin{array}{ll} \g(\bm{\mu}) + \tilde{\g}(\lm) & \text{if there exists $\bm{\mu} \in P_p^{(k)} - \{ \bm{\lambda} \}$ such that $\lm$ is a horizontal strip} \\ 0 & \text{otherwise (i.e. if $\bm{\lambda} = \0$)} \end{array} \right .. 
$
}
\]
Provided that $\g(\bm{\lambda})$ is well-defined (i.e. $\g(\bm{\mu}) + \tilde{\g}(\lm)$ is independent of $\bm{\mu}$) for all $\bm{\lambda} \in P_p^{(k)}$,
\[ \mc{L}_{\lm}(x;t) = t^{\tilde{\g}(\lm)} \mc{L}^P_{\lmp}(x;t^{-1}) = t^{\g(\bm{\lambda})-\g(\bm{\mu})} \mc{L}^P_{\lmp}(x;t^{-1}) \]
for all $\bm{\lambda},\bm{\mu} \in P_p^{(k)}$ such that $\lm$ is a horizontal strip.  To show $\g$ is well-defined, we induct on the number of cells in $\bm{\lambda}$.  Clearly $\g(\0) = 0$ is well-defined.  Fix $\bm{\lambda} \in P_p^{(k)} - \{\0\}$ and assume $\g$ is well-defined on elements of $P_p^{(k)}$ with strictly fewer boxes than $\bm{\lambda}$.  Fix $\bm{\alpha},\bm{\mu} \in P_p^{(k)} - \{ \bm{\lambda} \}$ such that $\bm{\lambda}/\bm{\alpha}$ and $\lm$ are horizontal strips.  There exist (not necessarily distinct) 
\[ \begin{aligned}
&\bm{\beta}^0 = \0,\ldots,\bm{\beta}^{r-1} = \bm{\alpha},\bm{\beta}^r = \bm{\lambda} \in P_p^{(k)}, \\
&\bm{\nu}^0 = \0,\ldots,\bm{\nu}^{r-1} = \bm{\mu}, \bm{\nu}^r = \bm{\lambda} \in P_p^{(k)} 
\end{aligned} \]
such that $\bm{\beta}^i / \bm{\beta}^{i-1}$ and $\bm{\nu}^i / \bm{\nu}^{i-1}$ are horizontal strips for all $i$.  Note that each sequence completely determines a configuration of paths on an $r\times M$ lattice with top boundary $\bm{\lambda}$ and bottom boundary $\bm{0}$, since they determine the state of the paths at every row.

Since the two configurations have the same top and bottom boundary, it is possible to get from one configuration to the other via corner flips
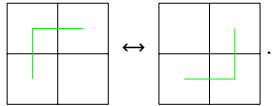
\[ 
\resizebox{1.5cm}{!}{
\begin{tikzpicture}[baseline=(current bounding box.center)] 
\draw (0,0) grid (2,2);
\draw[green] (0.5,0.5) -- (0.5,1.5) -- (1.5,1.5);
\end{tikzpicture} } 
\leftrightarrow
\resizebox{1.5cm}{!}{
\begin{tikzpicture}[baseline=(current bounding box.center)] 
\draw (0,0) grid (2,2);
\draw[green] (0.5,0.5) -- (1.5,0.5) -- (1.5,1.5);
\end{tikzpicture} } .
\] 
Some straightforward computations (which we verified with a computer) show that, for any configuration on a $2 \times 2$ white lattice and for any colors $h < i < j$ such that color $i$ has the form
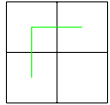
\[ 
\resizebox{1.5cm}{!}{
\begin{tikzpicture}[baseline=(current bounding box.center)] 
\draw (0,0) grid (2,2);
\draw[green] (0.5,0.5) -- (0.5,1.5) -- (1.5,1.5);
\end{tikzpicture} } 
\]
in the configuration, flipping color $i$ changes neither the number of boxes that are ``good" for $h$ and $i$, nor the number of boxes that are ``good" for $i$ and $j$. Using the flipping argument repeatedly, we see that the quantity
\[ \sum_b \sum_{i < j} \1_\text{$b$ is ``good" for $i$ and $j$} \]
is the same for both configurations. Therefore
\[ \tilde{\g}(\bm{\beta}^r/\bm{\beta}^{r-1}) + \tilde{\g}(\bm{\beta}^{r-1}/\bm{\beta}^{r-2}) + \ldots + \tilde{\g}(\bm{\beta}^1/\bm{\beta}^0) = \tilde{\g}(\bm{\nu}^r/\bm{\nu}^{r-1}) + \tilde{\g}(\bm{\nu}^{r-1}/\bm{\nu}^{r-2}) + \ldots + \tilde{\g}(\bm{\nu}^1/\bm{\nu}^0). \]
Applying the inductive hypothesis, we have
\[ \begin{aligned}
&\tilde{\g}(\bm{\beta}^r/\bm{\beta}^{r-1}) + \g(\bm{\beta}^{r-1}) - \g(\bm{\beta}^{r-2}) + \ldots + \g(\bm{\beta}^1) - \g(\bm{\beta}^0) \\
&= \tilde{\g}(\bm{\nu}^r/\bm{\nu}^{r-1}) + \g(\bm{\nu}^{r-1}) - \g(\bm{\nu}^{r-2}) + \ldots + \g(\bm{\nu}^1) - \g(\bm{\nu}^0). 
\end{aligned} \]
The sums telescope to give
\[ \tilde{\g}(\bm{\beta}^r/\bm{\beta}^{r-1}) + \g(\bm{\beta}^{r-1}) - \g(\bm{\beta}^0) = \tilde{\g}(\bm{\nu}^r/\bm{\nu}^{r-1}) + \g(\bm{\nu}^{r-1}) - \g(\bm{\nu}^0) \]
which we can rewrite as
\[ \tilde{\g}(\bm{\lambda}/\bm{\alpha}) + \g(\bm{\alpha}) - \g(\0) = \tilde{\g}(\lm) + \g(\bm{\mu}) - \g(\0). \]
Therefore $\g(\bm{\alpha}) + \tilde{\g}(\bm{\lambda}/\bm{\alpha}) = \g(\bm{\mu}) + \tilde{\g}(\lm)$.
\end{proof}

\begin{cor} \label{cor:dconjugate}
$\g(\bm{\lambda}) = \g(\bm{\lambda}')$
\end{cor}

\begin{proof}

We induct on the number of cells in $\bm{\lambda}$.  Note that $\g(\0) = 0 = \g(\0')$.  Fix $\bm{\lambda} \in P_p^{(k)} - \{\0\}$ and assume $\g(\bm{\mu}) = \g(\bm{\mu}')$ for all $\bm{\mu} \in P_p^{(k)}$ with strictly fewer cells than $\bm{\lambda}$.  Since $\bm{\lambda} \neq \0$, there exists $\bm{\mu} \in P_p^{(k)}$ that can be obtained by removing a single cell $u$ from $\bm{\lambda}$.  We want to show $\g(\bm{\lambda}) = \g(\bm{\lambda}')$.  It is enough to show $\tilde{\g}(\lm) = \tilde{\g}(\lmp)$, since then
\[ \g(\bm{\lambda}) = \g(\bm{\mu}) + \tilde{\g}(\lm) = \g(\bm{\mu}') + \tilde{\g}(\lmp) = \g(\bm{\lambda}'). \]

Let $\lambda^{(i)}$ be the partition to which $u$ belongs.  Since $\lm$ consists of the single cell $u$ in $\lambda^{(i)}/\mu^{(i)}$, every color in every box in $C_\lm$ is either vertical or absent, with the exception of the color $i$ in two adjacent boxes, which has the form
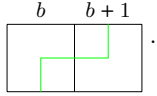
\[ 
\resizebox{2cm}{!}{
\begin{tikzpicture}[baseline=(current bounding box.center)] 
\draw (0,0) grid (2,1);
\draw[green] (0.5,0) -- (0.5,0.5) -- (1.5,0.5) -- (1.5,1);
\node[above] at (0.5,1) {$b$}; \node[above] at (1.5,1) {$b+1$}; 
\end{tikzpicture} } .
\]
Note that $C_\lmp$ is exactly the configuration $\psi(C_\lm)$ with white in place of purple.  Therefore every color in every box in $C_\lmp$ is either vertical or absent, with the exception of the color $i' = k-i$ in two adjacent boxes, which has the form 
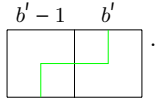
\[ 
\resizebox{2cm}{!}{
\begin{tikzpicture}[baseline=(current bounding box.center)] 
\draw (0,0) grid (2,1);
\draw[green] (0.5,0) -- (0.5,0.5) -- (1.5,0.5) -- (1.5,1);
\node[above] at (0.5,1) {$b'-1$}; \node[above] at (1.5,1) {$b'$}; 
\end{tikzpicture} } .
\]
We have
\[ 
\resizebox{0.85\textwidth}{!}{$
\begin{aligned}
\tilde{\g}(\lm) =& \# \{ j > i : \text{$j$ is vertical in box $b$ of $C_\lm$} \} + \# \{ h < i : \text{$h$ is absent in box $b+1$ of $C_\lm$} \} \\
=& \# \{ j' < i' : \text{$j'$ is absent in box $b'$ of $C_\lmp$} \} + \# \{ h' > i' : \text{$h'$ is vertical in box $b'-1$ of $C_\lmp$} \} \\
=& \tilde{\g}(\lmp).
\end{aligned}
$}
\]
\end{proof}

\begin{thm} \label{thm:SSLLTxyswap}
\[ \mc{L}^S_{\lm}(X_n;Y_m;t) = t^{\g(\bm{\lambda}) - \g(\bm{\mu})} \mc{L}^S_{\lmp}(Y_m;X_n;t^{-1}) \]
\end{thm}

\begin{proof} If $\lm$ is a horizontal strip, then by Lemma \ref{lem:dream}, we have
\[ \mc{L}_{\lm}(x;t) = t^{\g(\bm{\lambda})-\g(\bm{\mu})} \mc{L}^P_{\lmp}(x;t^{-1}). \]
If $\lm$ is a vertical strip, then by Lemma \ref{lem:dream} (with $t^{-1}$ in place of $t$ and $\lmp$ in place of $\lm$) and Corollary \ref{cor:dconjugate}, we have
\[ \mc{L}^P_{\lm}(x;t) = t^{\g(\bm{\lambda})-\g(\bm{\mu})} \mc{L}_{\lmp}(x;t^{-1}). \]
Using these lemmas at each row of our lattice, we have
\[ 
\resizebox{0.85\textwidth}{!}{
$
\begin{aligned}
&\mc{L}^S_{\lm}(X_n;Y_m;t) =  
\resizebox{2cm}{!}{
\begin{tikzpicture}[baseline=(current bounding box.center)] 
\draw[fill=violet] (0,3) rectangle (3,6); \draw (0,0) grid (3,6);
\node[left] at (0,1.5) {$\vdots$}; \node[left] at (0,4.5) {$\vdots$}; \node[below] at (1.5,6) {$\ldots$};
\node[below] at (1.5,0) {$\bm{\mu}$}; \node[above] at (1.5,6) {$\bm{\lambda}$};
\node[left] at (0,0.5) {$x_1$}; \node[left] at (0,2.5) {$x_n$}; \node[left] at (0,3.5) {$y_1$}; \node[left] at (0,5.5) {$y_m$}; 
\end{tikzpicture} } \\
&= \sum \mc{L}^P_{\bm{\lambda}/\bm{\alpha^{m+n-1}}}(y_m;t)\ldots\mc{L}^P_{\bm{\alpha^{n+1}}/\bm{\alpha^{n}}}(y_1;t) \mc{L}_{\bm{\alpha^{n}}/\bm{\alpha^{n-1}}}(x_n;t)\ldots\mc{L}_{\bm{\alpha^{1}}/\bm{\mu}}(x_1;t) \\
&= \sum t^{\g(\bm{\lambda})-\g(\bm{\mu})}\mc{L}_{\bm{\lambda}'/\bm{\alpha^{m+n-1'}}}(y_m;t^{-1})\ldots\mc{L}_{\bm{\alpha^{n+1'}}/\bm{\alpha^{n'}}}(y_1;t^{-1}) \mc{L}^P_{\bm{\alpha^{n'}}/\bm{\alpha^{n-1'}}}(x_n;t^{-1})\ldots\mc{L}^P_{\bm{\alpha^{1'}}/\bm{\mu}'}(x_1;t^{-1}) \\
&= t^{\g(\bm{\lambda})-\g(\bm{\mu})} \sum \mc{L}_{\bm{\lambda}'/\bm{\beta^{m+n-1}}}(y_m;t^{-1})\ldots\mc{L}_{\bm{\beta^{n+1}}/\bm{\beta^{n}}}(y_1;t^{-1}) \mc{L}^P_{\bm{\beta^{n}}/\bm{\beta^{n-1}}}(x_n;t^{-1})\ldots\mc{L}^P_{\bm{\beta^{1}}/\bm{\mu}'}(x_1;t^{-1}) \\
&= t^{\g(\bm{\lambda})-\g(\bm{\mu})} \resizebox{2cm}{!}{
\begin{tikzpicture}[baseline=(current bounding box.center)] 
\draw[fill=violet] (0,0) rectangle (3,3); \draw (0,0) grid (3,6);
\node[left] at (0,1.5) {$\vdots$}; \node[left] at (0,4.5) {$\vdots$}; \node[below] at (1.5,6) {$\ldots$};
\node[below] at (1.5,0) {$\bm{\mu}'$}; \node[above] at (1.5,6) {$\bm{\lambda}'$};
\node[left] at (0,0.5) {$x_1$}; \node[left] at (0,2.5) {$x_n$}; \node[left] at (0,3.5) {$y_1$}; \node[left] at (0,5.5) {$y_m$}; 
\end{tikzpicture} }
= t^{\g(\bm{\lambda})-\g(\bm{\mu})} \resizebox{2cm}{!}{
\begin{tikzpicture}[baseline=(current bounding box.center)] 
\draw[fill=violet] (0,3) rectangle (3,6); \draw (0,0) grid (3,6);
\node[left] at (0,1.5) {$\vdots$}; \node[left] at (0,4.5) {$\vdots$}; \node[below] at (1.5,6) {$\ldots$};
\node[below] at (1.5,0) {$\bm{\mu}'$}; \node[above] at (1.5,6) {$\bm{\lambda}'$};
\node[left] at (0,0.5) {$y_1$}; \node[left] at (0,2.5) {$y_m$}; \node[left] at (0,3.5) {$x_1$}; \node[left] at (0,5.5) {$x_n$}; 
\end{tikzpicture} }
= \mc{L}^S_{\lmp}(Y_m;X_n;t^{-1})
\end{aligned}
$
}
\]
where 
\begin{itemize}
\item the sums in the second and third lines are over all $\bm{\alpha}^0 = \bm{\mu},\ldots,\bm{\alpha}^{m+n} = \bm{\lambda}$ such that $\bm{\alpha}^i / \bm{\alpha}^{i-1}$ is a horizontal strip for all $i \leq n$ and a vertical strip for all $i > n$, 
\item the sum in the fourth line is over all $\bm{\beta}^0 = \bm{\mu}',\ldots,\bm{\beta}^{m+n} = \bm{\lambda}'$ such that $\bm{\beta}^i / \bm{\beta}^{i-1}$ is a vertical strip for all $i \leq n$ and a horizontal strip for all $i > n$, and
\item the second-to-last equality uses repeated applications of Prop. \ref{prop:YBEpurplewhite}, as illustrated below.
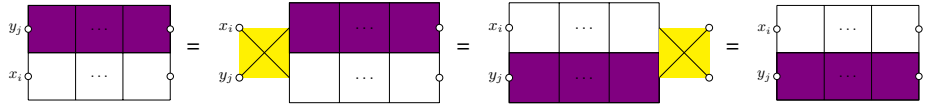
\[ \begin{aligned}
\resizebox{2.4cm}{!}{ \begin{tikzpicture}[baseline=(current bounding box.center)] 
\draw[fill=violet] (0,1) rectangle (3,2); \draw (0,0) grid (3,2); \node at (1.5,0.5) {$\ldots$}; \node at (1.5,1.5) {$\ldots$};
\draw[black,fill=white] (0,1.5) circle (.5ex); \draw[black,fill=white] (0,0.5) circle (.5ex);
\draw[black,fill=white] (3,1.5) circle (.5ex); \draw[black,fill=white] (3,0.5) circle (.5ex);
\node[left] at (0,0.5) {$x_i$}; \node[left] at (0,1.5) {$y_j$};
\end{tikzpicture} }
=
\resizebox{3.2cm}{!}{ \begin{tikzpicture}[baseline=(current bounding box.center)] 
\fill[yellow] (-1,0.5) rectangle (0,1.5); \draw (-1,1.5)--(0,0.5); \draw (-1,0.5)--(0,1.5);
\draw[black,fill=white] (-1,1.5) circle (.5ex); \draw[black,fill=white] (-1,0.5) circle (.5ex);
\draw[fill=violet] (0,1) rectangle (3,2); \draw (0,0) grid (3,2); \node at (1.5,0.5) {$\ldots$}; \node at (1.5,1.5) {$\ldots$};
\draw[black,fill=white] (3,1.5) circle (.5ex); \draw[black,fill=white] (3,0.5) circle (.5ex);
\node[left] at (-1,0.5) {$y_j$}; \node[left] at (-1,1.5) {$x_i$}; 
\end{tikzpicture} }
=
\resizebox{3.2cm}{!}{ \begin{tikzpicture}[baseline=(current bounding box.center)] 
\draw[fill=violet] (0,0) rectangle (3,1); \draw (0,0) grid (3,2); \node at (1.5,0.5) {$\ldots$}; \node at (1.5,1.5) {$\ldots$};
\draw[black,fill=white] (0,1.5) circle (.5ex); \draw[black,fill=white] (0,0.5) circle (.5ex);
\node[left] at (0,0.5) {$y_j$}; \node[left] at (0,1.5) {$x_i$};
\fill[yellow] (3,0.5) rectangle (4,1.5); \draw (3,1.5)--(4,0.5); \draw (3,0.5)--(4,1.5);
\draw[black,fill=white] (4,1.5) circle (.5ex); \draw[black,fill=white] (4,0.5) circle (.5ex);
\end{tikzpicture} }
=
\resizebox{2.4cm}{!}{ \begin{tikzpicture}[baseline=(current bounding box.center)] 
\draw[fill=violet] (0,0) rectangle (3,1); \draw (0,0) grid (3,2); \node at (1.5,0.5) {$\ldots$}; \node at (1.5,1.5) {$\ldots$};
\draw[black,fill=white] (0,1.5) circle (.5ex); \draw[black,fill=white] (0,0.5) circle (.5ex);
\draw[black,fill=white] (3,1.5) circle (.5ex); \draw[black,fill=white] (3,0.5) circle (.5ex);
\node[left] at (0,0.5) {$y_j$}; \node[left] at (0,1.5) {$x_i$};
\end{tikzpicture} }
\end{aligned} \]
\end{itemize}
\end{proof}

The technique used to swap a white row and a purple row at the end of the previous proof is sometimes called the ``train argument."  This technique is used again to prove the following lemma.

\begin{lem} \label{lem:permuterows}
The partition function associated to any lattice that can be obtained from the lattice $\latS_{n,m}$ (Definition \ref{def:latS}) by permuting the rows is equal to $\mc{L}_\lm^S(X_n;Y_m;t)$.  In particular, $\mc{L}_\lm^S(X_n;Y_m;t)$ is symmetric in the $X$ and $Y$ variables separately. 
\end{lem}

\begin{proof}
Two rows can be swapped using the train argument along with Proposition \ref{prop:YBEwhite} (to swap two white rows), Proposition \ref{prop:YBEpurplewhite} (to swap a white row and a purple row), or Proposition \ref{prop:YBEpurple} (to swap two purple rows).
\end{proof}

\begin{lem} \label{lem:cancellation}
Suppose $n,m \geq 1$.  Then $\mc{L}_\lm^S(X_{n-1},r;Y_{m-1},-r;t) = \mc{L}_\lm^S(X_{n-1};Y_{m-1};t)$.
\end{lem}

\begin{proof}

Using Lemma \ref{lem:permuterows}, we can write
\[ \begin{aligned}
\mc{L}_\lm^S(X_{n-1},r;Y_{m-1},-r;t) 
= \resizebox{2cm}{!}{
\begin{tikzpicture}[baseline=(current bounding box.center)] 
\draw[fill=violet] (0,3) rectangle (3,6); \draw[fill=violet] (0,7) rectangle (3,8); \draw (0,0) grid (3,8);
\node[left] at (0,1.5) {$\vdots$}; \node[left] at (0,4.5) {$\vdots$}; \node[below] at (1.5,8) {$\ldots$};
\node[below] at (1.5,0) {$\bm{\mu}$}; \node[above] at (1.5,8) {$\bm{\lambda}$};
\node[left] at (0,0.5) {$x_1$}; \node[left] at (0,2.5) {$x_{n-1}$}; \node[left] at (0,3.5) {$y_1$}; \node[left] at (0,5.5) {$y_{m-1}$}; \node[left] at (0,6.5) {$r$}; \node[left] at (0,7.5) {$-r$}; 
\end{tikzpicture} }
= \sum_{\bm{\alpha}}
\resizebox{2cm}{!}{
\begin{tikzpicture}[baseline=(current bounding box.center)] 
\draw[fill=violet] (0,1) rectangle (3,2); \draw (0,0) grid (3,2);
\node[below] at (1.5,2) {$\ldots$};
\node[below] at (1.5,0) {$\bm{\alpha}$}; \node[above] at (1.5,2) {$\bm{\lambda}$};
\node[left] at (0,0.5) {$r$}; \node[left] at (0,1.5) {$-r$}; 
\end{tikzpicture} }
\resizebox{2cm}{!}{
\begin{tikzpicture}[baseline=(current bounding box.center)] 
\draw[fill=violet] (0,3) rectangle (3,6); \draw (0,0) grid (3,6);
\node[left] at (0,1.5) {$\vdots$}; \node[left] at (0,4.5) {$\vdots$}; \node[below] at (1.5,6) {$\ldots$};
\node[below] at (1.5,0) {$\bm{\mu}$}; \node[above] at (1.5,6) {$\bm{\alpha}$};
\node[left] at (0,0.5) {$x_1$}; \node[left] at (0,2.5) {$x_{n-1}$}; \node[left] at (0,3.5) {$y_1$}; \node[left] at (0,5.5) {$y_{m-1}$};
\end{tikzpicture} } .
\end{aligned} \]
We will show that, for all $\bm{\alpha} \neq \bm{\lambda}$, there is an involution $\varphi_{\bm{\alpha}}$ on the set of configurations of the lattice
\[ 
L_{\bm{\lambda}/\bm{\alpha}} = \resizebox{2cm}{!}{
\begin{tikzpicture}[baseline=(current bounding box.center)] 
\draw[fill=violet] (0,1) rectangle (3,2); \draw (0,0) grid (3,2);
\node[below] at (1.5,2) {$\ldots$};
\node[below] at (1.5,0) {$\bm{\alpha}$}; \node[above] at (1.5,2) {$\bm{\lambda}$};
\node[left] at (0,0.5) {$r$}; \node[left] at (0,1.5) {$-r$}; 
\end{tikzpicture} } 
\]
such that $\weight(\varphi_{\bm{\alpha}}(C)) = -\weight(C)$ for all $C$.  Therefore
\[ 
\resizebox{0.85\textwidth}{!}{
$
\mc{L}_\lm^S(X_{n-1},r;Y_{m-1},-r;t) =
\resizebox{2cm}{!}{
\begin{tikzpicture}[baseline=(current bounding box.center)] 
\draw[fill=violet] (0,1) rectangle (3,2); \draw (0,0) grid (3,2);
\node[below] at (1.5,2) {$\ldots$};
\node[below] at (1.5,0) {$\bm{\lambda}$}; \node[above] at (1.5,2) {$\bm{\lambda}$};
\node[left] at (0,0.5) {$r$}; \node[left] at (0,1.5) {$-r$}; 
\end{tikzpicture} }
\resizebox{2cm}{!}{
\begin{tikzpicture}[baseline=(current bounding box.center)] 
\draw[fill=violet] (0,3) rectangle (3,6); \draw (0,0) grid (3,6);
\node[left] at (0,1.5) {$\vdots$}; \node[left] at (0,4.5) {$\vdots$}; \node[below] at (1.5,6) {$\ldots$};
\node[below] at (1.5,0) {$\bm{\mu}$}; \node[above] at (1.5,6) {$\bm{\lambda}$};
\node[left] at (0,0.5) {$x_1$}; \node[left] at (0,2.5) {$x_{n-1}$}; \node[left] at (0,3.5) {$y_1$}; \node[left] at (0,5.5) {$y_{m-1}$};
\end{tikzpicture} } 
=
\resizebox{2cm}{!}{
\begin{tikzpicture}[baseline=(current bounding box.center)] 
\draw[fill=violet] (0,3) rectangle (3,6); \draw (0,0) grid (3,6);
\node[left] at (0,1.5) {$\vdots$}; \node[left] at (0,4.5) {$\vdots$}; \node[below] at (1.5,6) {$\ldots$};
\node[below] at (1.5,0) {$\bm{\mu}$}; \node[above] at (1.5,6) {$\bm{\lambda}$};
\node[left] at (0,0.5) {$x_1$}; \node[left] at (0,2.5) {$x_{n-1}$}; \node[left] at (0,3.5) {$y_1$}; \node[left] at (0,5.5) {$y_{m-1}$};
\end{tikzpicture} } 
= \mc{L}_\lm^S(X_{n-1};Y_{m-1};t).
$
}
\]

Fix $\bm{\alpha} \neq \bm{\lambda}$ and fix a configuration $C$ on the lattice $L_{\bm{\lambda}/\bm{\alpha}}$.  Since $\bm{\alpha} \neq \bm{\lambda}$, there exist two consecutive columns $c$ and $c+1$ of $C$ and a color $i$ such that, in columns $c$ and $c+1$ of $C$, color $i$ has the form
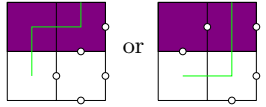
\[ 
\resizebox{1.5cm}{!}{
\begin{tikzpicture}[baseline=(current bounding box.center)] 
\draw[fill=violet] (0,1) rectangle (2,2); \draw (0,0) grid (2,2);
\draw[green] (0.5,0.5) -- (0.5,1.5) -- (1.5,1.5) -- (1.5,2);
\draw[black,fill=white] (1.5,0) circle (.5ex); \draw[black,fill=white] (2,0.5) circle (.5ex); \draw[black,fill=white] (2,1.5) circle (.5ex); \draw[black,fill=white] (1,0.5) circle (.5ex); \draw[black,fill=white] (1.5,1) circle (.5ex);
\end{tikzpicture} } 
\text{ or }
\resizebox{1.5cm}{!}{
\begin{tikzpicture}[baseline=(current bounding box.center)] 
\draw[fill=violet] (0,1) rectangle (2,2); \draw (0,0) grid (2,2);
\draw[green] (0.5,0.5) -- (1.5,0.5) -- (1.5,1.5) -- (1.5,2);
\draw[black,fill=white] (1.5,0) circle (.5ex); \draw[black,fill=white] (2,0.5) circle (.5ex); \draw[black,fill=white] (2,1.5) circle (.5ex); \draw[black,fill=white] (0.5,1) circle (.5ex); \draw[black,fill=white] (1,1.5) circle (.5ex);
\end{tikzpicture} }. 
\]
Let $c$ be the rightmost column for which there exists a color of this form in columns $c$ and $c+1$, and let $i$ be the largest color of this form in columns $c$ and $c+1$.  We define $\varphi_{\bm{\alpha}}(C)$ to be the result of flipping color $i$ in columns $c$ and $c+1$
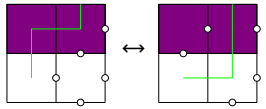
\[ 
\resizebox{1.5cm}{!}{
\begin{tikzpicture}[baseline=(current bounding box.center)] 
\draw[fill=violet] (0,1) rectangle (2,2); \draw (0,0) grid (2,2);
\draw[green] (0.5,0.5) -- (0.5,1.5) -- (1.5,1.5) -- (1.5,2);
\draw[black,fill=white] (1.5,0) circle (.5ex); \draw[black,fill=white] (2,0.5) circle (.5ex); \draw[black,fill=white] (2,1.5) circle (.5ex); \draw[black,fill=white] (1,0.5) circle (.5ex); \draw[black,fill=white] (1.5,1) circle (.5ex);
\end{tikzpicture} } 
\leftrightarrow
\resizebox{1.5cm}{!}{
\begin{tikzpicture}[baseline=(current bounding box.center)] 
\draw[fill=violet] (0,1) rectangle (2,2); \draw (0,0) grid (2,2);
\draw[green] (0.5,0.5) -- (1.5,0.5) -- (1.5,1.5) -- (1.5,2);
\draw[black,fill=white] (1.5,0) circle (.5ex); \draw[black,fill=white] (2,0.5) circle (.5ex); \draw[black,fill=white] (2,1.5) circle (.5ex); \draw[black,fill=white] (0.5,1) circle (.5ex); \draw[black,fill=white] (1,1.5) circle (.5ex);
\end{tikzpicture} }. 
\]
Clearly $\varphi_{\bm{\alpha}}$ is an involution.  To show $\weight(\varphi_{\bm{\alpha}}(C)) = -\weight(C)$, we need to show 
\begin{equation} \label{eq:2by2weight}
\weight \left ( \resizebox{1.5cm}{!}{
\begin{tikzpicture}[baseline=(current bounding box.center)] 
\draw[fill=violet] (0,1) rectangle (2,2); \draw (0,0) grid (2,2);
\draw[green] (0.5,0.5) -- (0.5,1.5) -- (1.5,1.5) -- (1.5,2);
\draw[black,fill=white] (1.5,0) circle (.5ex); \draw[black,fill=white] (2,0.5) circle (.5ex); \draw[black,fill=white] (2,1.5) circle (.5ex); \draw[black,fill=white] (1,0.5) circle (.5ex); \draw[black,fill=white] (1.5,1) circle (.5ex);
\node[left] at (0,0.5) {$r$}; \node[left] at (0,1.5) {$-r$};
\end{tikzpicture} } \right )
=
- \weight \left ( \resizebox{1.5cm}{!}{
\begin{tikzpicture}[baseline=(current bounding box.center)] 
\draw[fill=violet] (0,1) rectangle (2,2); \draw (0,0) grid (2,2);
\draw[green] (0.5,0.5) -- (1.5,0.5) -- (1.5,1.5) -- (1.5,2);
\draw[black,fill=white] (1.5,0) circle (.5ex); \draw[black,fill=white] (2,0.5) circle (.5ex); \draw[black,fill=white] (2,1.5) circle (.5ex); \draw[black,fill=white] (0.5,1) circle (.5ex); \draw[black,fill=white] (1,1.5) circle (.5ex);
\node[left] at (0,0.5) {$r$}; \node[left] at (0,1.5) {$-r$};
\end{tikzpicture} } \right ) 
\end{equation}
regardless of the paths taken by the other colors.  However, by the maximality of $c$ and $i$, we know that every color not equal to $i$ must have the form
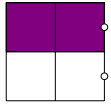
\[ 
\resizebox{1.5cm}{!}{
\begin{tikzpicture}[baseline=(current bounding box.center)] 
\draw[fill=violet] (0,1) rectangle (2,2); \draw (0,0) grid (2,2);
 \draw[black,fill=white] (2,0.5) circle (.5ex); \draw[black,fill=white] (2,1.5) circle (.5ex);
\end{tikzpicture} } 
\]
and every color greater than $i$ must not have the form
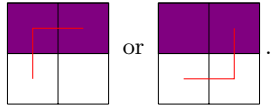
\[ 
\resizebox{1.5cm}{!}{
\begin{tikzpicture}[baseline=(current bounding box.center)] 
\draw[fill=violet] (0,1) rectangle (2,2); \draw (0,0) grid (2,2);
\draw[red] (0.5,0.5) -- (0.5,1.5) -- (1.5,1.5);
\end{tikzpicture} } 
\text{ or }
\resizebox{1.5cm}{!}{
\begin{tikzpicture}[baseline=(current bounding box.center)] 
\draw[fill=violet] (0,1) rectangle (2,2); \draw (0,0) grid (2,2);
\draw[red] (0.5,0.5) -- (1.5,0.5) -- (1.5,1.5);
\end{tikzpicture} }. 
\]
With these constraints, some straightforward computations (which we verified with a computer) show Equation \ref{eq:2by2weight} holds.
\end{proof}

\indent Combining Lemmas \ref{lem:permuterows} and \ref{lem:cancellation}, we can now conclude that the polynomials $\mc{L}_\lm^S(X;Y;t)$ are supersymmetric in the $X$ and $Y$ variables.

\begin{definition} \label{def:supersymmetric}
A family of polynomials $\{ p(X_n;Y_m) : n,m \in \mathbb{Z}_{\geq 0} \}$ is \textbf{supersymmetric} if 
\begin{itemize}
    \item $p(\sigma(X_n);Y_m) = p(X_n;Y_m;t)$ for any permutation $\sigma \in S_n$ \\
    (i.e. $p(X_n;Y_m)$ is symmetric in the $X$ variables),
    \item $p(X_n;\tau(Y_m)) = p(X_n;Y_m;t)$ for any permutation $\tau \in S_m$ \\
    (i.e. $p(X_n;Y_m)$ is symmetric in the $Y$ variables), and
    \item $p(X_{n-1},r;Y_{m-1},-r) = p(X_{n-1};Y_{m-1})$ when $n,m \geq 1$. 
\end{itemize}
\end{definition}

\begin{thm} \label{thm:supersymmetric}
The polynomials $\mc{L}_\lm^S(X_n;Y_m;t)$ are supersymmetric in the $X$ and $Y$ variables.  
\end{thm}

We proceed by proving a certain restriction property for the polynomials $\mc{L}_\lm^S(X_n;Y_m;t)$.

\begin{lem}[Restriction] \label{lem:restriction} We have
\[ \begin{aligned} 
&\mc{L}_\lm^S(X_{n-1},0;Y_m;t) = \mc{L}_\lm^S(X_{n-1};Y_m;t), \\
&\mc{L}_\lm^S(X_n;Y_{m-1},0;t) = \mc{L}_\lm^S(X_n;Y_{m-1};t). 
\end{aligned} \]
\end{lem}

\begin{proof}
Using Lemma \ref{lem:permuterows}, we can write
\[ \begin{aligned}
\mc{L}_\lm^S(X_{n-1},0;Y_m;t) 
= \resizebox{2cm}{!}{
\begin{tikzpicture}[baseline=(current bounding box.center)] 
\draw[fill=violet] (0,3) rectangle (3,6); \draw (0,0) grid (3,7);
\node[left] at (0,1.5) {$\vdots$}; \node[left] at (0,4.5) {$\vdots$}; \node[below] at (1.5,7) {$\ldots$};
\node[below] at (1.5,0) {$\bm{\mu}$}; \node[above] at (1.5,7) {$\bm{\lambda}$};
\node[left] at (0,0.5) {$x_1$}; \node[left] at (0,2.5) {$x_{n-1}$}; \node[left] at (0,3.5) {$y_1$}; \node[left] at (0,5.5) {$y_m$}; \node[left] at (0,6.5) {$0$}; 
\end{tikzpicture} }
= \sum_{\bm{\alpha}}
\resizebox{2cm}{!}{
\begin{tikzpicture}[baseline=(current bounding box.center)] 
\draw (0,0) grid (3,1);
\node[below] at (1.5,1) {$\ldots$};
\node[below] at (1.5,0) {$\bm{\alpha}$}; \node[above] at (1.5,1) {$\bm{\lambda}$};
\node[left] at (0,0.5) {$0$}; 
\end{tikzpicture} }
\resizebox{2cm}{!}{
\begin{tikzpicture}[baseline=(current bounding box.center)] 
\draw[fill=violet] (0,3) rectangle (3,6); \draw (0,0) grid (3,6);
\node[left] at (0,1.5) {$\vdots$}; \node[left] at (0,4.5) {$\vdots$}; \node[below] at (1.5,6) {$\ldots$};
\node[below] at (1.5,0) {$\bm{\mu}$}; \node[above] at (1.5,6) {$\bm{\alpha}$};
\node[left] at (0,0.5) {$x_1$}; \node[left] at (0,2.5) {$x_{n-1}$}; \node[left] at (0,3.5) {$y_1$}; \node[left] at (0,5.5) {$y_m$};
\end{tikzpicture} } 
\end{aligned} \]
It is easy to see that 
\[ \begin{aligned}
\resizebox{2cm}{!}{
\begin{tikzpicture}[baseline=(current bounding box.center)] 
\draw (0,0) grid (3,1);
\node[below] at (1.5,1) {$\ldots$};
\node[below] at (1.5,0) {$\bm{\alpha}$}; \node[above] at (1.5,1) {$\bm{\lambda}$};
\node[left] at (0,0.5) {$0$}; 
\end{tikzpicture} } = 1_{\bm{\lambda}=\bm{\alpha}}.
\end{aligned} \]
Therefore 
\begin{equation} \label{eq:restrictwhite}
\mc{L}_\lm^S(X_{n-1},0;Y_m;t) =
\resizebox{2cm}{!}{
\begin{tikzpicture}[baseline=(current bounding box.center)] 
\draw[fill=violet] (0,3) rectangle (3,6); \draw (0,0) grid (3,6);
\node[left] at (0,1.5) {$\vdots$}; \node[left] at (0,4.5) {$\vdots$}; \node[below] at (1.5,6) {$\ldots$};
\node[below] at (1.5,0) {$\bm{\mu}$}; \node[above] at (1.5,6) {$\bm{\lambda}$};
\node[left] at (0,0.5) {$x_1$}; \node[left] at (0,2.5) {$x_{n-1}$}; \node[left] at (0,3.5) {$y_1$}; \node[left] at (0,5.5) {$y_m$};
\end{tikzpicture} } = \mc{L}_\lm^S(X_{n-1};Y_m;t)
\end{equation}
A similar argument shows that 
\[ \mc{L}_\lm^S(X_n;Y_{m-1},0;t) = \mc{L}_\lm^S(X_n;Y_{m-1};t). \]  Alternatively, we can deduce
\[ \begin{aligned}
\mc{L}_\lm^S(X_n;Y_{m-1},0;t) &= t^{g(\bm{\lambda})-g(\bm{\mu})}\mc{L}_\lmp^S(Y_{m-1},0;X_n;t^{-1}) &\text{(Theorem \ref{thm:SSLLTxyswap})} \\
&= t^{g(\bm{\lambda})-g(\bm{\mu})}\mc{L}_\lmp^S(Y_{m-1};X_n;t^{-1}) &\text{(Equation \ref{eq:restrictwhite})} \\
&= t^{g(\bm{\lambda})-g(\bm{\mu})} (t^{-1})^{g(\bm{\lambda}')-g(\bm{\mu}')} \mc{L}_\lm^S(X_n;Y_{m-1};t) &\text{(Theorem \ref{thm:SSLLTxyswap})} \\
&= \mc{L}_\lm^S(X_n;Y_{m-1};t). &\text{(Corollary \ref{cor:dconjugate})} \\
\end{aligned} \] 
\end{proof}

\begin{lem} \label{lem:homogeneous}
The polynomial $\mc{L}_\lm^S(X_n;Y_m;t)$ is homogeneous in the $X$ and $Y$ variables of degree $|\lm| = |\bm{\lambda}| - |\bm{\mu}|$ i.e.
\[ \mc{L}_\lm^S(rX_n;rY_m;t) = r^{|\lm|} \mc{L}_\lm^S(X_n;Y_m;t). \] 
\end{lem}

\begin{proof} This follows from the fact that, in any configuration of the lattice 
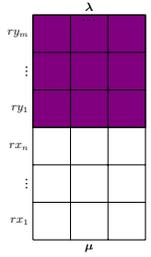
\[ 
\resizebox{2cm}{!}{
\begin{tikzpicture}[baseline=(current bounding box.center)] 
\draw[fill=violet] (0,3) rectangle (3,6); \draw (0,0) grid (3,6);
\node[left] at (0,1.5) {$\vdots$}; \node[left] at (0,4.5) {$\vdots$}; \node[below] at (1.5,6) {$\ldots$};
\node[below] at (1.5,0) {$\bm{\mu}$}; \node[above] at (1.5,6) {$\bm{\lambda}$};
\node[left] at (0,0.5) {$rx_1$}; \node[left] at (0,2.5) {$rx_n$}; \node[left] at (0,3.5) {$ry_1$}; \node[left] at (0,5.5) {$ry_m$}; 
\end{tikzpicture} } ,
\]
the total number of right steps taken by the paths is $|\lm|$. 
\end{proof}

\subsection{The Factorization Property}

The goal of this subsection is to prove the following lemma.

\begin{lem}[Factorization] \label{lem:factorization}
Fix $\bm{\lambda} \in P_p^{(k)}$.  Suppose there exist $\bm{\tau}$ and $\bm{\eta}$ such that, for all $i$,
\[ \lambda^{(i)} = (m+\tau^{(i)}_1,\ldots,m+\tau^{(i)}_n,\eta^{(i)}_1,\ldots,\eta^{(i)}_s) \]
where $s = p-n$.  Then
\[ \begin{aligned}
\mc{L}_{\bm{\lambda}}^S(X_n;Y_m;t) = \mc{L}_{\bm{\tau}}(X_n;t) 
\cdot t^{g(\bm{\eta})} \mc{L}_{\bm{\eta}'}(Y_m;t^{-1}) 
\cdot \prod_{l=0}^{k-1} \prod_{i=1}^n \prod_{j=1}^m (t^l x_i+y_j).
\end{aligned} \]
\end{lem}

Throughout this subsection, let $\bm{\lambda}$, $\bm{\tau}$, and $\bm{\eta}$ be as in the above lemma.  Moreover, it is easy to see that the above lemma holds if $n=0$ or $m=0$, so we will assume $n,m \geq 1$ throughout the rest of this subsection.  To prove the above lemma, we need two smaller lemmas.

\begin{lem} \label{lem:etaprime}
Let $\bm{\lambda}$, $\bm{\tau}$, and $\bm{\eta}$ be as in Lemma \ref{lem:factorization}.  The polynomial
\[ t^{g(\bm{\eta})} \mc{L}_{\bm{\eta}'}(Y_m;t^{-1}) = \mc{L}^P_{\bm{\eta}}(Y_m;t) \]
is a factor of the polynomial $\mc{L}_{\bm{\lambda}}^S(X_n;Y_m;t)$.  In fact, 
\[ 
\mc{L}_{\bm{\lambda}}^S(X_n;Y_m;t) 
= 
\mc{L}^P_{\bm{\eta}}(Y_m;t) \cdot \mc{L}^S_{m+\bm{\tau}}(X_n;Y_m;t)
\]
where $(m+\bm{\tau})^{(i)}_j = m + \bm{\tau}^{(i)}_j$ for all $i$ and $j$.
\end{lem}

\begin{lem} \label{lem:generalcancel} 
Let $\bm{\lambda}$, $\bm{\tau}$, and $\bm{\eta}$ be as in Lemma \ref{lem:factorization}.  Then
\[ \mc{L}_{\bm{\lambda}}^S(X_{n-1},r;Y_{m-1},-t^lr;t) =  0 \]
for all $l \in \{0,\ldots,k-1\}$.
\end{lem}

Given these two lemmas, let us prove Lemma \ref{lem:factorization}.

\begin{proof}[Proof of Lemma \ref{lem:factorization}]

Fix $l \in \{0,\ldots,k-1\}$.  Since  
\[ \mc{L}_{\bm{\lambda}}^S(X_{n-1},r;Y_{m-1},-t^lr;t) = 0 \]
by Lemma \ref{lem:generalcancel}, we know that $t^lx_n + y_m$ is a factor of $\mc{L}_{\bm{\lambda}}^S(X_n;Y_m;t)$.  Thus, since $\mc{L}_{\bm{\lambda}}^S(X_n;Y_m;t)$ is symmetric in the $X$ and $Y$ variables separately by Lemma \ref{lem:permuterows}, we know that 
\[ \prod_{i=1}^n \prod_{j=1}^m (t^l x_i+y_j) \]
is a factor of $\mc{L}_{\bm{\lambda}}^S(X_n;Y_m;t)$.  Since this holds for all $l \in \{0,\ldots,k-1\}$, we know that 
\[ \prod_{l=0}^{k-1} \prod_{i=1}^n \prod_{j=1}^m (t^l x_i+y_j) \]
is a factor of $\mc{L}_{\bm{\lambda}}^S(X_n;Y_m;t)$.  Moreover, since $t^{g(\bm{\eta})} \mc{L}_{\bm{\eta}'}(Y_m;t^{-1})$ is a factor of $\mc{L}_{\bm{\lambda}}^S(X_n;Y_m;t)$ by Lemma \ref{lem:etaprime}, we know that 
\[ t^{g(\bm{\eta})} \mc{L}_{\bm{\eta}'}(Y_m;t^{-1}) \cdot \prod_{l=0}^{k-1} \prod_{i=1}^n \prod_{j=1}^m (t^l x_i+y_j) \]
is a factor of $\mc{L}_{\bm{\lambda}}^S(X_n;Y_m;t)$.  Thus there is a polynomial $f(X_n;Y_m;t)$ such that
\begin{equation} \label{eq:lambdafactor}
\mc{L}_{\bm{\lambda}}^S(X_n;Y_m;t) = f(X_n;Y_m;t) \cdot t^{g(\bm{\eta})} \mc{L}_{\bm{\eta}'}(Y_m;t^{-1}) \cdot \prod_{l=0}^{k-1} \prod_{i=1}^n \prod_{j=1}^m (t^l x_i+y_j).
\end{equation}
It remains to show $f(X_n;Y_m;t)=\mc{L}_{\bm{\tau}}(X_n;t)$.  

Combining Equation \ref{eq:lambdafactor} with Lemma \ref{lem:etaprime}, we get
\begin{equation} 
\label{eq:mtaufactor} \mc{L}_{m+\bm{\tau}}^S(X_n;Y_m;t) = f(X_n;Y_m;t) \cdot \prod_{l=0}^{k-1} \prod_{i=1}^n \prod_{j=1}^m (t^l x_i+y_j) 
\end{equation}
Given a polynomial $p(X_n;Y_m;t)$, let $\mdy \left ( p(X_n;Y_m;t) \right )$ be the total degree of $p(X_n;Y_m;t)$ in the $Y$ variables.  Recall $\mc{L}^S_{m+\bm{\tau}}(X_n;Y_m;t)$ is the partition function associated to the lattice $S_{n,m}(m+\bm{\tau})$ of Def. \ref{def:latS}.
 
In any configuration of this lattice, a given path can take at most one step right in a purple row.  Thus, since there are $m$ purple rows and since there are $n$ paths of each of the $k$ colors, we have 
\[ \mdy \left ( \mc{L}^S_{m+\bm{\tau}}(X_n;Y_m;t) \right ) \leq mnk = \mdy \left ( \prod_{l=0}^{k-1} \prod_{i=1}^n \prod_{j=1}^m (t^l x_i+y_j) \right ). \]
This inequality along with Equation \ref{eq:mtaufactor} implies $f(X_n;Y_m;t) = h(X_n;t)$ for some polynomial $h(X_n;t)$.  It remains to show $h(X_n;t) = \mc{L}_{\bm{\tau}}(X_n;t)$.

We can rewrite Equation \ref{eq:mtaufactor} as
\[ 
\resizebox{2cm}{!}{
\begin{tikzpicture}[baseline=(current bounding box.center)] 
\draw[fill=violet] (0,3) rectangle (3,6); \draw (0,0) grid (3,6);
\node[left] at (0,1.5) {$\vdots$}; \node[left] at (0,4.5) {$\vdots$}; \node[below] at (1.5,6) {$\ldots$};
\node[below] at (1.5,0) {$\bm{\0}$}; \node[above] at (1.5,6) {$m+\bm{\tau}$};
\node[left] at (0,0.5) {$x_1$}; \node[left] at (0,2.5) {$x_n$}; \node[left] at (0,3.5) {$y_1$}; \node[left] at (0,5.5) {$y_m$}; 
\end{tikzpicture} }
= h(X_n;t) \cdot \prod_{l=0}^{k-1} \prod_{i=1}^n \prod_{j=1}^m (t^l x_i+y_j).
\]
We interpret both sides of this equation as polynomials in the $Y$ variables.  On the right-hand side, the coefficient of $y_1^{nk} \ldots y_m^{nk}$ is $h(X_n;t)$.  For the lattice on the left-hand side, a configuration has a weight of the form $p(X_n;t) \cdot y_1^{nk} \ldots y_m^{nk}$ for some polynomial $p(X_n;t)$ if and only if each path of each color takes one right step in each purple row.  Thus the $y_1^{nk} \ldots y_m^{nk}$ term on the left-hand side is exactly
\[ 
\begin{aligned}
\resizebox{4cm}{!}{
\begin{tikzpicture}[baseline=(current bounding box.center)] 
\draw[fill=violet] (0,3) rectangle (9,6); \draw (0,0) grid (9,6);
\node[left] at (0,1.5) {$\vdots$}; \node[left] at (0,4.5) {$\vdots$}; \node[below] at (4.5,6) {$\ldots$};
\node[below,scale=2] at (4.5,0) {$\bm{\0}$}; \node[above,scale=2] at (4.5,6) {$m+\bm{\tau}$}; \node[scale=2] at (4.5,3) {$\bm{\tau}$};
\node[left] at (0,0.5) {$x_1$}; \node[left] at (0,2.5) {$x_n$}; \node[left] at (0,3.5) {$y_1$}; \node[left] at (0,5.5) {$y_m$}; 
\draw[step=1.0,blue,line width = 1.5mm] (0.5,3) -- (0.5,3.5) -- (1.5,3.5) -- (1.5,4.5) -- (2.5,4.5) -- (2.5,5.5) -- (3.5,5.5) -- (3.5,6);
\draw[step=1.0,black,line width = 1.5mm] (2+0.5,3) -- (2+0.5,3.5) -- (2+1.5,3.5) -- (2+1.5,4.5) -- (2+2.5,4.5) -- (2+2.5,5.5) -- (2+3.5,5.5) -- (2+3.5,6);
\draw[step=1.0,red,line width = 1.5mm] (3+0.5,3) -- (3+0.5,3.5) -- (3+1.5,3.5) -- (3+1.5,4.5) -- (3+2.5,4.5) -- (3+2.5,5.5) -- (3+3.5,5.5) -- (3+3.5,6);
\draw[step=1.0,black,line width = 1.5mm] (5+0.5,3) -- (5+0.5,3.5) -- (5+1.5,3.5) -- (5+1.5,4.5) -- (5+2.5,4.5) -- (5+2.5,5.5) -- (5+3.5,5.5) -- (5+3.5,6);
\end{tikzpicture} }
= &
\resizebox{1.5cm}{!}{
\begin{tikzpicture}[baseline=(current bounding box.center)] 
\draw (0,0) grid (3,3);
\node[left] at (0,1.5) {$\vdots$}; \node[below] at (1.5,3) {$\ldots$};
\node[below,scale=2] at (1.5,0) {$\bm{\0}$}; \node[above,scale=2] at (1.5,3) {$\bm{\tau}$};
\node[left] at (0,0.5) {$x_1$}; \node[left] at (0,2.5) {$x_n$}; 
\end{tikzpicture} }
\cdot 
\resizebox{4cm}{!}{
\begin{tikzpicture}[baseline=(current bounding box.center)] 
\draw[fill=violet] (0,3) rectangle (9,6); \draw (0,3) grid (9,6);
\node[left] at (0,4.5) {$\vdots$}; \node[below] at (4.5,6) {$\ldots$};
\node[above,scale=2] at (4.5,6) {$m+\bm{\tau}$}; \node[below,scale=2] at (4.5,3) {$\bm{\tau}$};
\node[left] at (0,3.5) {$y_1$}; \node[left] at (0,5.5) {$y_m$}; 
\draw[step=1.0,blue,line width = 1.5mm] (0.5,3) -- (0.5,3.5) -- (1.5,3.5) -- (1.5,4.5) -- (2.5,4.5) -- (2.5,5.5) -- (3.5,5.5) -- (3.5,6);
\draw[step=1.0,black,line width = 1.5mm] (2+0.5,3) -- (2+0.5,3.5) -- (2+1.5,3.5) -- (2+1.5,4.5) -- (2+2.5,4.5) -- (2+2.5,5.5) -- (2+3.5,5.5) -- (2+3.5,6);
\draw[step=1.0,red,line width = 1.5mm] (3+0.5,3) -- (3+0.5,3.5) -- (3+1.5,3.5) -- (3+1.5,4.5) -- (3+2.5,4.5) -- (3+2.5,5.5) -- (3+3.5,5.5) -- (3+3.5,6);
\draw[step=1.0,black,line width = 1.5mm] (5+0.5,3) -- (5+0.5,3.5) -- (5+1.5,3.5) -- (5+1.5,4.5) -- (5+2.5,4.5) -- (5+2.5,5.5) -- (5+3.5,5.5) -- (5+3.5,6);
\end{tikzpicture} } \\
=& \mc{L}_{\bm{\tau}}(X_n;t) \cdot y_1^{nk} \ldots y_m^{nk}.
\end{aligned}
\]
(Here, each path takes one right step in each purple row.)  Thus $h(X_n;t) = \mc{L}_{\bm{\tau}}(X_n;t)$.
\end{proof}

We are left to prove Lemmas \ref{lem:etaprime} and \ref{lem:generalcancel}.

\begin{proof}[Proof of Lemma \ref{lem:etaprime}]

Using Theorem \ref{thm:SSLLTxyswap} and the fact that $g(\bm{\0}) = 0$, we have
\[ \mc{L}^P_{\bm{\eta}}(Y_m;t) = \mc{L}^S_{\bm{\eta}}(;Y_m;t) = t^{g(\bm{\eta})} \mc{L}^S_{\bm{\eta}'}(Y_m;;t^{-1}) = t^{g(\bm{\eta})} \mc{L}_{\bm{\eta}'}(Y_m;t^{-1}). \]
Recall $\mc{L}_{\bm{\lambda}}^S(X_n;Y_m;t)$ is the partition function associated to the lattice
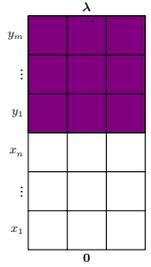
\[ 
\resizebox{2cm}{!}{
\begin{tikzpicture}[baseline=(current bounding box.center)] 
\draw[fill=violet] (0,3) rectangle (3,6); \draw (0,0) grid (3,6);
\node[left] at (0,1.5) {$\vdots$}; \node[left] at (0,4.5) {$\vdots$}; \node[below] at (1.5,6) {$\ldots$};
\node[below] at (1.5,0) {$\bm{\0}$}; \node[above] at (1.5,6) {$\bm{\lambda}$};
\node[left] at (0,0.5) {$x_1$}; \node[left] at (0,2.5) {$x_n$}; \node[left] at (0,3.5) {$y_1$}; \node[left] at (0,5.5) {$y_m$}; 
\end{tikzpicture} }.
\]
Any configuration of this lattice must have the form
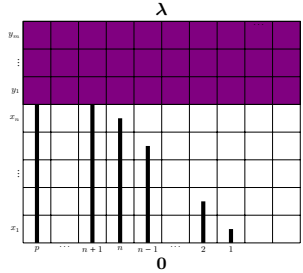
\[ 
\resizebox{4cm}{!}{
\begin{tikzpicture}[baseline=(current bounding box.center)] 
\draw[fill=violet] (0,5) rectangle (10,8); \draw (0,0) grid (10,8);
\node[left] at (0,2.5) {$\vdots$}; \node[left] at (0,6.5) {$\vdots$}; \node[below] at (8.5,8) {$\ldots$};
\node[below,scale=2] at (5,-0.3) {$\bm{\0}$}; \node[above,scale=2] at (5,8) {$\bm{\lambda}$};
\node[left] at (0,0.5) {$x_1$}; \node[left] at (0,4.5) {$x_n$}; \node[left] at (0,5.5) {$y_1$}; \node[left] at (0,7.5) {$y_m$}; 
\node[below] at (0.5,0) {$p$}; \node[below] at (1.5,0) {$\ldots$}; \node[below] at (2.5,0) {$n+1$}; \node[below] at (3.5,0) {$n$}; \node[below] at (4.5,0) {$n-1$}; \node[below] at (5.5,0) {$\ldots$}; \node[below] at (6.5,0) {$2$}; \node[below] at (7.5,0) {$1$};
\draw[step=1.0,black,line width = 1.5mm] (0.5,0) -- (0.5,5); \draw[step=1.0,black,line width = 1.5mm] (2.5,0) -- (2.5,5); \draw[step=1.0,black,line width = 1.5mm] (3.5,0) -- (3.5,4.5);
\draw[step=1.0,black,line width = 1.5mm] (4.5,0) -- (4.5,3.5);
\draw[step=1.0,black,line width = 1.5mm] (6.5,0) -- (6.5,1.5);
\draw[step=1.0,black,line width = 1.5mm] (7.5,0) -- (7.5,0.5);
\end{tikzpicture} }
\]
where we have labelled the columns for convenience.  Fix $j \in \{1,\ldots,m\}$.  Note that a path can take at most one step right in a given purple row.  Since $\lambda^{(i)}_n = m + \tau_n^{(i)} \geq m$ for all $i$, the paths starting in column $n$ must exit the $m$-th purple row weakly right of column $n-m$, so the paths starting in column $n$ must exit the $j$-th purple row weakly right of column $n-j$.  Since the paths starting in column $n+1$ enter the first purple row in column $n+1$, they must exit the $j$-th purple row weakly left of column $n+1-j$.  This argument shows that the remainder of the paths starting in columns $n+1,\ldots,p$ and the remainder of the paths starting in columns $1,\ldots,n$ can be chosen independently of each other, and that the weight of the overall configuration is the weight of the configuration consisting of the paths starting in columns $n+1,\ldots,p$ times the weight of the configuration consisting of the paths starting in columns $1,\ldots,n$.

It follows that
\[ \begin{aligned}
\mc{L}_{\bm{\lambda}}^S(X_n;Y_m;t) 
= 
\resizebox{4cm}{!}{
\begin{tikzpicture}[baseline=(current bounding box.center)] 
\draw[fill=violet] (0,5) rectangle (10,8); \draw (0,0) grid (10,8);
\node[left] at (0,2.5) {$\vdots$}; \node[left] at (0,6.5) {$\vdots$}; \node[below] at (8.5,8) {$\ldots$};
\node[below,scale=2] at (5,-0.3) {$\bm{\0}$}; \node[above,scale=2] at (5,8) {$\bm{\lambda}_{\{n+1,\ldots,p\}}$};
\node[left] at (0,0.5) {$x_1$}; \node[left] at (0,4.5) {$x_n$}; \node[left] at (0,5.5) {$y_1$}; \node[left] at (0,7.5) {$y_m$}; 
\node[below] at (0.5,0) {$p$}; \node[below] at (1.5,0) {$\ldots$}; \node[below] at (2.5,0) {$n+1$}; \node[below] at (3.5,0) {$n$}; \node[below] at (4.5,0) {$n-1$}; \node[below] at (5.5,0) {$\ldots$}; \node[below] at (6.5,0) {$2$}; \node[below] at (7.5,0) {$1$};
\draw[step=1.0,black,line width = 1.5mm] (0.5,0) -- (0.5,5); 
\draw[step=1.0,black,line width = 1.5mm] (2.5,0) -- (2.5,5); 
\end{tikzpicture} }
\cdot 
\resizebox{4cm}{!}{
\begin{tikzpicture}[baseline=(current bounding box.center)] 
\draw[fill=violet] (0,5) rectangle (10,8); \draw (0,0) grid (10,8);
\node[left] at (0,2.5) {$\vdots$}; \node[left] at (0,6.5) {$\vdots$}; \node[below] at (8.5,8) {$\ldots$};
\node[below,scale=2] at (5,-0.3) {$\bm{\0}$}; \node[above,scale=2] at (5,8) {$\bm{\lambda}_{\{1,\ldots,n\}}$};
\node[left] at (0,0.5) {$x_1$}; \node[left] at (0,4.5) {$x_n$}; \node[left] at (0,5.5) {$y_1$}; \node[left] at (0,7.5) {$y_m$}; 
\node[below] at (0.5,0) {$p$}; \node[below] at (1.5,0) {$\ldots$}; \node[below] at (2.5,0) {$n+1$}; \node[below] at (3.5,0) {$n$}; \node[below] at (4.5,0) {$n-1$}; \node[below] at (5.5,0) {$\ldots$}; \node[below] at (6.5,0) {$2$}; \node[below] at (7.5,0) {$1$};
\draw[step=1.0,black,line width = 1.5mm] (3.5,0) -- (3.5,4.5);
\draw[step=1.0,black,line width = 1.5mm] (4.5,0) -- (4.5,3.5);
\draw[step=1.0,black,line width = 1.5mm] (6.5,0) -- (6.5,1.5);
\draw[step=1.0,black,line width = 1.5mm] (7.5,0) -- (7.5,0.5);
\end{tikzpicture} }
\end{aligned} \]
where, for a set $S = \{s_1 < \ldots < s_j\} \subseteq \{1,\ldots,p\}$, we define 
\[ \bm{\lambda}_S = (\lambda^{(1)}_S,\ldots,\lambda^{(k)}_S)  \text{ with } \lambda^{(i)}_S = (\lambda^{(k)}_{s_1},\ldots,\lambda^{(k)}_{s_j}). \]
Since $\bm{\lambda}_{\{n+1,\ldots,p\}} = \bm{\eta}$, the first factor is exactly
\[ \begin{aligned}
\resizebox{4cm}{!}{
\begin{tikzpicture}[baseline=(current bounding box.center)] 
\draw[fill=violet] (0,5) rectangle (10,8); \draw (0,0) grid (10,8);
\node[left] at (0,2.5) {$\vdots$}; \node[left] at (0,6.5) {$\vdots$}; \node[below] at (8.5,8) {$\ldots$};
\node[below,scale=2] at (5,-0.3) {$\bm{\0}$}; \node[above,scale=2] at (5,8) {$\bm{\eta}$};
\node[left] at (0,0.5) {$x_1$}; \node[left] at (0,4.5) {$x_n$}; \node[left] at (0,5.5) {$y_1$}; \node[left] at (0,7.5) {$y_m$}; 
\node[below] at (0.5,0) {$p$}; \node[below] at (1.5,0) {$\ldots$}; \node[below] at (2.5,0) {$n+1$}; \node[below] at (3.5,0) {$n$}; \node[below] at (4.5,0) {$n-1$}; \node[below] at (5.5,0) {$\ldots$}; \node[below] at (6.5,0) {$2$}; \node[below] at (7.5,0) {$1$};
\draw[step=1.0,black,line width = 1.5mm] (0.5,0) -- (0.5,5); 
\draw[step=1.0,black,line width = 1.5mm] (2.5,0) -- (2.5,5); 
\end{tikzpicture} }
=
\resizebox{2cm}{!}{
\begin{tikzpicture}[baseline=(current bounding box.center)] 
\draw[fill=violet] (0,0) rectangle (3,3); \draw (0,0) grid (3,3);
\node[left] at (0,1.5) {$\vdots$}; \node[below] at (1.5,3) {$\ldots$};
\node[below,scale=2] at (1.5,0) {$\bm{\0}$}; \node[above,scale=2] at (1.5,3) {$\bm{\eta}$};
\node[left] at (0,0.5) {$y_1$}; \node[left] at (0,2.5) {$y_m$}; 
\end{tikzpicture} }
= \mc{L}^P_{\bm{\eta}}(Y_m;t).
\end{aligned} \]
Since $\bm{\lambda}_{\{1,\ldots,n\}} = m + \bm{\tau}$, the second factor is exactly
\[ \begin{aligned}
\resizebox{4cm}{!}{
\begin{tikzpicture}[baseline=(current bounding box.center)] 
\draw[fill=violet] (0,5) rectangle (10,8); \draw (0,0) grid (10,8);
\node[left] at (0,2.5) {$\vdots$}; \node[left] at (0,6.5) {$\vdots$}; \node[below] at (8.5,8) {$\ldots$};
\node[below,scale=2] at (5,-0.3) {$\bm{\0}$}; \node[above,scale=2] at (5,8) {$m+\bm{\tau}$};
\node[left] at (0,0.5) {$x_1$}; \node[left] at (0,4.5) {$x_n$}; \node[left] at (0,5.5) {$y_1$}; \node[left] at (0,7.5) {$y_m$}; 
\node[below] at (0.5,0) {$p$}; \node[below] at (1.5,0) {$\ldots$}; \node[below] at (2.5,0) {$n+1$}; \node[below] at (3.5,0) {$n$}; \node[below] at (4.5,0) {$n-1$}; \node[below] at (5.5,0) {$\ldots$}; \node[below] at (6.5,0) {$2$}; \node[below] at (7.5,0) {$1$};
\draw[step=1.0,black,line width = 1.5mm] (3.5,0) -- (3.5,4.5);
\draw[step=1.0,black,line width = 1.5mm] (4.5,0) -- (4.5,3.5);
\draw[step=1.0,black,line width = 1.5mm] (6.5,0) -- (6.5,1.5);
\draw[step=1.0,black,line width = 1.5mm] (7.5,0) -- (7.5,0.5);
\end{tikzpicture} }
=
\resizebox{2cm}{!}{
\begin{tikzpicture}[baseline=(current bounding box.center)] 
\draw[fill=violet] (0,3) rectangle (3,6); \draw (0,0) grid (3,6);
\node[left] at (0,1.5) {$\vdots$}; \node[left] at (0,4.5) {$\vdots$}; \node[below] at (1.5,6) {$\ldots$};
\node[below] at (1.5,0) {$\bm{\0}$}; \node[above] at (1.5,6) {$m+\bm{\tau}$};
\node[left] at (0,0.5) {$x_1$}; \node[left] at (0,2.5) {$x_n$}; \node[left] at (0,3.5) {$y_1$}; \node[left] at (0,5.5) {$y_m$}; 
\end{tikzpicture} }
= \mc{L}^S_{m+\bm{\tau}}(X_n;Y_m;t).
\end{aligned} \]
\end{proof}

\begin{proof}[Proof of Lemma \ref{lem:generalcancel}]

By Lemma \ref{lem:permuterows}, $\mc{L}_{\bm{\lambda}}^S(X_{n-1},r;Y_{m-1},-t^lr;t)$ is the partition function associated to the lattice
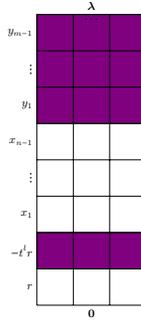
\[ 
\resizebox{2cm}{!}{
\begin{tikzpicture}[baseline=(current bounding box.center)] 
\draw[fill=violet] (0,5) rectangle (3,8); \draw[fill=violet] (0,1) rectangle (3,2); \draw (0,0) grid (3,8);
\node[left] at (0,3.5) {$\vdots$}; \node[left] at (0,6.5) {$\vdots$}; \node[below] at (1.5,8) {$\ldots$};
\node[below] at (1.5,0) {$\bm{\0}$}; \node[above] at (1.5,8) {$\bm{\lambda}$};
\node[left] at (0,2.5) {$x_1$}; \node[left] at (0,4.5) {$x_{n-1}$}; \node[left] at (0,5.5) {$y_1$}; \node[left] at (0,7.5) {$y_{m-1}$}; \node[left] at (0,0.5) {$r$}; \node[left] at (0,1.5) {$-t^lr$}; 
\end{tikzpicture} }
\] 
Any configuration of this lattice must have the form
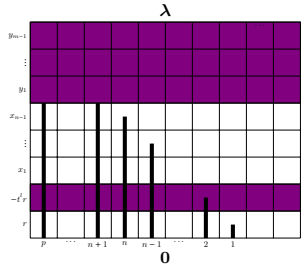
\[ 
\resizebox{4cm}{!}{
\begin{tikzpicture}[baseline=(current bounding box.center)] 
\draw[fill=violet] (0,5) rectangle (10,8); \draw[fill=violet] (0,1) rectangle (10,2); \draw (0,0) grid (10,8);
\node[left] at (0,3.5) {$\vdots$}; \node[left] at (0,6.5) {$\vdots$}; \node[below] at (8.5,8) {$\ldots$};
\node[below,scale=2] at (5,-0.3) {$\bm{\0}$}; \node[above,scale=2] at (5,8) {$\bm{\lambda}$};
\node[left] at (0,2.5) {$x_1$}; \node[left] at (0,4.5) {$x_{n-1}$}; \node[left] at (0,5.5) {$y_1$}; \node[left] at (0,7.5) {$y_{m-1}$}; \node[left] at (0,0.5) {$r$}; \node[left] at (0,1.5) {$-t^lr$}; 
\node[below] at (0.5,0) {$p$}; \node[below] at (1.5,0) {$\ldots$}; \node[below] at (2.5,0) {$n+1$}; \node[below] at (3.5,0) {$n$}; \node[below] at (4.5,0) {$n-1$}; \node[below] at (5.5,0) {$\ldots$}; \node[below] at (6.5,0) {$2$}; \node[below] at (7.5,0) {$1$};
\draw[step=1.0,black,line width = 1.5mm] (0.5,0) -- (0.5,5); \draw[step=1.0,black,line width = 1.5mm] (2.5,0) -- (2.5,5); \draw[step=1.0,black,line width = 1.5mm] (3.5,0) -- (3.5,4.5);
\draw[step=1.0,black,line width = 1.5mm] (4.5,0) -- (4.5,3.5);
\draw[step=1.0,black,line width = 1.5mm] (6.5,0) -- (6.5,1.5);
\draw[step=1.0,black,line width = 1.5mm] (7.5,0) -- (7.5,0.5);
\end{tikzpicture} }
\]
where we have labelled the columns for convenience.  In a configuration, if there exists a color $i$ such that the path of color $i$ starting in column 1 goes vertically in the bottom two rows, then color $i$ must have the form
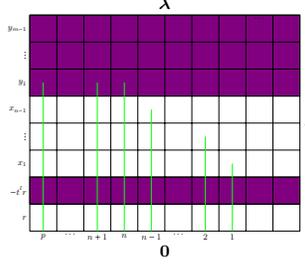
\[ 
\resizebox{4cm}{!}{
\begin{tikzpicture}[baseline=(current bounding box.center)] 
\draw[fill=violet] (0,5) rectangle (10,8); \draw[fill=violet] (0,1) rectangle (10,2); \draw (0,0) grid (10,8);
\node[left] at (0,3.5) {$\vdots$}; \node[left] at (0,6.5) {$\vdots$}; \node[below] at (8.5,8) {$\ldots$};
\node[below,scale=2] at (5,-0.3) {$\bm{\0}$}; \node[above,scale=2] at (5,8) {$\bm{\lambda}$};
\node[left] at (0,2.5) {$x_1$}; \node[left] at (0,4.5) {$x_{n-1}$}; \node[left] at (0,5.5) {$y_1$}; \node[left] at (0,7.5) {$y_{m-1}$}; \node[left] at (0,0.5) {$r$}; \node[left] at (0,1.5) {$-t^lr$}; 
\node[below] at (0.5,0) {$p$}; \node[below] at (1.5,0) {$\ldots$}; \node[below] at (2.5,0) {$n+1$}; \node[below] at (3.5,0) {$n$}; \node[below] at (4.5,0) {$n-1$}; \node[below] at (5.5,0) {$\ldots$}; \node[below] at (6.5,0) {$2$}; \node[below] at (7.5,0) {$1$};
\draw[green,thick] (0.5,0) -- (0.5,5.5); \draw[green,thick] (2.5,0) -- (2.5,5.5); \draw[green,thick] (3.5,0) -- (3.5,5.5);
\draw[green,thick] (4.5,0) -- (4.5,4.5);
\draw[green,thick] (6.5,0) -- (6.5,3.5);
\draw[green,thick] (7.5,0) -- (7.5,2.5);
\end{tikzpicture} }.
\]
A path can take at most one step right in a given purple row, so the path of color $i$ starting in column $n$ can make at most $m-1$ total steps right in the lattice.  However, since $\lambda^{(i)}_n = m + \tau_n^{(i)} \geq m$, the path of color $i$ starting in column $n$ must make at least $m$ total steps right in the lattice.  This is a contradiction, which means that in any configuration, every path starting in column 1 must make at least one step right somewhere in the bottom two rows.  Therefore
\[ \begin{aligned}
\mc{L}_\lm^S(X_{n-1},r;Y_{m-1},-t^lr;t) 
= \resizebox{2cm}{!}{
\begin{tikzpicture}[baseline=(current bounding box.center)] 
\draw[fill=violet] (0,5) rectangle (3,8); \draw[fill=violet] (0,1) rectangle (3,2); \draw (0,0) grid (3,8);
\node[left] at (0,3.5) {$\vdots$}; \node[left] at (0,6.5) {$\vdots$}; \node[below] at (1.5,8) {$\ldots$};
\node[below] at (1.5,0) {$\bm{\0}$}; \node[above] at (1.5,8) {$\bm{\lambda}$};
\node[left] at (0,2.5) {$x_1$}; \node[left] at (0,4.5) {$x_{n-1}$}; \node[left] at (0,5.5) {$y_1$}; \node[left] at (0,7.5) {$y_{m-1}$}; \node[left] at (0,0.5) {$r$}; \node[left] at (0,1.5) {$-t^lr$}; 
\end{tikzpicture} }
= \sum_{\bm{\alpha}}
\resizebox{2cm}{!}{
\begin{tikzpicture}[baseline=(current bounding box.center)] 
\draw[fill=violet] (0,1) rectangle (3,2); \draw (0,0) grid (3,2);
\node[below] at (1.5,2) {$\ldots$};
\node[below] at (1.5,0) {$\bm{\0}$}; \node[above] at (1.5,2) {$\bm{\alpha}$};
\node[left] at (0,0.5) {$r$}; \node[left] at (0,1.5) {$-t^lr$}; 
\end{tikzpicture} }
\resizebox{2cm}{!}{
\begin{tikzpicture}[baseline=(current bounding box.center)] 
\draw[fill=violet] (0,3) rectangle (3,6); \draw (0,0) grid (3,6);
\node[left] at (0,1.5) {$\vdots$}; \node[left] at (0,4.5) {$\vdots$}; \node[below] at (1.5,6) {$\ldots$};
\node[below] at (1.5,0) {$\bm{\alpha}$}; \node[above] at (1.5,6) {$\bm{\lambda}$};
\node[left] at (0,0.5) {$x_1$}; \node[left] at (0,2.5) {$x_{n-1}$}; \node[left] at (0,3.5) {$y_1$}; \node[left] at (0,5.5) {$y_{m-1}$};
\end{tikzpicture} } 
\end{aligned} \]
where the sum is over all $\bm{\alpha}$ such that $\alpha^{(i)}_1 > 0$ for all $i$.  We will show that, for all such $\bm{\alpha}$, there is an involution $\varphi_{\bm{\alpha}}$ on the set of configurations of the lattice
\[ 
L_{\bm{\alpha}} = \resizebox{2cm}{!}{
\begin{tikzpicture}[baseline=(current bounding box.center)] 
\draw[fill=violet] (0,1) rectangle (3,2); \draw (0,0) grid (3,2);
\node[below] at (1.5,2) {$\ldots$};
\node[below] at (1.5,0) {$\bm{\0}$}; \node[above] at (1.5,2) {$\bm{\alpha}$};
\node[left] at (0,0.5) {$r$}; \node[left] at (0,1.5) {$-t^lr$}; 
\end{tikzpicture} }
\]
such that $\weight(\varphi_{\bm{\alpha}}(C)) = -\weight(C)$ for all $C$.  Therefore
\[ \mc{L}_\lm^S(X_{n-1},r;Y_{m-1},-r;t) = 0. \]

Fix $\bm{\alpha}$ with $\alpha^{(i)}_1 > 0$ for all $i$ and fix a configuration $C$ on the lattice $L_{\bm{\alpha}}$.  We label the columns for convenience as follows.
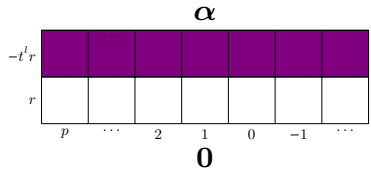
\[ 
\resizebox{5cm}{!}{
\begin{tikzpicture}[baseline=(current bounding box.center)] 
\draw[fill=violet] (0,1) rectangle (7,2); \draw (0,0) grid (7,2);
\node[below] at (1.5,2) {$\ldots$};
\node[below] at (0.5,0) {$p$}; \node[below] at (1.5,0) {$\ldots$}; \node[below] at (2.5,0) {$2$}; \node[below] at (3.5,0) {$1$}; \node[below] at (4.5,0) {$0$}; \node[below] at (5.5,0) {$-1$}; \node[below] at (6.5,0) {$\ldots$};
\node[below,scale=2] at (3.5,-0.3) {$\bm{\0}$}; \node[above,scale=2] at (3.5,2) {$\bm{\alpha}$};
\node[left] at (0,0.5) {$r$}; \node[left] at (0,1.5) {$-t^lr$}; 
\end{tikzpicture} } 
\]
Given a color $i$, let $c_i$ be the column in which the path of color $i$ starting in column 1 exits the lattice through the top.  Since every path starting in column 1 must make at least one step right, we have $c_i \leq 0$ for all $i$. We define an ordering $\prec$ on the colors by
\[ i \prec j \Leftrightarrow c_i > c_j \text{ OR } c_i = c_j \text{ and } i < j. \]
Let $i$ be the $(l+1)$-th largest color in this ordering (so that $l = \#\{ j : i \prec j \}$).  In columns $c_i+1$ and $c_i$ of $C$, color $i$ has the form
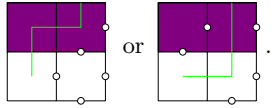
\[ 
\resizebox{1.5cm}{!}{
\begin{tikzpicture}[baseline=(current bounding box.center)] 
\draw[fill=violet] (0,1) rectangle (2,2); \draw (0,0) grid (2,2);
\draw[green] (0.5,0.5) -- (0.5,1.5) -- (1.5,1.5) -- (1.5,2);
\draw[black,fill=white] (1.5,0) circle (.5ex); \draw[black,fill=white] (2,0.5) circle (.5ex); \draw[black,fill=white] (2,1.5) circle (.5ex); \draw[black,fill=white] (1,0.5) circle (.5ex); \draw[black,fill=white] (1.5,1) circle (.5ex);
\end{tikzpicture} } 
\text{ or }
\resizebox{1.5cm}{!}{
\begin{tikzpicture}[baseline=(current bounding box.center)] 
\draw[fill=violet] (0,1) rectangle (2,2); \draw (0,0) grid (2,2);
\draw[green] (0.5,0.5) -- (1.5,0.5) -- (1.5,1.5) -- (1.5,2);
\draw[black,fill=white] (1.5,0) circle (.5ex); \draw[black,fill=white] (2,0.5) circle (.5ex); \draw[black,fill=white] (2,1.5) circle (.5ex); \draw[black,fill=white] (0.5,1) circle (.5ex); \draw[black,fill=white] (1,1.5) circle (.5ex);
\end{tikzpicture} }. 
\]
We define $\varphi_{\bm{\alpha}}(C)$ to be the result of flipping color $i$ in columns $c_i+1$ and $c_i$
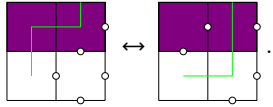
\[ 
\resizebox{1.5cm}{!}{
\begin{tikzpicture}[baseline=(current bounding box.center)] 
\draw[fill=violet] (0,1) rectangle (2,2); \draw (0,0) grid (2,2);
\draw[green] (0.5,0.5) -- (0.5,1.5) -- (1.5,1.5) -- (1.5,2);
\draw[black,fill=white] (1.5,0) circle (.5ex); \draw[black,fill=white] (2,0.5) circle (.5ex); \draw[black,fill=white] (2,1.5) circle (.5ex); \draw[black,fill=white] (1,0.5) circle (.5ex); \draw[black,fill=white] (1.5,1) circle (.5ex);
\end{tikzpicture} } 
\leftrightarrow
\resizebox{1.5cm}{!}{
\begin{tikzpicture}[baseline=(current bounding box.center)] 
\draw[fill=violet] (0,1) rectangle (2,2); \draw (0,0) grid (2,2);
\draw[green] (0.5,0.5) -- (1.5,0.5) -- (1.5,1.5) -- (1.5,2);
\draw[black,fill=white] (1.5,0) circle (.5ex); \draw[black,fill=white] (2,0.5) circle (.5ex); \draw[black,fill=white] (2,1.5) circle (.5ex); \draw[black,fill=white] (0.5,1) circle (.5ex); \draw[black,fill=white] (1,1.5) circle (.5ex);
\end{tikzpicture} }. 
\]
Clearly $\varphi_{\bm{\alpha}}$ is an involution.  To show $\weight(\varphi_{\bm{\alpha}}(C)) = -\weight(C)$, we need to show 
\begin{equation} \label{eq:2by2weightL}
\weight \left ( \resizebox{2cm}{!}{
\begin{tikzpicture}[baseline=(current bounding box.center)] 
\draw[fill=violet] (0,1) rectangle (2,2); \draw (0,0) grid (2,2);
\draw[green] (0.5,0.5) -- (0.5,1.5) -- (1.5,1.5) -- (1.5,2);
\draw[black,fill=white] (1.5,0) circle (.5ex); \draw[black,fill=white] (2,0.5) circle (.5ex); \draw[black,fill=white] (2,1.5) circle (.5ex); \draw[black,fill=white] (1,0.5) circle (.5ex); \draw[black,fill=white] (1.5,1) circle (.5ex);
\node[left] at (0,0.5) {$r$}; \node[left] at (0,1.5) {$-t^lr$};
\end{tikzpicture} } \right )
=
- \weight \left ( \resizebox{2cm}{!}{
\begin{tikzpicture}[baseline=(current bounding box.center)] 
\draw[fill=violet] (0,1) rectangle (2,2); \draw (0,0) grid (2,2);
\draw[green] (0.5,0.5) -- (1.5,0.5) -- (1.5,1.5) -- (1.5,2);
\draw[black,fill=white] (1.5,0) circle (.5ex); \draw[black,fill=white] (2,0.5) circle (.5ex); \draw[black,fill=white] (2,1.5) circle (.5ex); \draw[black,fill=white] (0.5,1) circle (.5ex); \draw[black,fill=white] (1,1.5) circle (.5ex);
\node[left] at (0,0.5) {$r$}; \node[left] at (0,1.5) {$-t^lr$};
\end{tikzpicture} } \right ) 
\end{equation}
regardless of the paths taken by the other colors.  We label the boxes for convenience as follows.
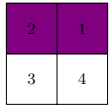
\[ 
\resizebox{1.5cm}{!}{
\begin{tikzpicture}[baseline=(current bounding box.center)] 
\draw[fill=violet] (0,1) rectangle (2,2); \draw (0,0) grid (2,2);
\node at (1.5,1.5) {$1$}; \node at (0.5,1.5) {$2$}; \node at (0.5,0.5) {$3$}; \node at (1.5,0.5) {$4$};
\end{tikzpicture} }
\]
Compared to the configuration with color $i$ absent, the presence of color $i$ in the form
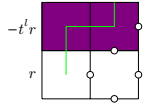
\[
\resizebox{2cm}{!}{
\begin{tikzpicture}[baseline=(current bounding box.center)] 
\draw[fill=violet] (0,1) rectangle (2,2); \draw (0,0) grid (2,2);
\draw[green] (0.5,0.5) -- (0.5,1.5) -- (1.5,1.5) -- (1.5,2);
\draw[black,fill=white] (1.5,0) circle (.5ex); \draw[black,fill=white] (2,0.5) circle (.5ex); \draw[black,fill=white] (2,1.5) circle (.5ex); \draw[black,fill=white] (1,0.5) circle (.5ex); \draw[black,fill=white] (1.5,1) circle (.5ex);
\node[left] at (0,0.5) {$r$}; \node[left] at (0,1.5) {$-t^lr$};
\end{tikzpicture} }
\]
contributes $-t^l r \cdot t^a$ to the overall weight, where
\[ a = \# \{ j > i : \text{$j$ is vertical in box 2} \} + \# \{ j < i : \text{$j$ exits right in box 3} \}. \]
Compared to the configuration with color $i$ absent, the presence of color $i$ in the form
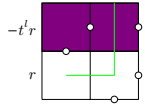
\[
\resizebox{2cm}{!}{
\begin{tikzpicture}[baseline=(current bounding box.center)] 
\draw[fill=violet] (0,1) rectangle (2,2); \draw (0,0) grid (2,2);
\draw[green] (0.5,0.5) -- (1.5,0.5) -- (1.5,1.5) -- (1.5,2);
\draw[black,fill=white] (1.5,0) circle (.5ex); \draw[black,fill=white] (2,0.5) circle (.5ex); \draw[black,fill=white] (2,1.5) circle (.5ex); \draw[black,fill=white] (0.5,1) circle (.5ex); \draw[black,fill=white] (1,1.5) circle (.5ex);
\node[left] at (0,0.5) {$r$}; \node[left] at (0,1.5) {$-t^lr$};
\end{tikzpicture} }
\]
contributes $r \cdot t^b$ to the overall weight, where
\[ \begin{aligned}
b &= \# \{ j > i : \text{$j$ appears in box 3} \} + \# \{ j < i : \text{$j$ exits right in box 1} \} \\
&\hspace{0.5cm} + \# \{ j < i : \text{$j$ exits right in box 3} \} + \# \{ j < i : \text{$j$ exits right in box 4} \}. 
\end{aligned} \]
It is easy to see that 
\[ \begin{aligned}
b-a &= \# \{ j > i : \text{$j$ appears in box 3} \} - \# \{ j > i : \text{$j$ is vertical in box 2} \} \\
&\hspace{0.5cm} + \# \{ j < i : \text{$j$ exits right in box 1} + \# \{ j < i : \text{$j$ exits right in box 4} \} \\
&= \# \{ j > i : c_i \geq c_j \} + \# \{ j < i : c_i > c_j \} = \# \{ j : j \succ i \} = l.
\end{aligned} \]
Therefore 
\[ \begin{aligned}
\displaystyle \frac{
\weight \left ( \resizebox{2cm}{!}{
\begin{tikzpicture}[baseline=(current bounding box.center)] 
\draw[fill=violet] (0,1) rectangle (2,2); \draw (0,0) grid (2,2);
\draw[green] (0.5,0.5) -- (0.5,1.5) -- (1.5,1.5) -- (1.5,2);
\draw[black,fill=white] (1.5,0) circle (.5ex); \draw[black,fill=white] (2,0.5) circle (.5ex); \draw[black,fill=white] (2,1.5) circle (.5ex); \draw[black,fill=white] (1,0.5) circle (.5ex); \draw[black,fill=white] (1.5,1) circle (.5ex);
\node[left] at (0,0.5) {$r$}; \node[left] at (0,1.5) {$-t^lr$};
\end{tikzpicture} } \right )
}{
\weight \left ( \resizebox{2cm}{!}{
\begin{tikzpicture}[baseline=(current bounding box.center)] 
\draw[fill=violet] (0,1) rectangle (2,2); \draw (0,0) grid (2,2);
\draw[green] (0.5,0.5) -- (1.5,0.5) -- (1.5,1.5) -- (1.5,2);
\draw[black,fill=white] (1.5,0) circle (.5ex); \draw[black,fill=white] (2,0.5) circle (.5ex); \draw[black,fill=white] (2,1.5) circle (.5ex); \draw[black,fill=white] (0.5,1) circle (.5ex); \draw[black,fill=white] (1,1.5) circle (.5ex);
\node[left] at (0,0.5) {$r$}; \node[left] at (0,1.5) {$-t^lr$};
\end{tikzpicture} } \right ) 
}
= \frac{-t^lr \cdot t^a}{r \cdot t^b} = -t^{l+a-b} = -1.
\end{aligned} \]
Thus Equation \ref{eq:2by2weightL} holds.
\end{proof}

\subsection{Swapping single rows}

\indent In this subsection, we prove an identity of the supersymmetric LLT polynomials in the case $p=1$ and $\bm{\mu} = \0$.  Since $p=1$, $\bm{\lambda}$ can be written as $((\lambda_1),\ldots,(\lambda_k))$.  Moreover, in any configuration of the lattice $\latS_{n,m}(\bm{\lambda})$, there is exactly one path of each of the $k$ colors, and these paths enter the lattice through the bottom in the same column.  Given nonnegative integers $\nu_1,\ldots,\nu_k$, we define
\[ \Inv((\nu_1),\ldots,(\nu_k)) = \# \{ a < b : \nu_a > \nu_b \}. \]
Given $\bm{\lambda} \in P_1^{(k)}$, we say that $\bm{\nu} \in P_1^{(k)}$ is a rearrangement of $\bm{\lambda}$ if there exists a permutation $\sigma \in S_k$ such that $\nu_i = \lambda_{\sigma(i)}$ for all $i \in \{1,\ldots,k\}$.

\begin{prop} \label{prop:single-rows-swap}
Let $\bm{\lambda} \in P_1^{(k)}$. If $\bm{\nu} \in P_1^{(k)}$ is a rearrangement of $\bm{\lambda}$, then 
\[ \mc{L}^S_{\bm{\lambda}}(X_n;Y_m;t) = t^{\Inv(\bm{\lambda})-\Inv(\bm{\nu})} \mc{L}^S_{\bm{\nu}}(X_n;Y_m;t). \]
\end{prop}

\begin{proof} We start with some simple reductions.  It is enough to consider the case
where $\lambda_1 \geq \cdots \geq \lambda_k$.  Thus it is enough to show that, given $i \in \{1,\ldots,k-1\}$, 
\[ \mc{L}^S_{\bm{\lambda}}(X_n;Y_m;t) = t \cdot \mc{L}^S_{\bm{\nu}}(X_n;Y_m;t) \]
where $\lambda_1 \geq \cdots \geq \lambda_i > \lambda_{i+1} \geq \cdots \geq \lambda_k$ and 
\[ \nu_j = \left \{ \begin{array}{ll} \lambda_j & j \neq i,i+1 \\ \lambda_{i+1} & j = i \\ \lambda_i & j = i+1 \end{array} \right . . \]
We will let blue be color $i$ and red be color $i+1$.

We will now define a bijection
\[ \rho : LC(\latS_{n,m}(\bm{\lambda})) \rightarrow LC(\latS_{n,m}(\bm{\nu})). \]
Fix a configuration $C \in \latS_{n,m}(\bm{\lambda})$.  Since $\lambda_i \geq \lambda_{i+1}$, the column in which color $i$ exits the lattice is strictly right of the column in which color $i+1$ exits the lattice.  Therefore, since color $i$ and color $i+1$ enter the lattice in the same column, $C$ must have a row in which color $i$ enters weakly left and exits strictly right of color $i+1$.  Thus $C$ must have a vertex of the form
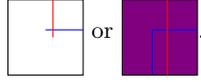
\[
\begin{tikzpicture}[baseline=(current bounding box.center)]
\draw (0,0) rectangle (1,1);
\draw[blue] (0.5,0.6)--(1,0.6);
\draw[red] (0.6,0.5)--(0.6,1);
\end{tikzpicture}
\text{ or }
\begin{tikzpicture}[baseline=(current bounding box.center)]
\draw[fill=violet] (0,0) rectangle (1,1);
\draw[blue] (0.4,0)--(0.4,0.6)--(1,0.6);
\draw[red] (0.6,0)--(0.6,1);
\end{tikzpicture} .
\]
Consider the North-Eastern-most vertex $V$ of this form.  Swap color $i$ and color $i+1$ in every vertex North-East of $V$.  For example,
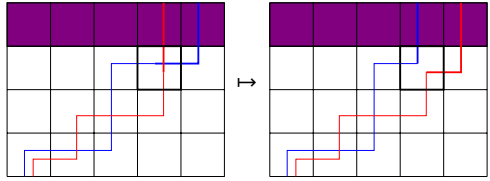
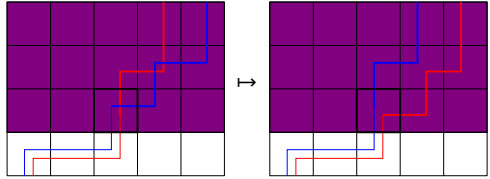
\[
	\resizebox{0.2\textwidth}{!}{
	\begin{tikzpicture}[baseline=(current bounding box.center)]
	\draw[fill=violet] (0,3) rectangle (5,4);
	\draw (0,0) grid (5,4);
	
	\draw[blue] (0.4,0)--(0.4,0.6)--(1.4,0.6)--(2.4,0.6)--(2.4,2.6)--(3.4,2.6);
	\draw[blue, very thick] (3.4,2.6)--(4.4,2.6)--(4.4,4);
	\draw[red] (0.6,0)--(0.6,0.4)--(1.6,0.4)--(1.6,1.4)--(3.6,1.4)--(3.6,2.4);
	\draw[red, very thick] (3.6,2.4)--(3.6,4);
	
	\draw[black, very thick] (3,2)--(4,2)--(4,3)--(3,3)--(3,2);

	\end{tikzpicture} }
	\mapsto
	\resizebox{0.2\textwidth}{!}{
	\begin{tikzpicture}[baseline=(current bounding box.center)]
	\draw[fill=violet] (0,3) rectangle (5,4);
	\draw (0,0) grid (5,4);
	
	\draw[blue] (0.4,0)--(0.4,0.6)--(1.4,0.6)--(2.4,0.6)--(2.4,2.6)--(3.4,2.6);
	\draw[blue, very thick] (3.4,2.6)--(3.4,4);
	\draw[red] (0.6,0)--(0.6,0.4)--(1.6,0.4)--(1.6,1.4)--(3.6,1.4)--(3.6,2.4);
	\draw[red, very thick] (3.6,2.4)--(4.4,2.4)--(4.4,4);
	
	\draw[black, very thick] (3,2)--(4,2)--(4,3)--(3,3)--(3,2);
	
	\end{tikzpicture} }
\]
and
\[ 
    \resizebox{0.2\textwidth}{!}{
	\begin{tikzpicture}[baseline=(current bounding box.center)]
	\draw[fill=violet] (0,1) rectangle (5,4);
	\draw (0,0) grid (5,4);
	
	\draw[red] (0.6,0)--(0.6,0.4)--(1.6,0.4)--(2.6,0.4)--(2.6,2.4)--(3.6,2.4);
	\draw[red, very thick] (2.6,1.4)--(2.6,2.4)--(3.6,2.4)--(3.6,4);
	\draw[blue] (0.4,0)--(0.4,0.6)--(2.4,0.6)--(2.4,1.6)--(3.4,1.6)--(3.4,2.6);
	\draw[blue, very thick] (2.4,1.6)--(3.4,1.6)--(3.4,2.6)--(4.6,2.6)--(4.6,4);
	
	\draw[black, very thick] (2,1)--(3,1)--(3,2)--(2,2)--(2,1);
	
	\end{tikzpicture} }
	\mapsto
	\resizebox{0.2\textwidth}{!}{
	\begin{tikzpicture}[baseline=(current bounding box.center)]
	\draw[fill=violet] (0,1) rectangle (5,4);
	\draw (0,0) grid (5,4);
	
	\draw[red] (0.6,0)--(0.6,0.4)--(1.6,0.4)--(2.6,0.4)--(2.6,1.4);
	\draw[red, very thick] (2.6,1.4)--(3.6,1.4)--(3.6,2.4)--(4.4,2.4)--(4.4,4);
	\draw[blue] (0.4,0)--(0.4,0.6)--(2.4,0.6)--(2.4,1.6);
	\draw[blue, very thick] (2.4,1.6)--(2.4,2.6)--(3.4,2.6)--(3.4,4);
	
	\draw[black, very thick] (2,1)--(3,1)--(3,2)--(2,2)--(2,1);
	
	\end{tikzpicture} }
\]
The result is a configuration $\rho(C) \in LC(\latS_{n,m}(\bm{\nu}))$.  It's clear that $\rho$ is a bijection.  We will now compare $\weight(C)$ with $\weight(\rho(C))$.  There are four types of vertices to consider.
\begin{enumerate}
    \item For the vertex $V$ itself, the effect of applying $\rho$ is
    
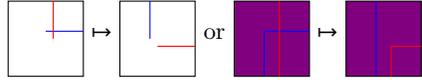
\[
    \begin{tikzpicture}[baseline=(current bounding box.center)]
    \draw (0,0) rectangle (1,1);
    \draw[blue] (0.5,0.6)--(1,0.6);
    \draw[red] (0.6,0.5)--(0.6,1);
    \end{tikzpicture}
    \mapsto
    \begin{tikzpicture}[baseline=(current bounding box.center)]
    \draw (0,0) rectangle (1,1);
    \draw[red] (0.5,0.4)--(1,0.4);
    \draw[blue] (0.4,0.5)--(0.4,1);
    \end{tikzpicture}
    \text{ or }
    \begin{tikzpicture}[baseline=(current bounding box.center)]
    \draw[fill=violet] (0,0) rectangle (1,1);
    \draw[blue] (0.4,0)--(0.4,0.6)--(1,0.6);
    \draw[red] (0.6,0)--(0.6,1);
    \end{tikzpicture} 
    \mapsto
    \begin{tikzpicture}[baseline=(current bounding box.center)]
    \draw[fill=violet] (0,0) rectangle (1,1);
    \draw[red] (0.6,0)--(0.6,0.4)--(1,0.4);
    \draw[blue] (0.4,0)--(0.4,1);
    \end{tikzpicture} .
    \]
    In either case, it is easy to see that the weight before applying $\rho$ is $t$ times the weight after applying $\rho$.
    \item For any vertex that is not North-East of $V$, then $\rho$ does not change the configuration, so the weight is not changed.
    \item For any vertex that is North-East of $V$ and either color $i$ or color $i+1$ is absent in this vertex, then $\rho$ swaps color $i$ or color $i+1$, but the weight is not changed.
    \item  For any vertex that is North-East of $V$ and both color $i$ and color $i+1$ are present in this vertex, then this vertex must have the form
    
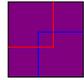
\[ 
    \begin{tikzpicture}[baseline=(current bounding box.center)]
    \draw[fill=violet] (0,0) rectangle (1,1);
    \draw[blue] (0.4,0)--(0.4,0.6)--(1,0.6);
    \draw[red] (0,0.4)--(0.6,0.4)--(0.6,1);
    \end{tikzpicture} .
    \]
    \noindent Applying $\rho$ swaps color $i$ and color $i+1$, resulting in
    
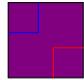
\[ 
    \begin{tikzpicture}[baseline=(current bounding box.center)]
    \draw[fill=violet] (0,0) rectangle (1,1);
    \draw[red] (0.6,0)--(0.6,0.4)--(1,0.4);
    \draw[blue] (0,0.6)--(0.4,0.6)--(0.4,1);
    \end{tikzpicture} .
    \]
    \noindent It is easy to see that the weight is not changed.

\end{enumerate}
Therefore
\[ \weight(C) = t \cdot \weight(\rho(C)) \]
and the proposition follows.
\end{proof}

\begin{remark}
The above proposition extends \cite[Prop. 5.5]{CGKM}, and in fact these two propositions are proven in nearly identical ways.
\end{remark}

\pagebreak

\section{Relating $\mc{L}^S$ to $\mc{G}$ \label{sec:relate-ls-to-g}}

\indent The goal of this section is to relate the partition function $\mc{L}^S_\lm$ of Definition \ref{def:latS} to the super ribbon function $\mc{G}^{(k)}_{\lambda/\mu}$ of Definition \ref{def:super-ribbon-function}.  In \cite{CYZZ} the authors construct a lattice model whose partition function is equal to the spin LLT polynomials. Lemma \ref{lem:white-spin} below, which relates our vertex model to the spin LLT polynomial, can also be interpreted as a mapping between our vertex model and the one in \cite{CYZZ} (see Remark \ref{rem:CYZZ}).  We will adopt the following conventions.  Fix a positive integer $k$.  Let $\lm$ be the $k$-quotient of $\lambda/\mu$.  Let $\mc{A} = \{1 < 2 < \cdots < n \}$ and $\mc{A'} = \{ 1' < 2' < \cdots < m' \}$.  We use the ordering $1 < 2 < \ldots < n < 1' < 2' < \ldots < m'$ on $\mc{A} \cup \mc{A'}$.  

We begin by constructing a bijection
\[ \SSSYT(\lm) \rightarrow LC(\latS_{n,m}(\lm)) \]
where $LC(\latS_{n,m}(\lm))$ is the set of configurations on the lattice $\latS_{n,m}(\lm)$ from Definition \ref{def:latS}.  We do this in the usual way in which each row of $i$-th tableaux maps to a path of color $i$.  Precisely, given $\bm{T} = (T^{(1)},\ldots,T^{(k)}) \in \SSSYT(\lm)$, the corresponding $C \in LC(\latS_{n,m}(\lm))$ is constructed as follows.  Fix $i \in [k]$ and fix a row
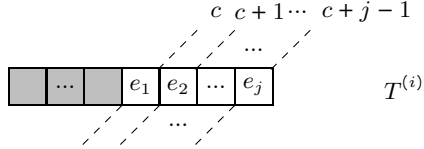
\begin{center}
\ytableausetup{nosmalltableaux}
\ytableausetup{nobaseline}
\begin{ytableau}
\none & \none & \none & \none & \none & \none[c] & \none[\enspace c+1] & \none[\enspace...] & \none[\hspace{1cm} c+j-1] \\
\none & \none & \none & \none & \none[\diag] & \none[\diag] & \none[...] & \none[\diag] \\
*(lightgray) & *(lightgray)... & *(lightgray) & e_1 & e_2 & ... & e_j & \none & \none & \none & \none[T^{(i)}] \\
\none & \none & \none[\diag] & \none[\diag] & \none[...] & \none[\diag]
\end{ytableau}
\end{center}
\noindent in $T^{(i)}$.  (Here we have labelled the diagonal content lines going through the row.)  The corresponding path in $C$ has color $i$, enters the lattice via the bottom of column $c$, exits the lattice via the top of column $c+j$, and crosses from column $c+l-1$ to column $c+l$ at 
\[ \left \{ \begin{array}{ll} \text{the $a$-th white row} & \text{if $e_l = a \in \mc{A}$} \\ \text{the $a$-th purple row} & \text{if $e_l = a' \in \mc{A'}$} \end{array} \right . \]
for each index $l \in [j]$.  Recall that the Littlewood $k$-quotient map is a bijection
\[ \SRT_k(\lambda/\mu) \rightarrow \SSSYT(\lm). \]
The composition of these two bijections gives a bijection
\[ \theta: \SRT_k(\lambda/\mu) \rightarrow LC(\latS_{n,m}(\lm)). \]

\begin{example}
Let $n=3$, $m=4$, and $k=3$.  Recall Example \ref{ex:extended-lqm}.
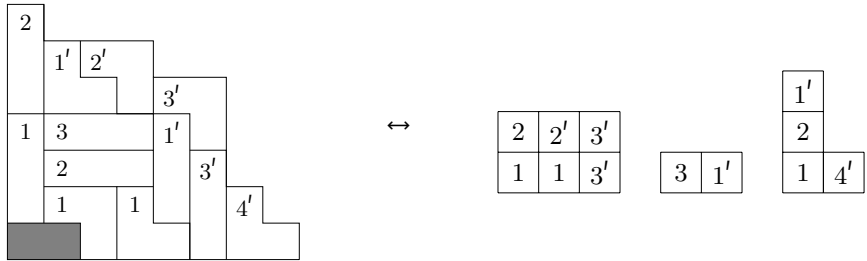
\[ \resizebox{4cm}{!}{
\begin{tikzpicture}[baseline=(current bounding box.center)]
\fill[gray] (0,0) rectangle (2,1);
\draw[thin] (8,0) -- (0,0) -- (0,7) -- (1,7) -- (1,6) -- (4,6) -- (4,5) -- (6,5) -- (6,2) -- (7,2) -- (7,1) -- (8,1) -- (8,0);
\draw[thin] (2,0) -- (2,1) -- (0,1);
\draw[thin] (1,1) -- (1,7);
\draw[thin] (0,4) -- (5,4) -- (5,0);
\draw[thin] (1,2) -- (3,2) -- (3,0);
\draw[thin] (1,3) -- (4,3) -- (4,1) -- (5,1) -- (5,0);
\draw[thin] (3,2) -- (4,2);
\draw[thin] (2,6) -- (2,5) -- (3,5) -- (3,4) -- (4,4) -- (4,3);
\draw[thin] (5,3) -- (6,3) -- (6,0);
\draw[thin] (4,5) -- (4,4);
\node[scale=2] at (0.5,3.5) {$1$}; \node[scale=2] at (0.5,6.5) {$2$};
\node[scale=2] at (1.5,1.5) {$1$}; \node[scale=2] at (1.5,2.5) {$2$}; \node[scale=2] at (1.5,3.5) {$3$}; \node[scale=2] at (1.5,5.5) {$1'$};
\node[scale=2] at (2.5,5.5) {$2'$}; \node[scale=2] at (3.5,1.5) {$1$};
\node[scale = 2] at (4.5,3.5) {$1'$}; \node[scale = 2] at (4.5,4.5) {$3'$};
\node[scale = 2] at (5.5,2.5) {$3'$}; \node[scale = 2] at (6.5,1.5) {$4'$};
\end{tikzpicture}
}
\hspace{1cm} \leftrightarrow \hspace{1cm}
\resizebox{5cm}{!}{
\begin{tikzpicture}[baseline=(current bounding box.center)]
\draw (0,0) grid (3,2);
\draw (4+0,0) grid (4+2,1);
\draw (7+0,0) grid (7+1,3); \draw (7+1,0) grid (7+2,1);
\node[scale=2] at (0.5,0.5) {$1$}; \node[scale=2] at (0.5,1.5) {$2$}; \node[scale=2] at (1.5,0.5) {$1$}; \node[scale=2] at (1.5,1.5) {$2'$}; \node[scale=2] at (2.5,0.5) {$3'$}; \node[scale=2] at (2.5,1.5) {$3'$};
\node[scale=2] at (4+0.5,0.5) {$3$}; \node[scale=2] at (4+1.5,0.5) {$1'$};
\node[scale=2] at (7+0.5,0.5) {$1$}; \node[scale=2] at (7+1.5,0.5) {$4'$}; \node[scale=2] at (7+0.5,1.5) {$2$}; \node[scale=2] at (7+0.5,2.5) {$1'$};
\end{tikzpicture} }
\]
The corresponding configuration is
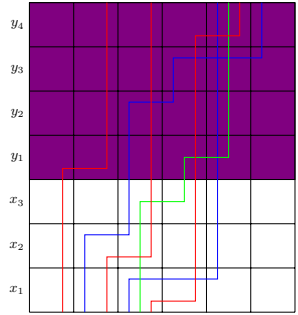
\[
\resizebox{4cm}{!}{
\begin{tikzpicture}[baseline=(current bounding box.center)] 
\draw[fill=violet] (-1,3) rectangle (5,7); \draw (-1,0) grid (5,7);
\node[left] at (-1,0.5) {$x_1$}; \node[left] at (-1,1.5) {$x_2$}; \node[left] at (-1,2.5) {$x_3$}; \node[left] at (-1,3.5) {$y_1$}; \node[left] at (-1,4.5) {$y_2$}; \node[left] at (-1,5.5) {$y_3$}; \node[left] at (-1,6.5) {$y_4$};
\draw[blue] (1.25,0)--(1.25,0.75)--(3.25,0.75)--(3.25,5.75)--(4.25,5.75)--(4.25,7);
\draw[blue] (0.25,0)--(0.25,1.75)--(1.25,1.75)--(1.25,4.75)--(2.25,4.75)--(2.25,5.75)--(3.25,5.75)--(3.25,7);
\draw[green] (1.5,0)--(1.5,2.5)--(2.5,2.5)--(2.5,3.5)--(3.5,3.5)--(3.5,7);
\draw[red] (1.75,0)--(1.75,0.25)--(2.75,0.25)--(2.75,6.25)--(3.75,6.25)--(3.75,7);
\draw[red] (0.75,0)--(0.75,1.25)--(1.75,1.25)--(1.75,7);
\draw[red] (-0.25,0)--(-0.25,3.25)--(0.75,3.25)--(0.75,7);
\end{tikzpicture} }
\]
where blue is color $1$, green is color $2$, and red is color $3$.
\end{example}

\begin{remark}
When $m=0$, the bijection $\SSSYT(\lm) \rightarrow LC(\latS_{n,m}(\lm))$ becomes a bijection
\[ \SSYT(\lm) \rightarrow LC(\latW_n(\lm)). \]
This bijection was used in \cite{CGKM} to prove Theorem \ref{thm:cgkm-lattice}. 
\end{remark}

\begin{remark}
The bijection $\theta$ restricts to bijections
\[ \begin{aligned} 
&\HRS_k(\lambda/\mu) \rightarrow LC(\latW_1(\lm)), \hspace{1cm} \VRS_k(\lambda/\mu) \rightarrow LC(\latP_1(\lm))
\end{aligned} \]
by artificially labelling each ribbon in a horizontal (vertical) $k$-ribbon strip with a $1$ ($1'$).
\end{remark}

\indent For the rest of this section, we will switch to drawing our Young diagrams in Russian convention, so that rows are oriented SW-to-NE and columns are oriented SE-to-NW.  The reason for this switch is to allow for an elegant graphical interpretation of $\theta$. Let $T \in \SRT_k(\lambda/\mu)$ and let $C = \theta(T) \in LC(\latS_{n,m}(\lm))$.  By the construction of $\theta$, we note that
\begin{enumerate}
    \item for each $i \in \mc{A}$, the horizontal ribbon strip $\lambda_{\leq i}/\lambda_{\leq i-1}$ of ribbons labelled $i$ in $T$ corresponds to the $i$-th white row in $C$; and
    \item for each $i' \in \mc{A'}$, the vertical ribbon strip $\lambda_{\leq i}/\lambda_{\leq i-1}$ of ribbons labelled $i'$ in $T$ corresponds to the $i$-th purple row in $C$.
\end{enumerate}
Given a horizontal (vertical) ribbon strip inside $T$, we ``drop down" the Maya diagrams of the top and bottom boundaries to obtain the top and bottom boundaries of the corresponding white (purple) row in $C$.  Moreover, the top and bottom boundaries of the row uniquely determine the path configuration of the row.

\begin{example} \label{ex:theta-graphical} Take $k=3$.  Let blue be color 1, green be color 2, and red be color 3.  Consider the following horizontal 3-ribbon strip of shape $(6,6,3,0,0,0)/(0,0,0,0,0,0)$.
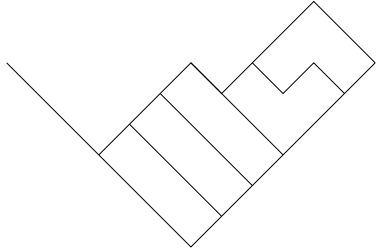
\begin{center}
\resizebox{5cm}{!}{
\begin{tikzpicture}
\draw (-6,6)--(0,0)--(6,6);
\draw (-3,3)--(0,6)--(1,5)--(4,8)--(6,6);
\draw (-2,4)--(1,1); \draw (-1,5)--(2,2); \draw (0,6)--(3,3);
\draw (2,6)--(3,5)--(4,6)--(5,5);
\end{tikzpicture} }  
\end{center}
We now color the steps on the top and bottom boundaries, from left to right.  A South-East step in position $i \mod k$ gets color $i$.  North-East steps are not colored.
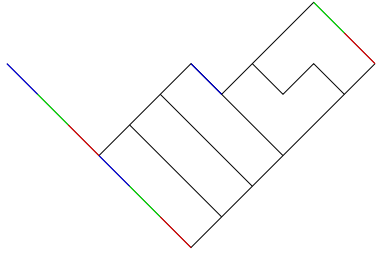
\begin{center}
\resizebox{5cm}{!}{
\begin{tikzpicture}
\draw (-6,6)--(0,0)--(6,6);
\draw (-3,3)--(0,6)--(1,5)--(4,8)--(6,6);
\draw (-2,4)--(1,1); \draw (-1,5)--(2,2); \draw (0,6)--(3,3);
\draw (2,6)--(3,5)--(4,6)--(5,5);
\draw[blue, thick] (-6,6)--(-5,5); \draw[blue, thick] (0,6)--(1,5); \draw[blue,thick] (-3,3)--(-2,2);
\draw[green, thick] (-5,5)--(-4,4); \draw[green, thick] (4,8)--(5,7); \draw[green,thick] (-2,2)--(-1,1);
\draw[red, thick] (-4,4)--(-3,3); \draw[red, thick] (5,7)--(6,6); \draw[red,thick] (-1,1)--(0,0);
\end{tikzpicture} }
\end{center}
We now ``drop down" the steps on the top and bottom boundaries of the horizontal 3-ribbon strip to obtain the top and bottom boundaries of the corresponding white row.  The steps in positions $(j-1)k+1,\ldots,(j-1)k+k$ correspond to the $j$-th leftmost vertex.  A step of color $i$ corresponds to a particle of color $i$, indicating that a path of color $i$ is incident at the edge.
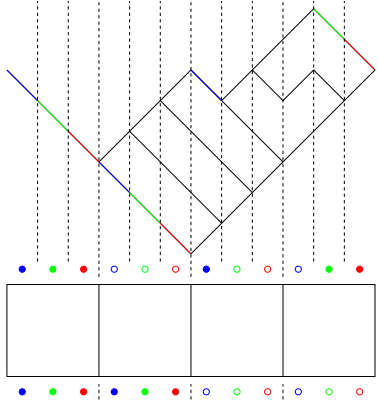
\begin{center}
\resizebox{5cm}{!}{
\begin{tikzpicture}
\draw (-6,6)--(0,0)--(6,6);
\draw (-3,3)--(0,6)--(1,5)--(4,8)--(6,6);
\draw (-2,4)--(1,1); \draw (-1,5)--(2,2); \draw (0,6)--(3,3);
\draw (2,6)--(3,5)--(4,6)--(5,5);
\draw[blue, thick] (-6,6)--(-5,5); \draw[blue, thick] (0,6)--(1,5); \draw[blue,thick] (-3,3)--(-2,2);
\draw[green, thick] (-5,5)--(-4,4); \draw[green, thick] (4,8)--(5,7); \draw[green,thick] (-2,2)--(-1,1);
\draw[red, thick] (-4,4)--(-3,3); \draw[red, thick] (5,7)--(6,6); \draw[red,thick] (-1,1)--(0,0);
\draw[dashed] (-5,-0.25)--(-5,8.25);
\draw[dashed] (-4,-0.25)--(-4,8.25);
\draw[dashed] (-3,-0.75)--(-3,8.25);
\draw[dashed] (-2,-0.25)--(-2,8.25);
\draw[dashed] (-1,-0.25)--(-1,8.25);
\draw[dashed] (0,-0.75)--(0,8.25);
\draw[dashed] (1,-0.25)--(1,8.25);
\draw[dashed] (2,-0.25)--(2,8.25);
\draw[dashed] (3,-0.75)--(3,8.25);
\draw[dashed] (4,-0.25)--(4,8.25);
\draw[dashed] (5,-0.25)--(5,8.25);
\draw[blue,fill=blue] (-5.5,-0.5) circle (0.1);
\draw[green,fill=green] (-4.5,-0.5) circle (0.1);
\draw[red,fill=red] (-3.5,-0.5) circle (0.1);
\draw[blue] (-2.5,-0.5) circle (0.1);
\draw[green] (-1.5,-0.5) circle (0.1);
\draw[red] (-0.5,-0.5) circle (0.1);
\draw[blue,fill=blue] (0.5,-0.5) circle (0.1);
\draw[green] (1.5,-0.5) circle (0.1);
\draw[red] (2.5,-0.5) circle (0.1);
\draw[blue] (3.5,-0.5) circle (0.1);
\draw[green,fill=green] (4.5,-0.5) circle (0.1);
\draw[red,fill=red] (5.5,-0.5) circle (0.1);

\draw (-6,-1) rectangle (6,-4);
\draw (-3,-4)--(-3,-1); \draw (0,-4)--(0,-1); \draw (3,-4)--(3,-1);
\draw[blue,fill=blue] (-5.5,-4.5) circle (0.1);
\draw[green,fill=green] (-4.5,-4.5) circle (0.1);
\draw[red,fill=red] (-3.5,-4.5) circle (0.1);
\draw[blue,fill=blue] (-2.5,-4.5) circle (0.1);
\draw[green,fill=green] (-1.5,-4.5) circle (0.1);
\draw[red,fill=red] (-0.5,-4.5) circle (0.1);
\draw[blue] (0.5,-4.5) circle (0.1);
\draw[green] (1.5,-4.5) circle (0.1);
\draw[red] (2.5,-4.5) circle (0.1);
\draw[blue] (3.5,-4.5) circle (0.1);
\draw[green] (4.5,-4.5) circle (0.1);
\draw[red] (5.5,-4.5) circle (0.1);
\draw[dashed] (-3,-4.75)--(-3,-4.25);
\draw[dashed] (0,-4.75)--(0,-4.25);
\draw[dashed] (3,-4.75)--(3,-4.25);
\end{tikzpicture} }
\end{center}
With these top and bottom boundary conditions (and requiring that no paths be incident at the left and right edges of the row), there is a unique path configuration.
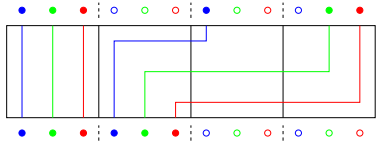
\begin{center}
\resizebox{5cm}{!}{
\begin{tikzpicture}
\draw[blue,fill=blue] (-5.5,-0.5) circle (0.1);
\draw[green,fill=green] (-4.5,-0.5) circle (0.1);
\draw[red,fill=red] (-3.5,-0.5) circle (0.1);
\draw[blue] (-2.5,-0.5) circle (0.1);
\draw[green] (-1.5,-0.5) circle (0.1);
\draw[red] (-0.5,-0.5) circle (0.1);
\draw[blue,fill=blue] (0.5,-0.5) circle (0.1);
\draw[green] (1.5,-0.5) circle (0.1);
\draw[red] (2.5,-0.5) circle (0.1);
\draw[blue] (3.5,-0.5) circle (0.1);
\draw[green,fill=green] (4.5,-0.5) circle (0.1);
\draw[red,fill=red] (5.5,-0.5) circle (0.1);
\draw[dashed] (-3,-0.75)--(-3,-0.25);
\draw[dashed] (0,-0.75)--(0,-0.25);
\draw[dashed] (3,-0.75)--(3,-0.25);
\draw (-6,-1) rectangle (6,-4);
\draw (-3,-4)--(-3,-1); \draw (0,-4)--(0,-1); \draw (3,-4)--(3,-1);
\draw[blue,fill=blue] (-5.5,-4.5) circle (0.1);
\draw[green,fill=green] (-4.5,-4.5) circle (0.1);
\draw[red,fill=red] (-3.5,-4.5) circle (0.1);
\draw[blue,fill=blue] (-2.5,-4.5) circle (0.1);
\draw[green,fill=green] (-1.5,-4.5) circle (0.1);
\draw[red,fill=red] (-0.5,-4.5) circle (0.1);
\draw[blue] (0.5,-4.5) circle (0.1);
\draw[green] (1.5,-4.5) circle (0.1);
\draw[red] (2.5,-4.5) circle (0.1);
\draw[blue] (3.5,-4.5) circle (0.1);
\draw[green] (4.5,-4.5) circle (0.1);
\draw[red] (5.5,-4.5) circle (0.1);
\draw[dashed] (-3,-4.75)--(-3,-4.25);
\draw[dashed] (0,-4.75)--(0,-4.25);
\draw[dashed] (3,-4.75)--(3,-4.25);
\draw[blue] (-5.5,-4)--(-5.5,-1); \draw[blue] (-2.5,-4)--(-2.5,-1.5)--(0.5,-1.5)--(0.5,-1);
\draw[green] (-4.5,-4)--(-4.5,-1); \draw[green] (-1.5,-4)--(-1.5,-2.5)--(4.5,-2.5)--(4.5,-1);
\draw[red] (-3.5,-4)--(-3.5,-1); \draw[red] (-0.5,-4)--(-0.5,-3.5)--(5.5,-3.5)--(5.5,-1);
\end{tikzpicture} }
\end{center}

\end{example}

\begin{remark} \label{rmk:theta-commutes}
We leave it as an exercise for the reader to verify that our graphical interpretation of $\theta$ is correct (see Lemma \ref{lem:lqm-properties}).  We also leave it as an exercise for the reader to check that the diagram
\begin{center}
\begin{tikzcd}
\HRS_k(\lambda/\mu) \arrow[r, "conjugate"] \arrow[d, "\theta"] & \VRS_k(\lambda'/\mu') \arrow[d, "\theta"] \\
LC(\latW_1(\lm)) \arrow[r, "\psi"]                             & LC(\latP_1(\lmp))                        
\end{tikzcd}
\end{center}
commutes.  Here $\psi$ is the bijection defined at the beginning of Section \ref{sec:identities}.  (This result is not needed in the rest of this paper.)
\end{remark}

\indent In order to relate $\mc{L}^S_\lm$ to $\mc{G}^{(k)}_{\lambda/\mu}$, we must consider how the spin
\[ \spin(T) = \sum_{\text{ribbons $R$ in $T$}} (h(R)-1) \]
of a horizontal (vertical) ribbon strip $T$ appears in the corresponding path configuration $\theta(T)$ of a single white (purple) row.  Clearly $\spin(T)$ equals the number of positions that the tile 
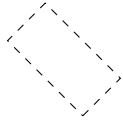
\begin{center}
\begin{tikzpicture}[scale=0.5,baseline=(current bounding box.center)] \draw[dashed] (0,0)--(1,1)--(-1,3)--(-2,2)--(0,0); \end{tikzpicture} 
\end{center}
(two cells in the same column and ribbon) can be placed in $T$.  For example, if $T$ is
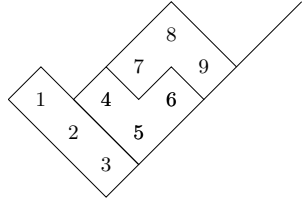
\begin{center}
\resizebox{4cm}{!}{
\begin{tikzpicture}
\draw (-3,3)--(0,0)--(1,1)--(-2,4)--(-3,3);
\draw (1,1)--(3,3)--(2,4)--(1,3)--(0,4)--(-1,3)--(1,1);
\draw (0,4)--(2,6)--(4,4)--(3,3);
\draw (4,4)--(6,6);
\node[scale=2] at (-2,3) {$1$}; \node[scale=2] at (-1,2) {$2$}; \node[scale=2] at (0,1) {$3$};
\node[scale=2] at (0,3) {$4$}; \node[scale=2] at (1,2) {$5$}; \node[scale=2] at (2,3) {$6$};
\node[scale=2] at (0,3) {$4$}; \node[scale=2] at (1,2) {$5$}; \node[scale=2] at (2,3) {$6$};
\node[scale=2] at (1,4) {$7$}; \node[scale=2] at (2,5) {$8$}; \node[scale=2] at (3,4) {$9$};
\end{tikzpicture} }  
\end{center}
there are 4 such positions - cells 1 and 2, cells 2 and 3, cells 4 and 5, and cells 8 and 9 - and indeed the spin is 4.  We can count these positions according to the  ``slice" containing the middle of each tile.  In our example, the slices are given by 
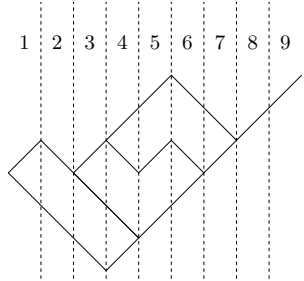
\begin{center}
\resizebox{4cm}{!}{
\begin{tikzpicture}
\draw (-3,3)--(0,0)--(1,1)--(-2,4)--(-3,3);
\draw (1,1)--(3,3)--(2,4)--(1,3)--(0,4)--(-1,3)--(1,1);
\draw (0,4)--(2,6)--(4,4)--(3,3);
\draw (4,4)--(6,6);
\draw[dashed] (-2,-0.25)--(-2,8.25);
\draw[dashed] (-1,-0.25)--(-1,8.25);
\draw[dashed] (0,-0.25)--(0,8.25);
\draw[dashed] (1,-0.25)--(1,8.25);
\draw[dashed] (2,-0.25)--(2,8.25);
\draw[dashed] (3,-0.25)--(3,8.25);
\draw[dashed] (4,-0.25)--(4,8.25);
\draw[dashed] (5,-0.25)--(5,8.25);
\node[scale=2] at (-2.5,7) {$1$}; \node[scale=2] at (-1.5,7) {$2$}; \node[scale=2] at (-0.5,7) {$3$}; \node[scale=2] at (0.5,7) {$4$}; \node[scale=2] at (1.5,7) {$5$}; \node[scale=2] at (2.5,7) {$6$}; \node[scale=2] at (3.5,7) {$7$}; \node[scale=2] at (4.5,7) {$8$}; \node[scale=2] at (5.5,7) {$9$}; 
\end{tikzpicture} }  
\end{center}
so slices 2, 3, 4, and 6 each contribute 1 to the spin and the other slices do not contribute to the spin.  In each slice, there are only four shapes that can appear.
\begin{enumerate}
    \item a column parallelogram (two adjacent halves of cells in the same column and ribbon)
    \item a row parallelogram (two adjacent halves of cells in the same row and ribbon)
    \item a head triangle (half of the head of a ribbon)
    \item a tail triangle (half of the tail of a ribbon)
\end{enumerate}
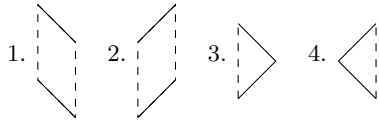
\begin{center}
\begin{tabular}{cccc}
1. 
    \begin{tikzpicture}[scale=0.5,baseline=(current bounding box.center)] \draw[dashed] (0,0)--(-1,1)--(-1,3)--(0,2)--(0,0); \draw (0,0)--(-1,1); \draw (0,2)--(-1,3); \end{tikzpicture} 
& 2.
    \begin{tikzpicture}[scale=0.5,baseline=(current bounding box.center)] \draw[dashed] (0,0)--(1,1)--(1,3)--(0,2)--(0,0); \draw (0,0)--(1,1); \draw (0,2)--(1,3); \end{tikzpicture} 
& 3.
    \begin{tikzpicture}[scale=0.5,baseline=(current bounding box.center)] \draw[dashed] (0,0)--(0,2); \draw (0,0)--(1,1)--(0,2); \end{tikzpicture} 
& 4.
    \begin{tikzpicture}[scale=0.5,baseline=(current bounding box.center)] \draw[dashed] (0,0)--(0,2); \draw (0,0)--(-1,1)--(0,2); \end{tikzpicture} 
\end{tabular}
\end{center}
(Of course, a slice could also consist of a single SE step or a single NE step.)  Any other shapes cannot appear in a slice, because if it did, the ribbon containing the shape would either contain a $2 \times 2$ square or not be a valid skew shape.  It is clear that the contribution of a slice to the spin equals the number of column parallelograms in the slice.

In the following two lemmas, we use the above discussion to characterize the spin in terms of $\theta(T)$, when $T$ is a horizontal/vertical ribbon strip.

\begin{lem} \label{lem:white-spin}
Let $T$ be a horizontal $k$-ribbon strip.  In the corresponding white row $\theta(T)$, 
\[
\spin(T) = \sum_{\alb} \left(
\# \begin{tikzpicture}[baseline=(current bounding box.center)]
\draw (0,0)--(1,0)--(1,1)--(0,1)--(0,0);
\draw[blue] (0,0.6)--(0.5,0.6);
\draw[red] (0,0.4)--(0.5,0.4);
\end{tikzpicture}
+
\# \begin{tikzpicture}[baseline=(current bounding box.center)]
\draw (0,0)--(1,0)--(1,1)--(0,1)--(0,0);
\draw[blue] (1,0.6)--(0.5,0.6);
\draw[red] (1,0.4)--(0.5,0.4);
\end{tikzpicture}
+
\# \begin{tikzpicture}[baseline=(current bounding box.center)]
\draw (0,0)--(1,0)--(1,1)--(0,1)--(0,0);
\draw[blue] (0.4,0)--(0.4,0.5);
\draw[red] (0,0.4)--(0.5,0.4);
\end{tikzpicture}
+
\# \begin{tikzpicture}[baseline=(current bounding box.center)]
\draw (0,0)--(1,0)--(1,1)--(0,1)--(0,0);
\draw[blue] (1,0.6)--(0.5,0.6);
\draw[red] (0.6,0.5)--(0.6,1);
\end{tikzpicture}
\right)
\]
where blue is color $a$ and red is color $b$.
\end{lem}
\begin{proof}
The fact that $T$ is a horizontal ribbon strip restricts the possible forms of the slices.
\begin{enumerate}
    \item If a head triangle appears, it must be at the bottom of the slice.  This is because the head of a ribbon must touch the SE boundary.
    \item If a tail triangle appears, it must be at the top of the slice.  Suppose the tail triangle of ribbon $R$ appears below a shape in ribbon $S$ in slice $i$.  Then, in all slices in which both $R$ and $S$ appear, $S$ is above $R$.  Note that $R$ appears in slices $i,\ldots,i+k$ as every ribbon has length $k$. Similarly $S$ appears in slices $i-j,\ldots,i-j+k$ for some $j \in \{0,\ldots,k\}$.  In particular, in slice $i-j+k$, the head triangle of $S$ appears above $R$, which contradicts the first restriction.
\end{enumerate}
With these restrictions in mind, one can show that each slice must have one of the following five forms:
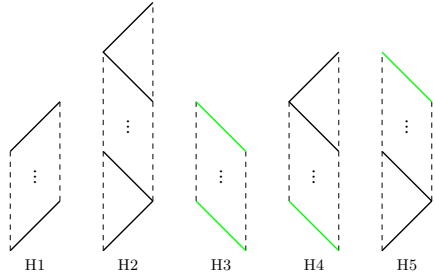
\begin{center}
\resizebox{0.4\textwidth}{!}{
\begin{tabular}{ccccccccc}
     \begin{tikzpicture}
     \draw[dashed] (0,0)--(0,2); \draw[dashed] (1,1)--(1,3); \node at (0.5,1.5) {$\vdots$};
     \draw[thick] (0,0)--(1,1); \draw[thick] (0,2)--(1,3);
     \end{tikzpicture}
     & &
     \begin{tikzpicture}
     \draw[dashed] (0,0)--(0,4); \draw[dashed] (1,1)--(1,5); \node at (0.5,2.5) {$\vdots$};
     \draw[thick] (0,0)--(1,1); \draw[thick] (1,1)--(0,2); \draw[thick] (1,5)--(0,4)--(1,3);
     \end{tikzpicture}
     & &
     \begin{tikzpicture}
     \draw[dashed] (0,0)--(0,2); \draw[dashed] (-1,1)--(-1,3); \node at (-0.5,1.5) {$\vdots$};
     \draw[green,thick] (0,0)--(-1,1); \draw[green,thick] (-1,3)--(0,2);
     \end{tikzpicture}
     & &
     \begin{tikzpicture}
     \draw[dashed] (0,0)--(0,4); \draw[dashed] (-1,1)--(-1,3); \node at (-0.5,1.5) {$\vdots$};
     \draw[green,thick] (0,0)--(-1,1); \draw[thick] (0,2)--(-1,3)--(0,4);
     \end{tikzpicture}
     & &
     \begin{tikzpicture}
     \draw[dashed] (0,0)--(0,4); \draw[dashed] (1,1)--(1,3); \node at (0.5,2.5) {$\vdots$};
     \draw[thick] (0,0)--(1,1); \draw[thick] (1,1)--(0,2); \draw[green,thick] (0,4)--(1,3);
     \end{tikzpicture}
     \\
     H1 & & H2 & & H3 & & H4 & & H5
\end{tabular}
}
\end{center}
Here $\vdots$ indicates arbitrarily many (possibly 0) copies of the shape, and South-East steps on the top/bottom boundaries are colored as in our graphical interpretation of $\theta$ (see Example \ref{ex:theta-graphical}).  
\begin{remark}\label{rem:CYZZ}
The five types of slices we draw here are in bijection with the five types of allowed vertices given in \cite[Figure 14]{CYZZ}. To see this suppose we are looking at a slice whose top and bottom boundaries correspond to color $a$. Then we can map it to a vertex of the form given in \cite{CYZZ} by assigning arrows to the edges of the vertex according to the rules: 
\begin{enumerate}
    \item If the top boundary of the slice is a NE (SE) step then the top edge of the vertex gets a down (up) arrow. Similarly for the bottom boundary of the slice.
    \item If the slice contains a head triangle, assign a left arrow to the $k^{th}$ eastern horizontal edge. If the slice contains a tail triangle, assign a left arrow to the $1^{st}$ western horizontal edge. 
    \item If a ribbon whose head is in a slice of color $b$ passes through the slice, then assign left arrows to the the $(b+k-a \mod k)^{th}$ eastern horizontal edge and the $((b+k-a \mod k) + 1)^{st}$ western horizontal edge. 
    \item Otherwise assign right arrows to the horizontal edges. 
\end{enumerate}
Assigning each slice a weight of $x$ if there is a head triangle, and $t^{\# \text{col. parellelogram}}$, makes this a weight-preserving bijection. Through the bijection $\theta$ defined above, the rest of this lemma can be seen as giving a weight-preserving bijection between our vertex model and that of \cite{CYZZ}.
\end{remark}

We claim that there is a one-to-one correspondence between the five possible forms of a slice in $T$ and the five possible path configurations of the corresponding color in the corresponding white vertex in $\theta(T)$.
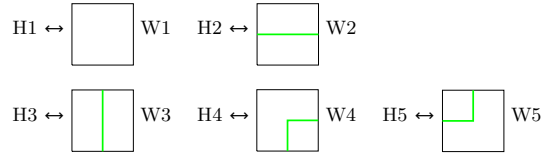
\begin{center}
\resizebox{0.5\textwidth}{!}{
\begin{tabular}{ccc}
H1 $\leftrightarrow$
     \begin{tikzpicture}[baseline=(current bounding box.center)]
     \draw (0,0)--(1,0)--(1,1)--(0,1)--(0,0);
     \end{tikzpicture} W1
     &
H2 $\leftrightarrow$
     \begin{tikzpicture}[baseline=(current bounding box.center)]
     \draw (0,0)--(1,0)--(1,1)--(0,1)--(0,0);
     \draw[green,thick] (0,0.5)--(1,0.5);
     \end{tikzpicture} W2
     & \\
     \\
H3 $\leftrightarrow$
     \begin{tikzpicture}[baseline=(current bounding box.center)]
     \draw (0,0)--(1,0)--(1,1)--(0,1)--(0,0);
     \draw[green,thick] (0.5,0)--(0.5,1);
     \end{tikzpicture} W3
     &
H4 $\leftrightarrow$
     \begin{tikzpicture}[baseline=(current bounding box.center)]
     \draw (0,0)--(1,0)--(1,1)--(0,1)--(0,0);
     \draw[green,thick] (0.5,0)--(0.5,0.5)--(1,0.5);
     \end{tikzpicture} W4
     &
H5 $\leftrightarrow$
     \begin{tikzpicture}[baseline=(current bounding box.center)]
     \draw (0,0)--(1,0)--(1,1)--(0,1)--(0,0);
     \draw[green,thick] (0,0.5)--(0.5,0.5)--(0.5,1);
     \end{tikzpicture} W5
\end{tabular}
}
\end{center}
The correspondence is obvious for H3, H4, and H5 from the top/bottom boundary conditions.  Moreover, from the top/bottom boundary conditions, slices of the form H1 or H2 correspond to path configurations of the form W1 or W2.  To show the correspondence for H1 and H2, we will show that a slice of the form H2 always gives a configuration of the form W2 and vice-versa.  
\begin{itemize}
\item Suppose slice $i$ has the form H2.  It contains the head triangle of a ribbon, so  slice $i-k$ will contain the tail triangle of this ribbon. This slice then must have the form H2 or H4.  If slice $i-k$ has the form H2, then we can repeat this argument to show that slice $i-2k$ has the form H2 or H4.  Since there are finitely many ribbons, eventually we find that slice $i-jk$ has the form H4, for some positive integer $j$, and slices $i-(j-1)k,\ldots,i-k,i$ have the form H2.  Since slice $i-jk$ has the form H4, the corresponding path configuration has the form W4, in which the path exits right.  Thus the path configuration corresponding to slice $i-(j-1)k$ must have the path entering left, so it must have the form W2, in which the path exits right.  Repeating, we conclude that the path configuration corresponding to slice $i$ has the form W2. 
\item Suppose slice $i$ corresponds to a path configuration of the form W2.  We know a path enters the slice from the left, so slice $i-k$ must correspond to a path configuration in which the path exits right. This path configuration must have the form W2 or W4.  If slice $i-k$ corresponds to a path configuration of the form W2, then we can repeat this argument to show that slice $i-2k$ corresponds to a path configuration of the form W2 or W4.  Since there are finitely many vertices, eventually we find that slice $i-jk$ corresponds to a path configuration of the form W4 and slices $i-(j-1)k,\ldots,i-k,i$ correspond to path configurations of the form W2, for some positive integer $j$.  Since slice $i-jk$ corresponds to a path configuration of the form W4, it must have the form H4, so it contains the tail triangle of a ribbon.  We see that slice $i-(j-1)k$ contains the head triangle of this ribbon, so this slice has the form H2. It follows that slice $i-(j-1)k$ also contains the tail triangle of a ribbon.  Repeating, we conclude that slice $i$ has the form H2.
\end{itemize}

Recall that slice $i = (j-1)k+a$ in $T$ corresponds to color $a$ in the $j$-th leftmost vertex $V$ in $\theta(T)$.  The contribution of this slice to $\spin(T)$ equals the number of column parallelograms in the slice.  Looking at the five possible forms of a slice, we see that this equals zero if slice $i$ has the form H1, which is equivalent to color $a$ being absent in $V$.  Otherwise, it equals the number of ribbons $R$ that appear in slice $i$ but whose head/tail triangles do not.  

Let $R$ be such a ribbon, and let slice $s$ be the slice that contains the tail triangle of $R$.  Since slice $a$ contains a column parallelogram of $R$ but not the head/tail triangle of $R$, we have $s < i < s+k$.  Let $b \in \{1,\ldots,k\}$ be such that $b \equiv s \mod k$, and note that $b \neq a$.  If $b < a$, then the tail triangle of $R$ appears in slice $(j-1)k+b$, so slice $(j-1)k+b$ has the form H2 or H4, so the path configuration of color $b$ in $V$ has the form W2 or W4, so color $b$ exits $V$ to the right.  If $b > a$, then the head triangle of $R$ appears in slice $(j-1)k+b$, so slice $(j-1)k+b$ has the form H2 or H5, so the path configuration of color $b$ in $V$ has the form W2 or W5, so color $b$ enters $V$ from the right.

We can now conclude
\[ \begin{aligned} 
\spin(T) &= \sum_{V} \sum_a \1_{\text{$a$ appears in $V$}} \cdot \left ( \sum_{b < a} \1_{\text{$b$ exits $V$ to the right}}+ \sum_{b > a} \1_{\text{$b$ enters $V$ from the left}} \right ) \\
&= \sum_{\alb} \left(
\# \begin{tikzpicture}[baseline=(current bounding box.center)]
\draw (0,0)--(1,0)--(1,1)--(0,1)--(0,0);
\draw[blue] (1,0.6)--(0.5,0.6);
\draw[red] (1,0.4)--(0.5,0.4);
\end{tikzpicture}
+
\# \begin{tikzpicture}[baseline=(current bounding box.center)]
\draw (0,0)--(1,0)--(1,1)--(0,1)--(0,0);
\draw[blue] (1,0.6)--(0.5,0.6);
\draw[red] (0.6,0.5)--(0.6,1);
\end{tikzpicture}
\right) 
+ \sum_{\alb} \left(
\# \begin{tikzpicture}[baseline=(current bounding box.center)]
\draw (0,0)--(1,0)--(1,1)--(0,1)--(0,0);
\draw[blue] (0,0.6)--(0.5,0.6);
\draw[red] (0,0.4)--(0.5,0.4);
\end{tikzpicture}
+
\# \begin{tikzpicture}[baseline=(current bounding box.center)]
\draw (0,0)--(1,0)--(1,1)--(0,1)--(0,0);
\draw[blue] (0.4,0)--(0.4,0.5);
\draw[red] (0,0.4)--(0.5,0.4);
\end{tikzpicture}
\right) \\
&= \sum_{\alb} \left(
\# \begin{tikzpicture}[baseline=(current bounding box.center)]
\draw (0,0)--(1,0)--(1,1)--(0,1)--(0,0);
\draw[blue] (0,0.6)--(0.5,0.6);
\draw[red] (0,0.4)--(0.5,0.4);
\end{tikzpicture}
+
\# \begin{tikzpicture}[baseline=(current bounding box.center)]
\draw (0,0)--(1,0)--(1,1)--(0,1)--(0,0);
\draw[blue] (1,0.6)--(0.5,0.6);
\draw[red] (1,0.4)--(0.5,0.4);
\end{tikzpicture}
+
\# \begin{tikzpicture}[baseline=(current bounding box.center)]
\draw (0,0)--(1,0)--(1,1)--(0,1)--(0,0);
\draw[blue] (0.4,0)--(0.4,0.5);
\draw[red] (0,0.4)--(0.5,0.4);
\end{tikzpicture}
+
\# \begin{tikzpicture}[baseline=(current bounding box.center)]
\draw (0,0)--(1,0)--(1,1)--(0,1)--(0,0);
\draw[blue] (1,0.6)--(0.5,0.6);
\draw[red] (0.6,0.5)--(0.6,1);
\end{tikzpicture}
\right)
\end{aligned} \]
where $V$ varies over the vertices in $\theta(T)$ and $a$ and $b$ vary over the colors.
\end{proof}

\begin{lem} 
Let $T$ be a vertical $k$-ribbon strip.  In the corresponding purple row $\theta(T)$, 
\[
\spin(T) = \sum_{\alb} \left(
\# \begin{tikzpicture}[baseline=(current bounding box.center)]
\draw[fill=violet] (0,0) rectangle (1,1);
\draw[blue] (0.4,0)--(0.4,0.6)--(1,0.6);
\draw[red] (0.6,0)--(0.6,1);
\end{tikzpicture}
+
\# \begin{tikzpicture}[baseline=(current bounding box.center)]
\draw[fill=violet] (0,0) rectangle (1,1);
\draw[blue] (0,0.6)--(1,0.6);
\draw[blue] (0.4,0)--(0.4,1);
\draw[red] (0.6,0)--(0.6,1);
\end{tikzpicture}
+
\# \begin{tikzpicture}[baseline=(current bounding box.center)]
\draw[fill=violet] (0,0) rectangle (1,1);
\draw[blue] (0.4,0)--(0.4,1);
\draw[red] (0,0.4)--(0.6,0.4)--(0.6,1);
\end{tikzpicture}
+
\# \begin{tikzpicture}[baseline=(current bounding box.center)]
\draw[fill=violet] (0,0) rectangle (1,1);
\draw[blue] (0.4,0)--(0.4,1);
\draw[red] (0.6,0)--(0.6,1);
\draw[red] (0,0.4)--(1,0.4);
\end{tikzpicture}
\right)
\]
where blue is color $a$ and red is color $b$.
\end{lem}

\begin{proof}
We follow the same ideas as in the proof of the previous lemma.  The fact that $T$ is a vertical ribbon strip restricts the possible forms of the slices.
\begin{enumerate}
    \item If a tail triangle appears, it must be at the bottom of the slice.
    \item If a head triangle appears, it must be at the top of the slice.
\end{enumerate}
With these restrictions in mind, we see that each slice must have one of the following five forms.
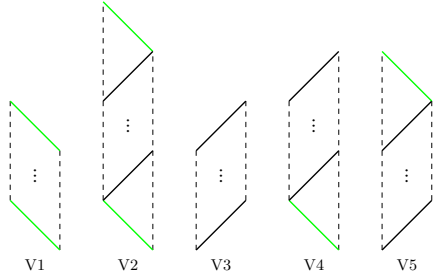
\begin{center}
\resizebox{0.4\textwidth}{!}{
\begin{tabular}{ccccccccc}
     \begin{tikzpicture}
     \draw[dashed] (0,0)--(0,2); \draw[dashed] (-1,1)--(-1,3); \node at (-0.5,1.5) {$\vdots$};
     \draw[green,thick] (0,0)--(-1,1); \draw[green,thick] (-1,3)--(0,2);
     \end{tikzpicture}
     & &
     \begin{tikzpicture}
     \draw[dashed] (0,0)--(0,4); \draw[dashed] (-1,1)--(-1,5); \node at (-0.5,2.5) {$\vdots$};
     \draw[green,thick] (0,0)--(-1,1); \draw[green,thick] (-1,5)--(0,4);
     \draw[thick] (0,4)--(-1,3); \draw[thick] (0,2)--(-1,1);
     \end{tikzpicture}
     & &
     \begin{tikzpicture}
     \draw[dashed] (0,0)--(0,2); \draw[dashed] (1,1)--(1,3); \node at (0.5,1.5) {$\vdots$};
     \draw[thick] (0,0)--(1,1); \draw[thick] (0,2)--(1,3);
     \end{tikzpicture}
     & &
     \begin{tikzpicture}
     \draw[dashed] (0,0)--(0,4); \draw[dashed] (-1,1)--(-1,3); \node at (-0.5,2.5) {$\vdots$};
     \draw[green,thick] (0,0)--(-1,1);
     \draw[thick] (0,2)--(-1,1); \draw[thick] (0,4)--(-1,3);
     \end{tikzpicture}
     & &
     \begin{tikzpicture}
     \draw[dashed] (0,0)--(0,4); \draw[dashed] (1,1)--(1,3); \node at (0.5,1.5) {$\vdots$};
     \draw[thick] (0,0)--(1,1); \draw[thick] (1,3)--(0,2); \draw[green,thick] (0,4)--(1,3);
     \end{tikzpicture}
     \\
     V1 & & V2 & & V3 & & V4 & & V5
\end{tabular}
}
\end{center}
Arguing similarly to Lemma \ref{lem:white-spin} one can show there is a one-to-one correspondence between the five possible forms of a slice in $T$ and the five possible path configurations of the corresponding color in the corresponding purple vertex of $\theta(T)$.
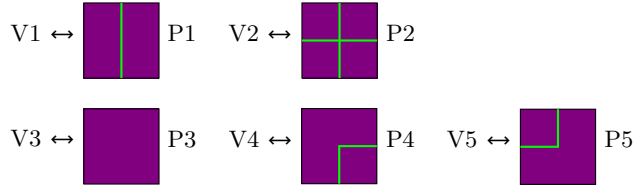
\begin{center}
\begin{tabular}{ccc}   
     V1 $\leftrightarrow$ \begin{tikzpicture}[baseline=(current bounding box.center)]
     \draw[fill=violet] (0,0) rectangle (1,1);
     \draw[green,thick] (0.5,0)--(0.5,1);
     \end{tikzpicture} P1
     &
     V2 $\leftrightarrow$ \begin{tikzpicture}[baseline=(current bounding box.center)]
     \draw[fill=violet] (0,0) rectangle (1,1);
     \draw[green,thick] (0,0.5)--(1,0.5);
     \draw[green,thick] (0.5,0)--(0.5,1);
     \end{tikzpicture} P2
     & \\ \\
     V3 $\leftrightarrow$ \begin{tikzpicture}[baseline=(current bounding box.center)]
     \draw[fill=violet] (0,0) rectangle (1,1);
     \end{tikzpicture} P3
     &
     V4 $\leftrightarrow$ \begin{tikzpicture}[baseline=(current bounding box.center)]
     \draw[fill=violet] (0,0) rectangle (1,1);
     \draw[green,thick] (0.5,0)--(0.5,0.5)--(1,0.5);
     \end{tikzpicture} P4
     &
     V5 $\leftrightarrow$ \begin{tikzpicture}[baseline=(current bounding box.center)]
     \draw[fill=violet] (0,0) rectangle (1,1);
     \draw[green,thick] (0,0.5)--(0.5,0.5)--(0.5,1);
     \end{tikzpicture} P5
\end{tabular}
\end{center}

Recall that slice $i = (j-1)k+a$ in $T$ corresponds to color $a$ in the $j$-th leftmost vertex $V$ in $\theta(T)$.  The contribution of this slice to $\spin(T)$ equals the number of column parallelograms in the slice.  Looking at the five possible forms of a slice, we see that this equals zero unless slice $i$ has the form V1, which is equivalent to the path configuration of color $a$ having form P1 in $V$, that is, color $a$ being vertical in $V$.  
In that case, the contribution to the spin equals the number of ribbons $R$ that appear in slice $i$ but whose head/tail triangles do not.  One can show that this equals the number of smaller colors $b < a$ that exits $V$ to the right plus the number of larger colors $b > a$ that enter $V$ from the left.

Fro this we conclude 
\[ \begin{aligned} 
\spin(T) &= \sum_{V} \sum_a \1_{\text{$a$ is vertical in $V$}} \cdot \left ( \sum_{b < a} \1_{\text{$b$ exits $V$ to the right}}+ \sum_{b > a} \1_{\text{$b$ enters $V$ from the left}} \right ) \\
&= \sum_{\alb} \left(
\# \begin{tikzpicture}[baseline=(current bounding box.center)]
\draw[fill=violet] (0,0) rectangle (1,1);
\draw[blue] (0.4,0)--(0.4,0.6)--(1,0.6);
\draw[red] (0.6,0)--(0.6,1);
\end{tikzpicture}
+
\# \begin{tikzpicture}[baseline=(current bounding box.center)]
\draw[fill=violet] (0,0) rectangle (1,1);
\draw[blue] (0,0.6)--(1,0.6);
\draw[blue] (0.4,0)--(0.4,1);
\draw[red] (0.6,0)--(0.6,1);
\end{tikzpicture}
\right) + \sum_{\alb} \left(
\# \begin{tikzpicture}[baseline=(current bounding box.center)]
\draw[fill=violet] (0,0) rectangle (1,1);
\draw[blue] (0.4,0)--(0.4,1);
\draw[red] (0,0.4)--(0.6,0.4)--(0.6,1);
\end{tikzpicture}
+
\# \begin{tikzpicture}[baseline=(current bounding box.center)]
\draw[fill=violet] (0,0) rectangle (1,1);
\draw[blue] (0.4,0)--(0.4,1);
\draw[red] (0.6,0)--(0.6,1);
\draw[red] (0,0.4)--(1,0.4);
\end{tikzpicture}
\right) \\
&= \sum_{\alb} \left(
\# \begin{tikzpicture}[baseline=(current bounding box.center)]
\draw[fill=violet] (0,0) rectangle (1,1);
\draw[blue] (0.4,0)--(0.4,0.6)--(1,0.6);
\draw[red] (0.6,0)--(0.6,1);
\end{tikzpicture}
+
\# \begin{tikzpicture}[baseline=(current bounding box.center)]
\draw[fill=violet] (0,0) rectangle (1,1);
\draw[blue] (0,0.6)--(1,0.6);
\draw[blue] (0.4,0)--(0.4,1);
\draw[red] (0.6,0)--(0.6,1);
\end{tikzpicture}
+
\# \begin{tikzpicture}[baseline=(current bounding box.center)]
\draw[fill=violet] (0,0) rectangle (1,1);
\draw[blue] (0.4,0)--(0.4,1);
\draw[red] (0,0.4)--(0.6,0.4)--(0.6,1);
\end{tikzpicture}
+
\# \begin{tikzpicture}[baseline=(current bounding box.center)]
\draw[fill=violet] (0,0) rectangle (1,1);
\draw[blue] (0.4,0)--(0.4,1);
\draw[red] (0.6,0)--(0.6,1);
\draw[red] (0,0.4)--(1,0.4);
\end{tikzpicture}
\right)
\end{aligned} \]
where $V$ varies over the vertices in $\theta(T)$ and $a$ and $b$ vary over the colors.
\end{proof}

We are now ready to relate $\mc{L}^S_\lm$ to $\mc{G}^{(k)}_{\lambda/\mu}$.

\begin{prop} \label{prop-L-equals-G} Let $\lm$ be the $k$-quotient of $\lambda/\mu$.  Then
\[
\mathcal{L}^S_{\lm}(X_n;Y_m;t) = t^{\square} \mathcal{G}_{\lambda/\mu}^{(k)}(X_n;Y_m;t^{1/2})
\]
for some half-integer $\square \in \frac{1}{2} \mathbb{Z}$.  In fact, in any configuration of the lattice $\latS_{n,m}(\lm)$,
\[
\square = \frac{1}{2}\sum_{\alb} \left(
\# \begin{tikzpicture}[baseline=(current bounding box.center)]
\draw[fill=violet] (0,0) rectangle (1,1);
\draw[blue] (0.4,0)--(0.4,0.6)--(1,0.6);
\draw[red] (0.6,0)--(0.6,1);
\end{tikzpicture}
+
\# \begin{tikzpicture}[baseline=(current bounding box.center)]
\draw[fill=violet] (0,0) rectangle (1,1);
\draw[blue] (0,0.6)--(1,0.6);
\draw[blue] (0.4,0)--(0.4,1);
\draw[red] (0.6,0)--(0.6,1);
\end{tikzpicture}
-
\# \begin{tikzpicture}[baseline=(current bounding box.center)]
\draw[fill=violet] (0,0) rectangle (1,1);
\draw[blue] (0.4,0)--(0.4,1);
\draw[red] (0,0.4)--(0.6,0.4)--(0.6,1);
\end{tikzpicture}
-
\# \begin{tikzpicture}[baseline=(current bounding box.center)]
\draw[fill=violet] (0,0) rectangle (1,1);
\draw[blue] (0.4,0)--(0.4,1);
\draw[red] (0.6,0)--(0.6,1);
\draw[red] (0,0.4)--(1,0.4);
\end{tikzpicture}
\right) + 
\frac{1}{2} \sum_{\alb} \left(
\# \begin{tikzpicture}[baseline=(current bounding box.center)]
\draw (0,0)--(1,0)--(1,1)--(0,1)--(0,0);
\draw[blue] (1,0.6)--(0.5,0.6);
\draw[red] (0.6,0.5)--(0.6,1);
\end{tikzpicture}
-
\# \begin{tikzpicture}[baseline=(current bounding box.center)]
\draw (0,0)--(1,0)--(1,1)--(0,1)--(0,0);
\draw[blue] (0.4,0)--(0.4,0.5);
\draw[red] (0,0.4)--(0.5,0.4);
\end{tikzpicture}
\right) .
\]

\end{prop}
\begin{proof}
Recall the bijection $\theta : \SRT_k(\lambda/\mu) \rightarrow LC(\latS_{n,m}(\lm))$.  It is enough to show that 
\[ \weight(\theta(T)) = t^{\square} t^{\spin(T)} x^{\weight(T)} y^{\weight'(T)} \]
for all $T \in \SRT_k(\lambda/\mu)$, for some half-integer $\square$ that does not depend on $T$.  Fix a super ribbon tableaux $T \in \SRT_k(\lambda/\mu)$ and let $C = \theta(T) \in LC(\latS_{n,m}(\lm))$ be the corresponding path configuration.  It is clear that the $X$ weights ($Y$ weights) match, since each ribbon labelled $a \in \mc{A}$ ($a' \in \mc{A'}$) in $T$ corresponds to a path taking a step right in the $a$-th white (purple) row of $C$.  We are left to consider the powers of $t$.  From the previous lemmas, we see that
\[
\begin{aligned}
\spin(T) &= \sum_{\alb} \left(
\# \begin{tikzpicture}[baseline=(current bounding box.center)]
\draw (0,0)--(1,0)--(1,1)--(0,1)--(0,0);
\draw[blue] (0,0.6)--(0.5,0.6);
\draw[red] (0,0.4)--(0.5,0.4);
\end{tikzpicture}
+
\# \begin{tikzpicture}[baseline=(current bounding box.center)]
\draw (0,0)--(1,0)--(1,1)--(0,1)--(0,0);
\draw[blue] (1,0.6)--(0.5,0.6);
\draw[red] (1,0.4)--(0.5,0.4);
\end{tikzpicture}
+
\# \begin{tikzpicture}[baseline=(current bounding box.center)]
\draw (0,0)--(1,0)--(1,1)--(0,1)--(0,0);
\draw[blue] (0.4,0)--(0.4,0.5);
\draw[red] (0,0.4)--(0.5,0.4);
\end{tikzpicture}
+
\# \begin{tikzpicture}[baseline=(current bounding box.center)]
\draw (0,0)--(1,0)--(1,1)--(0,1)--(0,0);
\draw[blue] (1,0.6)--(0.5,0.6);
\draw[red] (0.6,0.5)--(0.6,1);
\end{tikzpicture}
\right) \\
&\enspace +  \sum_{\alb} \left(
\# \begin{tikzpicture}[baseline=(current bounding box.center)]
\draw[fill=violet] (0,0) rectangle (1,1);
\draw[blue] (0.4,0)--(0.4,0.6)--(1,0.6);
\draw[red] (0.6,0)--(0.6,1);
\end{tikzpicture}
+
\# \begin{tikzpicture}[baseline=(current bounding box.center)]
\draw[fill=violet] (0,0) rectangle (1,1);
\draw[blue] (0,0.6)--(1,0.6);
\draw[blue] (0.4,0)--(0.4,1);
\draw[red] (0.6,0)--(0.6,1);
\end{tikzpicture}
+
\# \begin{tikzpicture}[baseline=(current bounding box.center)]
\draw[fill=violet] (0,0) rectangle (1,1);
\draw[blue] (0.4,0)--(0.4,1);
\draw[red] (0,0.4)--(0.6,0.4)--(0.6,1);
\end{tikzpicture}
+
\# \begin{tikzpicture}[baseline=(current bounding box.center)]
\draw[fill=violet] (0,0) rectangle (1,1);
\draw[blue] (0.4,0)--(0.4,1);
\draw[red] (0.6,0)--(0.6,1);
\draw[red] (0,0.4)--(1,0.4);
\end{tikzpicture}
\right) \\
&= \sum_{\alb} \left(
2 \cdot \# \begin{tikzpicture}[baseline=(current bounding box.center)]
\draw (0,0)--(1,0)--(1,1)--(0,1)--(0,0);
\draw[blue] (1,0.6)--(0.5,0.6);
\draw[red] (1,0.4)--(0.5,0.4);
\end{tikzpicture}
+
\# \begin{tikzpicture}[baseline=(current bounding box.center)]
\draw (0,0)--(1,0)--(1,1)--(0,1)--(0,0);
\draw[blue] (0.4,0)--(0.4,0.5);
\draw[red] (0,0.4)--(0.5,0.4);
\end{tikzpicture}
+
\# \begin{tikzpicture}[baseline=(current bounding box.center)]
\draw (0,0)--(1,0)--(1,1)--(0,1)--(0,0);
\draw[blue] (1,0.6)--(0.5,0.6);
\draw[red] (0.6,0.5)--(0.6,1);
\end{tikzpicture}
\right) \\
&\enspace +  \sum_{\alb} \left(
\# \begin{tikzpicture}[baseline=(current bounding box.center)]
\draw[fill=violet] (0,0) rectangle (1,1);
\draw[blue] (0.4,0)--(0.4,0.6)--(1,0.6);
\draw[red] (0.6,0)--(0.6,1);
\end{tikzpicture}
+
\# \begin{tikzpicture}[baseline=(current bounding box.center)]
\draw[fill=violet] (0,0) rectangle (1,1);
\draw[blue] (0,0.6)--(1,0.6);
\draw[blue] (0.4,0)--(0.4,1);
\draw[red] (0.6,0)--(0.6,1);
\end{tikzpicture}
+
\# \begin{tikzpicture}[baseline=(current bounding box.center)]
\draw[fill=violet] (0,0) rectangle (1,1);
\draw[blue] (0.4,0)--(0.4,1);
\draw[red] (0,0.4)--(0.6,0.4)--(0.6,1);
\end{tikzpicture}
+
\# \begin{tikzpicture}[baseline=(current bounding box.center)]
\draw[fill=violet] (0,0) rectangle (1,1);
\draw[blue] (0.4,0)--(0.4,1);
\draw[red] (0.6,0)--(0.6,1);
\draw[red] (0,0.4)--(1,0.4);
\end{tikzpicture}
\right) \\
&= 2\coinv(C) - \sum_{\alb} \left(
\# \begin{tikzpicture}[baseline=(current bounding box.center)]
\draw (0,0)--(1,0)--(1,1)--(0,1)--(0,0);
\draw[blue] (1,0.6)--(0.5,0.6);
\draw[red] (0.6,0.5)--(0.6,1);
\end{tikzpicture}
-
\# \begin{tikzpicture}[baseline=(current bounding box.center)]
\draw (0,0)--(1,0)--(1,1)--(0,1)--(0,0);
\draw[blue] (0.4,0)--(0.4,0.5);
\draw[red] (0,0.4)--(0.5,0.4);
\end{tikzpicture}
\right) \\
&\enspace + 2\coinv'(C) - \sum_{\alb} \left(
\# \begin{tikzpicture}[baseline=(current bounding box.center)]
\draw[fill=violet] (0,0) rectangle (1,1);
\draw[blue] (0.4,0)--(0.4,0.6)--(1,0.6);
\draw[red] (0.6,0)--(0.6,1);
\end{tikzpicture}
+
\# \begin{tikzpicture}[baseline=(current bounding box.center)]
\draw[fill=violet] (0,0) rectangle (1,1);
\draw[blue] (0,0.6)--(1,0.6);
\draw[blue] (0.4,0)--(0.4,1);
\draw[red] (0.6,0)--(0.6,1);
\end{tikzpicture}
-
\# \begin{tikzpicture}[baseline=(current bounding box.center)]
\draw[fill=violet] (0,0) rectangle (1,1);
\draw[blue] (0.4,0)--(0.4,1);
\draw[red] (0,0.4)--(0.6,0.4)--(0.6,1);
\end{tikzpicture}
-
\# \begin{tikzpicture}[baseline=(current bounding box.center)]
\draw[fill=violet] (0,0) rectangle (1,1);
\draw[blue] (0.4,0)--(0.4,1);
\draw[red] (0.6,0)--(0.6,1);
\draw[red] (0,0.4)--(1,0.4);
\end{tikzpicture}
\right)
\end{aligned}
\]
where 
\[
\coinv(C) = \sum_{\alb} \left(
\# \begin{tikzpicture}[baseline=(current bounding box.center)]
\draw (0,0)--(1,0)--(1,1)--(0,1)--(0,0);
\draw[blue] (1,0.6)--(0.5,0.6);
\draw[red] (1,0.4)--(0.5,0.4);
\end{tikzpicture}
+
\# \begin{tikzpicture}[baseline=(current bounding box.center)]
\draw (0,0)--(1,0)--(1,1)--(0,1)--(0,0);
\draw[blue] (1,0.6)--(0.5,0.6);
\draw[red] (0.6,0.5)--(0.6,1);
\end{tikzpicture}
\right) 
\]
is the power of $t$ coming from the white boxes in $C$ and 
\[
\coinv'(C) = \sum_{\alb} \left(
\# \begin{tikzpicture}[baseline=(current bounding box.center)]
\draw[fill=violet] (0,0) rectangle (1,1);
\draw[blue] (0.4,0)--(0.4,0.6)--(1,0.6);
\draw[red] (0.6,0)--(0.6,1);
\end{tikzpicture}
+
\# \begin{tikzpicture}[baseline=(current bounding box.center)]
\draw[fill=violet] (0,0) rectangle (1,1);
\draw[blue] (0,0.6)--(1,0.6);
\draw[blue] (0.4,0)--(0.4,1);
\draw[red] (0.6,0)--(0.6,1);
\end{tikzpicture}
\right)
\]
is the power of $t$ coming from the purple boxes in $C$.  Thus
\[ \frac{1}{2} \spin(T) + \square = \coinv(C) + \coinv'(C) \] 
where
\[
\square = \frac{1}{2}\sum_{\alb} \left(
\# \begin{tikzpicture}[baseline=(current bounding box.center)]
\draw[fill=violet] (0,0) rectangle (1,1);
\draw[blue] (0.4,0)--(0.4,0.6)--(1,0.6);
\draw[red] (0.6,0)--(0.6,1);
\end{tikzpicture}
+
\# \begin{tikzpicture}[baseline=(current bounding box.center)]
\draw[fill=violet] (0,0) rectangle (1,1);
\draw[blue] (0,0.6)--(1,0.6);
\draw[blue] (0.4,0)--(0.4,1);
\draw[red] (0.6,0)--(0.6,1);
\end{tikzpicture}
-
\# \begin{tikzpicture}[baseline=(current bounding box.center)]
\draw[fill=violet] (0,0) rectangle (1,1);
\draw[blue] (0.4,0)--(0.4,1);
\draw[red] (0,0.4)--(0.6,0.4)--(0.6,1);
\end{tikzpicture}
-
\# \begin{tikzpicture}[baseline=(current bounding box.center)]
\draw[fill=violet] (0,0) rectangle (1,1);
\draw[blue] (0.4,0)--(0.4,1);
\draw[red] (0.6,0)--(0.6,1);
\draw[red] (0,0.4)--(1,0.4);
\end{tikzpicture}
\right) + 
\frac{1}{2} \sum_{\alb} \left(
\# \begin{tikzpicture}[baseline=(current bounding box.center)]
\draw (0,0)--(1,0)--(1,1)--(0,1)--(0,0);
\draw[blue] (1,0.6)--(0.5,0.6);
\draw[red] (0.6,0.5)--(0.6,1);
\end{tikzpicture}
-
\# \begin{tikzpicture}[baseline=(current bounding box.center)]
\draw (0,0)--(1,0)--(1,1)--(0,1)--(0,0);
\draw[blue] (0.4,0)--(0.4,0.5);
\draw[red] (0,0.4)--(0.5,0.4);
\end{tikzpicture}
\right) .
\]
If we can show that $\square$ is independent of the configuration $C$, then the result follows.  It is enough to check this for two colors.  This can be shown with a standard corner flipping argument, which we verified with a computer.
\end{proof}

\pagebreak

\section{Cauchy Identity}\label{sec:Cauchy}

Recall from \cite{ABW,Lam,CGKM} that the LLT polynomials satisfy a Cauchy identity. We would like to prove a similar theorem for the supersymmetric LLT polynomials. We will do so in the style of \cite{CGKM,WHEELER2016543}.

We will need the following change of variables for the $L$ and $L'$ weights, which we represent as a gray face and light-purple face, respectively. 
\[
\begin{aligned}
&\begin{tikzpicture}[baseline = (current bounding box).center]
\draw[fill = gray] (0,0) rectangle (1,1);
\node[below] at (0.5,0) {$\I$};
\node[above] at (0.5,1) {$\K$};
\node[left] at (0,0.5) {$\J$};
\node[right] at (1,0.5) {$\L$};
\node at (0.5,0.5) {$\bar x$};
\end{tikzpicture} 
:= x^kt^{\binom{k}{2}}L^{(k)}_{\bar x}(\I,\J;\K,\L), &\begin{tikzpicture}[baseline = (current bounding box).center]
\draw[fill = violet!40!white] (0,0) rectangle (1,1);
\node[below] at (0.5,0) {$\I$};
\node[above] at (0.5,1) {$\K$};
\node[left] at (0,0.5) {$\J$};
\node[right] at (1,0.5) {$\L$};
\node at (0.5,0.5) {$\bar x$};
\end{tikzpicture} 
:= x^k L'^{(k)}_{1/x}(\I,\J;\K,\L) 
\end{aligned}
\]
Here $\bar x = \frac{1}{xt^{k-1}}$. This change of variable is chosen in part so that the gray face in which every color appears as a horizontal path and the light-purple face in which every color appears as a cross have weight one. 
\[ \begin{aligned}
&\begin{tikzpicture}[baseline = (current bounding box).center]
\draw[fill = gray] (0,0) rectangle (1,1);
\draw[ultra thick] (0,0.5)--(1,0.5);
\end{tikzpicture} 
=1,
&\begin{tikzpicture}[baseline = (current bounding box).center]
\draw[fill = violet!40!white] (0,0) rectangle (1,1);
\draw[ultra thick] (0,0.5)--(1,0.5);
\draw[ultra thick] (0.5,0)--(0.5,1);
\end{tikzpicture} 
=1
\end{aligned} \]
It will also be useful to define $P_{l_1,l_2}^{(k)}$, for $0\le l_1,l_2 \le \infty$, to be the set of $k$-tuples of partitions with $l_1$ parts whose largest part is less than or equal to $l_2$.

\subsection{Single rows}

In order to construct our Cauchy identity, we will employ infinitely long rows of vertices. To help us align the partitions on the top/bottom boundaries of a row with the columns in the row, we will mark the zero content line on the top/bottom boundaries of our rows with x's.  One can think of the zero content line as follows: if the bottom (top) boundary of the row were given by $\bm{\lambda}=\bm{0}$, then on the bottom (top) boundary of the row, to the left of the x every column would be dense with paths, and to the right of the x every column would be empty.

For the white and purple vertices, it is relatively easy to define a row of infinite length.  We start by defining the following finite length rows, where the allowed states on the top and bottom boundaries are indexed by partitions $P_{l_1,l_2}^{(k)}$. Pictorially these are given by
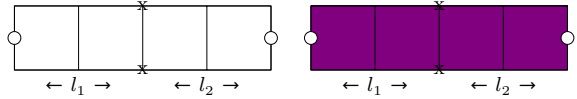
\[
\begin{aligned}
\resizebox{0.25\textwidth}{!}{
\begin{tikzpicture}[baseline = (current bounding box).center]
\draw[] (0,0) rectangle (4,1);
\draw (0,0) grid (4,1);
\draw[fill=white] (0,0.5) circle (0.1);
\draw[fill=white] (4,0.5) circle (0.1);
\node at (2,0) {x}; \node at (2,1) {x};
\node[below] at (1,0) {$\leftarrow$ $l_1$ $\rightarrow$};
\node[below] at (3,0) {$\leftarrow$ $l_2$ $\rightarrow$};
\end{tikzpicture}
} \enspace \resizebox{0.25\textwidth}{!}{
\begin{tikzpicture}[baseline = (current bounding box).center]
\draw[fill = violet] (0,0) rectangle (4,1);
\draw (0,0) grid (4,1);
\draw[fill=white] (0,0.5) circle (0.1);
\draw[fill=white] (4,0.5) circle (0.1);
\node at (2,0) {x}; \node at (2,1) {x};
\node[below] at (1,0) {$\leftarrow$ $l_1$ $\rightarrow$};
\node[below] at (3,0) {$\leftarrow$ $l_2$ $\rightarrow$};
\end{tikzpicture}
}
\end{aligned}
\]
Each row has length $l_1+l_2$ and we explicitly mark the zero content line.  We can increase $l_1$ by extending the partitions indexing the boundary states with zero parts; similarly we can increase $l_2$ by adding empty faces to the right. Note that increasing $l_1$ adds faces of the form
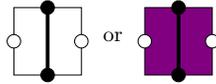
\[
\resizebox{0.2\textwidth}{!}{
\begin{tikzpicture}[baseline = (current bounding box).center]
\draw[] (0,0) rectangle (1,1);
\draw[fill=black] (0.5,0) circle (0.1);
\draw[fill=black] (0.5,1) circle (0.1);
\draw[fill=white] (0,0.5) circle (0.1);
\draw[fill=white] (1,0.5) circle (0.1);
\draw[ultra thick] (0.5,0) -- (0.5,1);
\end{tikzpicture} \text{ or }
\begin{tikzpicture}[baseline = (current bounding box).center]
\draw[fill = violet] (0,0) rectangle (1,1);
\draw[fill=black] (0.5,0) circle (0.1);
\draw[fill=black] (0.5,1) circle (0.1);
\draw[fill=white] (0,0.5) circle (0.1);
\draw[fill=white] (1,0.5) circle (0.1);
\draw[ultra thick] (0.5,0) -- (0.5,1);
\end{tikzpicture}
}
\]
on the left, while increasing $l_2$ adds faces of the form
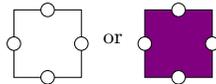
\[
\resizebox{0.2\textwidth}{!}{
\begin{tikzpicture}[baseline = (current bounding box).center]
\draw[] (0,0) rectangle (1,1);
\draw[fill=white] (0.5,0) circle (0.1);
\draw[fill=white] (0.5,1) circle (0.1);
\draw[fill=white] (0,0.5) circle (0.1);
\draw[fill=white] (1,0.5) circle (0.1);
\end{tikzpicture} \text{ or }
\begin{tikzpicture}[baseline = (current bounding box).center]
\draw[fill = violet] (0,0) rectangle (1,1);
\draw[fill=white] (0.5,0) circle (0.1);
\draw[fill=white] (0.5,1) circle (0.1);
\draw[fill=white] (0,0.5) circle (0.1);
\draw[fill=white] (1,0.5) circle (0.1);
\end{tikzpicture}
}
\]
on the right; since these vertices have weight 1, increasing $l_1$ and $l_2$ does not change the weight of the row.  In fact, we may take $l_1,l_2\to \infty$ and allow the boundary states to be indexed by partitions with infinitely many parts as long as only finitely many parts are nonzero.

For the gray and light-purple vertices, we must be slightly more careful. For the gray faces, we consider a row of finite length, such that the allowed states on the bottom are indexed by partitions in $P_{l_1,l_2}^{(k)}$, and the allowed states on the top are indexed by partitions in $P_{l_1-1,l_2}^{(k)}$. We draw this as
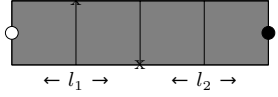
\[
\resizebox{0.25\textwidth}{!}{
\begin{tikzpicture}[baseline = (current bounding box).center]
\draw[fill = gray] (0,0) rectangle (4,1);
\draw (0,0) grid (4,1);
\draw[fill=white] (0,0.5) circle (0.1);
\draw[fill=black] (4,0.5) circle (0.1);
\node at (2,0) {x}; \node at (1,1) {x};
\node[below] at (1,0) {$\leftarrow$ $l_1$ $\rightarrow$};
\node[below] at (3,0) {$\leftarrow$ $l_2$ $\rightarrow$};
\end{tikzpicture}
}
\]
The boundary condition on the right allows us to increase $l_2$ by adding faces where every path is horizontal without changing the weight of the row. However, increasing $l_1$ by adding zero parts to the partitions does affect the weight since faces where all the paths are vertical have a nontrivial contribution due to the change of variable.  

For the light-purple faces, we consider a row of finite length, such that the allowed states on the bottom are indexed by partitions in $P_{l_1,l_2}^{(k)}$, and the allowed states on the top are indexed by partitions in $P_{l_1+1,l_2-1}^{(k)}$.  We draw this as
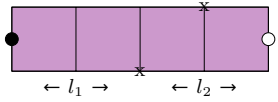
\[
\resizebox{0.25\textwidth}{!}{
\begin{tikzpicture}[baseline = (current bounding box).center]
\draw[fill = violet!40!white] (0,0) rectangle (4,1);
\draw (0,0) grid (4,1);
\draw[fill=black] (0,0.5) circle (0.1);
\draw[fill=white] (4,0.5) circle (0.1);
\node at (2,0) {x}; \node at (3,1) {x};
\node[below] at (1,0) {$\leftarrow$ $l_1$ $\rightarrow$};
\node[below] at (3,0) {$\leftarrow$ $l_2$ $\rightarrow$};
\end{tikzpicture}
}
\]
In this case, we can increase $l_1$ by adding zero parts to the partitions without changing the weight of the row, as this amounts to adding faces on the left where every color is a cross. However, increasing $l_2$ by adding empty faces on the right does affect the weight. Later, we will see that the contribution to the weight coming from increasing $l_1$ in the case of the gray faces, and the contribution to the weight coming from increasing $l_2$ in the case of the light-purple faces, cancels out in the Cauchy identity, allowing us to circumvent this issue.

\begin{remark}
Suppose the bottom boundary of a row is indexed by $\bm{\mu}$ while the top boundary is indexed by $\bm{\lambda}$. Recall that, for the white faces, in order for the row to have a nonzero weight, $\bm{\lambda}$ must be obtained from $\bm{\mu}$ by adding a horizontal strip. Similarly, for purple faces, $\bm{\lambda}$ must be obtained from $\bm{\mu}$ by adding a vertical strip.

For the gray faces, the effect of shifting the zero content line is such that, in order for the row to have a nonzero weight, $\bm{\lambda}$ must be obtained from $\bm{\mu}$ by removing a horizontal strip. For light-purple, $\bm{\lambda}$ must be obtained from $\bm{\mu}$ by removing a vertical strip.

\end{remark}

\subsection{Some partition functions}

Here we will construct certain lattice models, using the single rows above, whose partition functions will be used in our Cauchy identities. In what follows, we will always consider our partitions $\bm{\lambda}$ and $\bm{\mu}$ to be tuples of partitions, each with a infinitely many parts, only finitely many of which are nonzero. We will truncate the partitions, removing only zero parts, to limit the number of parts as needed.

Given $\bm{\lambda}$ and $\bm{\mu}$, choose positive integers $l_1,l_2$ such that each partition has less than or equal to $l_1$ nonzero parts and largest part less than or equal to $l_2$. Truncate $\bm{\lambda}$ and $\bm{\mu}$ so that they are in $P_{l_1,l_2}^{(k)}$. Recall from Section \ref{sec:VertexModel} that for the white faces we have
\[
\mathcal{L}_{\bm{\lambda}/\bm{\mu}}(X_m;t) =\resizebox{3cm}{!}{
\begin{tikzpicture}[baseline=(current bounding box.center)] 
\draw (0,0) grid (4,3);
\node[left] at (0,1.5) {$\vdots$};
\node[below] at (2,0) {$\bm{\mu}$}; \node[above] at (2,3) {$\bm{\lambda}$};
\node[] at (2,0) {x}; \node[] at (2,1) {x}; \node[] at (2,2) {x}; \node[] at (2,3) {x};
\node[left] at (0,0.5) {$x_1$}; \node[left] at (0,2.5) {$x_m$}; 
\draw[fill=white] (4,0.5) circle (0.1); \draw[fill=white] (4,1.5) circle (0.1); \draw[fill=white] (4,2.5) circle (0.1);
\draw[fill=white] (0,0.5) circle (0.1); \draw[fill=white] (0,1.5) circle (0.1); \draw[fill=white] (0,2.5) circle (0.1);
\node[below] at (1,0) {$\leftarrow\; l_1\; \rightarrow$};
\node[below] at (3,0) {$\leftarrow\; l_2\; \rightarrow$};
\end{tikzpicture} }
\]
and for the purple faces we have
\[
\mathcal{L}^P_{\bm{\lambda}/\bm{\mu}}(X_m;t) =\resizebox{3cm}{!}{
\begin{tikzpicture}[baseline=(current bounding box.center)] 
\draw[fill=violet] (0,0) rectangle (4,3);
\draw (0,0) grid (4,3);
\node[left] at (0,1.5) {$\vdots$};
\node[below] at (2,0) {$\bm{\mu}$}; \node[above] at (2,3) {$\bm{\lambda}$};
\node[] at (2,0) {x}; \node[] at (2,1) {x}; \node[] at (2,2) {x}; \node[] at (2,3) {x};
\node[left] at (0,0.5) {$x_1$}; \node[left] at (0,2.5) {$x_m$}; 
\draw[fill=white] (4,0.5) circle (0.1); \draw[fill=white] (4,1.5) circle (0.1); \draw[fill=white] (4,2.5) circle (0.1);
\draw[fill=white] (0,0.5) circle (0.1); \draw[fill=white] (0,1.5) circle (0.1); \draw[fill=white] (0,2.5) circle (0.1);
\node[below] at (1,0) {$\leftarrow\; l_1\; \rightarrow$};
\node[below] at (3,0) {$\leftarrow\; l_2\; \rightarrow$};
\end{tikzpicture} }
\]
where both are independent of the choice of $l_1$ and $l_2$. In particular, the limit as $l_1,l_2 \to \infty$ of these partition functions is well-defined.

For the gray faces, fix the number of variables $m$. This time, given $\bm{\lambda}$ and $\bm{\mu}$, choose positive integers $l_1,l_2$ such that each partition of $\bm{\lambda}$ has less than or equal to $l_1$ nonzero parts, each partition of $\bm{\mu}$ has less than or equal to $l_1-m$ nonzero parts, and each partition of both tuples has largest part less than or equal to $l_2$. Truncate the partitions so that each partition in $\bm{\lambda}\in P_{l_1,l_2}^{(k)}$ and$\bm{\mu}\in P_{l_1-m,l_2}^{(k)}$. Define
\[
\mathcal{L}^*_{\bm{\lambda}/\bm{\mu}}(X_m;t) :=\resizebox{3cm}{!}{
\begin{tikzpicture}[baseline=(current bounding box.center)] 
\draw[fill=gray] (0,0) rectangle (5,3);
\draw (0,0) grid (5,3);
\node[left] at (0,1.5) {$\vdots$};
\node[below] at (2.5,-0.5) {$\bm{\lambda}$}; \node[above] at (2.5,3) {$\bm{\mu}$};
\node[] at (4,0) {x}; \node[] at (3,1) {x}; \node[] at (2,2) {x}; \node[] at (1,3) {x};
\node[left] at (0,0.5) {$\bar x_1$}; \node[left] at (0,2.5) {$\bar x_m$}; 
\draw[fill=black] (5,0.5) circle (0.1); \draw[fill=black] (5,1.5) circle (0.1); \draw[fill=black] (5,2.5) circle (0.1);
\draw[fill=white] (0,0.5) circle (0.1); \draw[fill=white] (0,1.5) circle (0.1); \draw[fill=white] (0,2.5) circle (0.1);
\node[below] at (2,0) {$\leftarrow\; l_1\; \rightarrow$};
\node[below] at (4.5,0) {$\leftarrow\; l_2\; \rightarrow$};
\end{tikzpicture} 
}
\]
We have the following proposition.
\begin{prop}\label{prop:Lstar}
\[
\mathcal{L}^*_{\bm{\lambda}/\bm{\mu}}(X_m;t) = (x_1\ldots x_m)^{(l_1-m)k}(x^{\rho_m})^k t^{(ml_1-\binom{m+1}{2})\binom{k}{2}} t^{d(\bm{\lambda},\bm{\mu})}\mathcal{L}_{\bm{\lambda}/\bm{\mu}}(X_m;t)
\]
where $d(\bm{\lambda},\bm{\mu})$ and $\mathcal{L}_{\bm{\lambda}/\bm{\mu}}$ are independent of $l_1$ and $l_2$. Furthermore, $d(\bm{\lambda},\bm{\mu})$ is given by
\[
d(\bm{\lambda},\bm{\mu}) = \sum_{a<b} \# \{i,j | \tilde\mu^{(a)}_j-j>\tilde\mu^{(b)}_i-i\} - \sum_{a<b} \# \{i,j | \lambda^{(a)}_j-j>\lambda^{(b)}_i-i\}
\]
where $\tilde \mu =(\underbrace{l_2,\ldots,l_2}_{m},\mu)$.
\end{prop}
The proof is essentially identical to that given in Prop 6.11 of \cite{CGKM} for which this is a slight generalization.  Note that $\mathcal{L}^*$ is independent of $l_2$ and we may take $l_2\to\infty$.

For the light-purple faces, we again fix the number of variables $m$. Now given $\bm{\lambda}$ and $\bm{\mu}$ let $l_1, l_2$ be positive integers such that the number of nonzero parts of every partition in $\bm{\lambda}$ and $\bm{\mu}$ is less than or equal to $l_1$, the largest part of every partition in $\bm{\lambda}$ is less than or equal to $l_2$, and the largest part of every partition in $\bm{\mu}$ is less than or equal to $l_2-m$. Truncate the partitions so that $\bm{\lambda}\in P_{l_1,l_2}^{(k)}$, and $\bm{\mu}\in P_{l_1+m,l_2-m}^{(k)}$. Define
\[
(\mathcal{L}^P)^*_{\bm{\lambda}/\bm{\mu}}(X_m;t) :=\resizebox{3cm}{!}{
\begin{tikzpicture}[baseline=(current bounding box.center)] 
\draw[fill=violet!40!white] (0,0) rectangle (5,3);
\draw (0,0) grid (5,3);
\node[left] at (0,1.5) {$\vdots$};
\node[below] at (2.5,-0.5) {$\bm{\lambda}$}; \node[above] at (2.5,3) {$\bm{\mu}$};
\node[] at (4,3) {x}; \node[] at (3,2) {x}; \node[] at (2,1) {x}; \node[] at (1,0) {x};
\node[left] at (0,0.5) {$\bar x_1$}; \node[left] at (0,2.5) {$\bar x_m$}; 
\draw[fill=white] (5,0.5) circle (0.1); \draw[fill=white] (5,1.5) circle (0.1); \draw[fill=white] (5,2.5) circle (0.1);
\draw[fill=black] (0,0.5) circle (0.1); \draw[fill=black] (0,1.5) circle (0.1); \draw[fill=black] (0,2.5) circle (0.1);
\node[below] at (0.5,0) {$\leftarrow\; l_1\; \rightarrow$};
\node[below] at (3,0) {$\leftarrow\; l_2\; \rightarrow$};
\end{tikzpicture} 
}
\]

We would like to be able to write $(\mathcal{L}^P)^*$ in terms of $\mathcal{L}^P$. In order to do so we will prove a series of lemmas. 

Let $\bm{\lambda} = (\lambda^{(1)},\ldots, \lambda^{(k)})$. Define the \textbf{complement} of $\bm{\lambda}$ to be
\[
\bm{\lambda}^c = ((\lambda^{(k)})^c,\ldots, (\lambda^{(1)})^c),
\]
that is, we complement each partition and reverse their order in the tuple.

\begin{lem}
Given tuples of partitions $\bm{\lambda}\in P_{l_1,l_2}^{(k)}$ and $\bm{\mu}\in P_{l_1,l_2-m}^{(k)}$. Let $\tilde {\bm{\mu}} \in P_{l_1,l_2}^{(k)}$ be the tuple of partitions one gets by adding $m$ to every part of every partition in $\bm{\mu}$. Then
\[
\mathcal{L}^P_{\tilde{\bm{\mu}}/\bm{\lambda}}(X_m;t) = t^{d_P(\bm{\lambda},\bm{\mu})} \mathcal{L}^P_{\bm{\lambda}^c/\tilde{\bm{\mu}}^c}(X_m;t)
\]
where $d_P(\bm{\lambda},\bm{\mu})$ is independent of the choice of $l_1$ and $l_2$.
\end{lem}
\begin{proof}
There is a bijection between configurations with bottom boundary $\bm{\lambda}$ and top boundary $\tilde{\bm{\mu}}$ and configurations with bottom boundary $\tilde{\bm{\mu}}^c$ and top boundary $\bm{\lambda}^c$ (where the complement is taken in an $l_1\times l_2$ box) given by drawing the corresponding paths, rotating the drawing 180 degrees, swapping the order of the colors, and reversing the order of the $x_i$. 

For example, with $\bm{\lambda} = ((2,1,0),(1,1,1))$, $\bm{\mu} = ((1,0,0),(1,1,0))$, $l_1=m=3$, $l_2=4$,  we have $\bm{\lambda}^c=((3,3,3),(4,3,2))$, $\tilde{\bm{\mu}} = ((4,3,3),(4,4,3))$, $\tilde{\bm{\mu}}^c = ((1,0,0),(1,1,0))$. For a particular configuration we would map
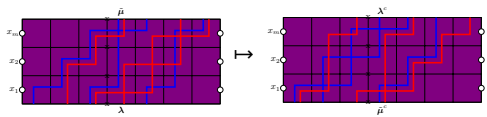
\[
\resizebox{3cm}{!}{
\begin{tikzpicture}[baseline=(current bounding box.center)] 
\draw[fill=violet] (0,0) rectangle (7,3);
\draw (0,0) grid (7,3);
\node[left] at (0,1.5) {$x_2$};
\node[] at (3,0) {x}; \node[] at (3,1) {x}; \node[] at (3,2) {x}; \node[] at (3,3) {x};
\node[left] at (0,0.5) {$x_1$}; \node[left] at (0,2.5) {$x_m$}; 
\draw[fill=white] (7,0.5) circle (0.1); \draw[fill=white] (7,1.5) circle (0.1); \draw[fill=white] (7,2.5) circle (0.1);
\draw[fill=white] (0,0.5) circle (0.1); \draw[fill=white] (0,1.5) circle (0.1); \draw[fill=white] (0,2.5) circle (0.1);
\draw[ultra thick, blue] (4.4,0)--(4.4,0.6)--(5.4,0.6)--(5.4,2.6)--(6.4,2.6)--(6.4,3);
\draw[ultra thick, blue] (2.4,0)--(2.4,0.6)--(3.4,0.6)--(3.4,2.6)--(4.4,2.6)--(4.4,3);
\draw[ultra thick, blue] (0.4,0)--(0.4,0.6)--(1.4,0.6)--(1.4,1.6)--(2.4,1.6)--(2.4,2.6)--(3.4,2.6)--(3.4,3);
\draw[ultra thick, red] (3.6,0)--(3.6,0.4)--(4.6,0.4)--(4.6,1.4)--(5.6,1.4)--(5.6,2.4)--(6.6,2.4)--(6.6,3);
\draw[ultra thick, red] (2.6,0)--(2.6,0.4)--(3.6,0.4)--(3.6,1.4)--(4.6,1.4)--(4.6,2.4)--(5.6,2.4)--(5.6,3);
\draw[ultra thick, red] (1.6,0)--(1.6,1.4)--(2.6,1.4)--(2.6,2.4)--(3.6,2.4)--(3.6,3);
\node[below] at (3.5,0) {$\bm{\lambda}$}; \node[above] at (3.5,3) {$\tilde{\bm{\mu}}$};
\end{tikzpicture} } 
\mapsto
\resizebox{3cm}{!}{
\begin{tikzpicture}[baseline=(current bounding box.center)] 
\draw[fill=violet] (0,0) rectangle (7,3);
\draw (0,0) grid (7,3);
\node[left] at (0,1.5) {$x_2$};
\node[] at (3,0) {x}; \node[] at (3,1) {x}; \node[] at (3,2) {x}; \node[] at (3,3) {x};
\node[left] at (0,0.5) {$x_1$}; \node[left] at (0,2.5) {$x_m$}; 
\draw[fill=white] (7,0.5) circle (0.1); \draw[fill=white] (7,1.5) circle (0.1); \draw[fill=white] (7,2.5) circle (0.1);
\draw[fill=white] (0,0.5) circle (0.1); \draw[fill=white] (0,1.5) circle (0.1); \draw[fill=white] (0,2.5) circle (0.1);
\draw[ultra thick, blue] (3.4,0)--(3.4,0.6)--(4.4,0.6)--(4.4,1.6)--(5.4,1.6)--(5.4,3);
\draw[ultra thick, blue] (0.4,0)--(0.4,0.6)--(1.4,0.6)--(1.4,1.6)--(2.4,1.6)--(2.4,2.6)--(3.4,2.6)--(3.4,3);
\draw[ultra thick, blue] (1.4,0)--(1.4,0.6)--(2.4,0.6)--(2.4,1.6)--(3.4,1.6)--(3.4,2.6)--(4.4,2.6)--(4.4,3);
\draw[ultra thick, red] (3.6,0)--(3.6,0.4)--(4.6,0.4)--(4.6,1.4)--(5.6,1.4)--(5.6,2.4)--(6.6,2.4)--(6.6,3);
\draw[ultra thick, red] (2.6,0)--(2.6,0.4)--(3.6,0.4)--(3.6,2.4)--(4.6,2.4)--(4.6,3);
\draw[ultra thick, red] (0.6,0)--(0.6,0.4)--(1.6,0.4)--(1.6,2.4)--(2.6,2.4)--(2.6,3);
\node[above] at (3.5,3) {$\bm{\lambda}^c$}; \node[below] at (3.5,0) {$\tilde{\bm{\mu}}^c$};
\end{tikzpicture} } 
\]

It is easy that under this bijection the $x$ weight does not change as horizontal steps remain horizontal steps. A corner flipping argument shows that the difference in the power of $t$ before and after the mapping does not depend on the configuration (see for example Lemma 6.6 of \cite{CGKM}).  This shows that
\[
\mathcal{L}^P_{\tilde{\bm{\mu}}/\bm{\lambda}}(X_m;t) = t^{d_P(\bm{\lambda},\bm{\mu})} \mathcal{L}^P_{\bm{\lambda}^c/\tilde{\bm{\mu}}^c}(X_m;t).
\]
Note that increasing $l_1$ by adding zero parts to $\bm{\lambda}$ and parts of size $m$ to $\tilde{\bm{\mu}}$ does not change the power of $t$ on either side of the bijection as this only adds paths that staircase. Similarly increasing $l_2$ by adding empty columns does not effect the power of $t$ on either side. Thus $d_P(\bm{\lambda},\bm{\mu})$ is independent of $l_1$ and $l_2$.
\end{proof}

In fact, we can calculate $d_P(\bm{\lambda},\bm{\mu})$ explicitly. 
\begin{lem}\label{lem:dp}
Let $\bm{\lambda},\bm{\mu}$ be as in the previous lemma.  Then
\[
d_P(\bm{\lambda},\bm{\mu}) = \sum_{a<b} \#\{i,j| \mu^{(a)}_j-j > \mu^{(b)}_i-i\} - \sum_{a<b} \#\{i,j| \lambda^{(a)}_j-j > \lambda^{(b)}_i-i\}.
\]
\end{lem}
\begin{proof}
It is enough to compute $d_P$ for two colors. First let's assume that $\bm{\mu}$ is empty. In this case, every part in $\tilde{\bm{\mu}}$ has size $m$. We can calculate $d_P(\bm{\lambda},\bm{0})$ using any choice of configuration. We will pick the configuration of $\tilde{\bm{\mu}}/\bm{\lambda}$ such all the paths are as low as possible, call it $T$. In this case, each path will begin as a staircase going right until it reaches the column in which it ends, and will then travel vertically to its end point. Consider a single path of the first color (blue) corresponding to the $j$-th row of the skew-shape. For it to contribute a power of $t$ a path of the second color (red) must travel vertically in a face in which the first color exits right. Suppose we have such a red path, corresponding to the $i$-th row. As the paths only travel vertically in the column in which they end, this means that the blue path must end to the right of the red path, that is, $j<i$. Further, in order for the red path to cross the blue path while traveling vertically it's staircase must be below the blue staircase, so the blue path must start to the red paths left. That is, $\lambda^{(1)}_j - j < \lambda^{(2)}_i-i$. We see that
\[
\text{coinv}'(T) = \#\{j<i|\lambda^{(1)}_j - j \le \lambda^{(2)}_i-i\}.
\]

Using our mapping, the configuration $T$ gets mapped to a configuration $T'$ of $\bm{\lambda}^c/\tilde{\bm{\mu}}^c$ in which all the paths are as high as possible. In this case, the paths all begin travelling vertically and then staircase to thier end point. Similar reasoning shows that for the $j$-th blue path to exit right while the $i$-th red path is vertical, we must have that the blue path begins in the same column or to the left of the red path, and the blue path ends to the right of red path. This gives
\[
\begin{aligned}
\text{coinv}'(T') = & \#\{j\ge i|(\lambda^c)^{(1)}_j-j>(\lambda^c)^{(2)}_i-i\} \\
= & \#\{j\ge i| \lambda^{(1)}_j - j > \lambda^{(2)}_i-i\}.
\end{aligned}
\]
From this we find
\[
\begin{aligned}
d_P(\bm{\lambda},\bm{0}) = &  \#\{j<i|\lambda^{(1)}_j - j \le \lambda^{(2)}_i-i\} - \#\{j\ge i| \lambda^{(1)}_j - j > \lambda^{(2)}_i-i\}\\
 = & \#\{j<i|\lambda^{(1)}_j - j \le \lambda^{(2)}_i-i\} + \#\{j<i|\lambda^{(1)}_j-j>\lambda^{(2)}_i-i\} \\
 & -\#\{j<i|\lambda^{(1)}_j-j>\lambda^{(2)}_i-i\} - \#\{j\ge i| \lambda^{(1)}_j - j > \lambda^{(2)}_i-i\} \\
 = & \#\{j<i\} - \#\{i,j| \lambda^{(1)}_j-j > \lambda^{(2)}_i-i\}
\end{aligned}
\]
Noting that $\#\{i,j| \mu^{(1)}_j-j > \mu^{(2)}_i-i\} = \#\{j<i\}$ when $\bm{\mu}=\bm{0}$, and summing over all pairs of colors $a<b$ gives the result.

To prove the general case, let $\bm{\lambda}$ and $\bm{\mu}$ be as in the statement of the lemma. Consider a lattice with $n+m$ rows.  Let $\bm{\nu} = \bm{0}$, and $\tilde{\bm{\nu}}$ be such that every part has length $n+m$. From the above calculation we know that
\[
d_P(\bm{\lambda},\bm{0}) = \sum_{a<b}\#\{j<i\} - \sum_{a<b}\#\{i,j| \lambda^{(a)}_j-j > \lambda^{(b)}_i-i\}.
\]
This can be calculated using any configuration of $\bm{\lambda}/\tilde{\bm{\nu}}$. Let us choose the configuration such that the top boundary of the $m$-th row is given by $\tilde{\bm{\mu}}$. Then the contribution the the change in power of $t$ from the rows above the $m$-th row is given by 
\[
d_P(\tilde{\bm{\mu}},\bm{0}) = \sum_{a<b}\#\{j<i\} - \sum_{a<b}\#\{i,j| \tilde\mu^{(a)}_j-j > \tilde\mu^{(b)}_i-i\}
\]
while the contribution from the $m$-th row and below is given by $d_P(\bm{\lambda},\bm{\mu})$. Since the contribution from the two pieces must equal the overall change in power of $t$, we see that
\[
\begin{aligned}
d_P(\bm{\lambda},\bm{\mu}) = & d_P(\bm{\lambda},\bm{0}) - d_P(\tilde{\bm{\mu}},\bm{0}) \\
= & \sum_{a<b}\#\{i,j| \tilde\mu^{(a)}_j-j > \tilde\mu^{(b)}_i-i\} -  \sum_{a<b}\#\{i,j| \lambda^{(a)}_j-j > \lambda^{(b)}_i-i\} \\
= & \sum_{a<b}\#\{i,j| \mu^{(a)}_j-j > \mu^{(b)}_i-i\} -  \sum_{a<b}\#\{i,j| \lambda^{(a)}_j-j > \lambda^{(b)}_i-i\}
\end{aligned}
\]
as desired.
\end{proof}

\begin{lem}
Let $\bm{\lambda},\bm{\mu},m,l_1,l_2$ be as in the previous lemma. We have
\[
\resizebox{3cm}{!}{
\begin{tikzpicture}[baseline=(current bounding box.center)] 
\draw[fill=violet!40!white] (0,0) rectangle (5,3);
\draw (0,0) grid (5,3);
\node[left] at (0,1.5) {$\vdots$};
\node[below] at (2.5,-0.5) {$\bm{\lambda}$}; \node[above] at (2.5,3) {$\tilde{\bm{\mu}}$};
\node[] at (3,0) {x}; \node[] at (3,1) {x}; \node[] at (3,2) {x}; \node[] at (3,3) {x};
\node[left] at (0,0.5) {$\bar x_1$}; \node[left] at (0,2.5) {$\bar x_m$}; 
\draw[fill=white] (5,0.5) circle (0.1); \draw[fill=white] (5,1.5) circle (0.1); \draw[fill=white] (5,2.5) circle (0.1);
\draw[fill=white] (0,0.5) circle (0.1); \draw[fill=white] (0,1.5) circle (0.1); \draw[fill=white] (0,2.5) circle (0.1);
\node[below] at (1.5,0) {$\leftarrow\; l_1\; \rightarrow$};
\node[below] at (4,0) {$\leftarrow\; l_2\; \rightarrow$};
\end{tikzpicture} 
} 
=(x_1\ldots x_m)^{k(l_1+l_2)} t^{d_P(\bm{\lambda},\bm{\mu})} \mathcal{L}^P_{\bm{\lambda}^c/\tilde{\bm{\mu}}^c}(X_m^{-1};t) 
\]
\end{lem}
\begin{proof}
To change from purple face to light-purple faces we take $x\mapsto \frac{1}{x}$ and multiply every face by $x^k$. Along with the previous lemma this gives
\[
\resizebox{3cm}{!}{
\begin{tikzpicture}[baseline=(current bounding box.center)] 
\draw[fill=violet!40!white] (0,0) rectangle (5,3);
\draw (0,0) grid (5,3);
\node[left] at (0,1.5) {$\vdots$};
\node[below] at (2.5,-0.5) {$\bm{\lambda}$}; \node[above] at (2.5,3) {$\tilde{\bm{\mu}}$};
\node[] at (3,0) {x}; \node[] at (3,1) {x}; \node[] at (3,2) {x}; \node[] at (3,3) {x};
\node[left] at (0,0.5) {$\bar x_1$}; \node[left] at (0,2.5) {$\bar x_m$}; 
\draw[fill=white] (5,0.5) circle (0.1); \draw[fill=white] (5,1.5) circle (0.1); \draw[fill=white] (5,2.5) circle (0.1);
\draw[fill=white] (0,0.5) circle (0.1); \draw[fill=white] (0,1.5) circle (0.1); \draw[fill=white] (0,2.5) circle (0.1);
\node[below] at (1.5,0) {$\leftarrow\; l_1\; \rightarrow$};
\node[below] at (4,0) {$\leftarrow\; l_2\; \rightarrow$};
\end{tikzpicture} 
} 
= (x_1\ldots x_m)^{k(l_1+l_2)} t^{d_P(\bm{\lambda},\bm{\mu})} \mathcal{L}^P_{\bm{\lambda}^c/\tilde{\bm{\mu}}^c}(\frac{1}{x_1},\ldots, \frac{1}{x_m};t)
\]
as desired.
\end{proof}
\begin{lem}
Let $\bm{\lambda},\bm{\mu},m,l_1,l_2$ be as in the previous lemma, except now consider $\bm{\mu}$ as an element of $P_{l_1+m,l_2}^{(k)}$ (that is, add $m$ more parts equal to zero). We have
\[
\begin{aligned}
\resizebox{3cm}{!}{
\begin{tikzpicture}[baseline=(current bounding box.center)] 
\draw[fill=violet!40!white] (0,0) rectangle (5,3);
\draw (0,0) grid (5,3);
\node[left] at (0,1.5) {$\vdots$};
\node[below] at (2.5,-0.5) {$\bm{\lambda}$}; \node[above] at (2.5,3) {$\bm{\mu}$};
\node[] at (1,0) {x}; \node[] at (2,1) {x}; \node[] at (3,2) {x}; \node[] at (4,3) {x};
\node[left] at (0,0.5) {$\bar x_1$}; \node[left] at (0,2.5) {$\bar x_m$}; 
\draw[fill=white] (5,0.5) circle (0.1); \draw[fill=white] (5,1.5) circle (0.1); \draw[fill=white] (5,2.5) circle (0.1);
\draw[fill=black] (0,0.5) circle (0.1); \draw[fill=black] (0,1.5) circle (0.1); \draw[fill=black] (0,2.5) circle (0.1);
\node[below] at (0.5,0) {$\leftarrow\; l_1\; \rightarrow$};
\node[below] at (3,0) {$\leftarrow\; l_2\; \rightarrow$};
\end{tikzpicture} 
} 
& =  (x_1\ldots x_m)^{k(l_1+l_2-m+1)}(x^{\rho_m})^k t^{d_P(\bm{\lambda},\bm{\mu})} \mathcal{L}^P_{\bm{\lambda}^c/\tilde{\bm{\mu}}^c}(X_m^{-1};t) 
\end{aligned}
\]
\end{lem}
\begin{proof}
The paths entering from the right must end packed to the right at the top. This, along with the shift right by $m$ of the zero content line, means that the top boundary is now given by $\bm{\mu}$. Adding the path entering from the right changes the weight by the factor $(x_m^{m-1}\ldots x_2)^{-k} = (x_1\ldots x_m)^{-k(m-1)}(x^{\rho_m})^k $.
\end{proof}
Finally we must relate the LLT polynomial of a skew partition with that of its complement.
\begin{lem}
Given tuples of partitions $\bm{\lambda}\in P_{l_1,l_2}^{(k)}$ and $\bm{\mu}\in P_{l_1,l_2-m}^{(k)}$. Let $\tilde {\bm{\mu}}\in P_{l_1,l_2}^{(k)}$ be the tuple of partitions one gets by adding $m$ to every part of every partition in $\bm{\mu}$. We have
\[
\mathcal{L}^P_{\bm{\lambda}/\bm{\mu}}(X_m;t) = (x_1\ldots x_m)^{kl_1} \mathcal{L}^P_{\bm{\lambda}^c/\tilde{\bm{\mu}}^c}(X_m^{-1};t).
\]
\end{lem}
\begin{proof}
Given a filling of $\bm{\lambda}/\bm{\mu}$ and integers $l_1,l_2$ as above, we can biject to a filling of $\bm{\lambda}^c/\tilde{\bm{\mu}}^c$ as follows. For each skew partition in $\bm{\lambda}/\bm{\mu}$ draw it inside an $l_1\times l_2$ box. Given any valid filling of the skew shape, go from left to right, row-by-row and fill the cells of the box with the largest available integer not already used in that row. After rotating 180 degrees the newly filled cells of the box are a valid filling for the corresponding partition in $\bm{\lambda}^c/\tilde{\bm{\mu}}^c$. For example, let $\lambda = (3,3,1,0),\mu = (1,0,0,0), l_1=m=4,l_2=6$.  Then we have $\lambda^c = (6,5,3,3)$ and $\tilde \mu^c = (2,2,2,1)$. Consider the filling below.
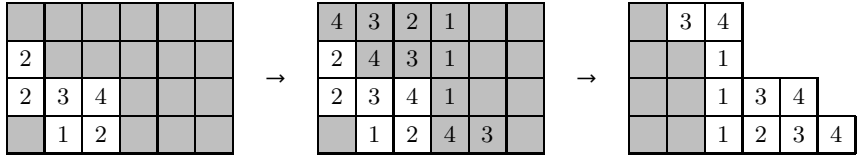
\[ \begin{ytableau}
*(lightgray) & *(lightgray) & *(lightgray) & *(lightgray) & *(lightgray) & *(lightgray)\\
2 & *(lightgray) & *(lightgray) & *(lightgray) & *(lightgray) & *(lightgray)\\
2 & 3 & 4 & *(lightgray) & *(lightgray) & *(lightgray)\\
*(lightgray) & 1 & 2 & *(lightgray) & *(lightgray) & *(lightgray)
\end{ytableau}
\quad \rightarrow \quad
\begin{ytableau}
*(lightgray) 4 & *(lightgray) 3 & *(lightgray) 2 & *(lightgray)  1 & *(lightgray) & *(lightgray) \\
2 & *(lightgray) 4 & *(lightgray) 3 & *(lightgray) 1 & *(lightgray) & *(lightgray) \\
2 & 3 & 4 & *(lightgray) 1 & *(lightgray) & *(lightgray) \\
*(lightgray) & 1 & 2 & *(lightgray) 4 & *(lightgray) 3 & *(lightgray)
\end{ytableau}
\quad \rightarrow \quad
\begin{ytableau}
 *(lightgray) & 3 & 4 & \none[] & \none[] & \none[] \\
 *(lightgray)& *(lightgray) & 1 & \none[] & \none[] & \none[] \\
 *(lightgray) & *(lightgray) & 1 & 3 & 4 & \none[] \\
 *(lightgray) & *(lightgray) & 1 & 2 & 3 & 4
\end{ytableau}
\]
Note that under this map the $x$ weights transform as $x^T\mapsto (x_1\ldots x_m)^{kl_1} (x^T)^{-1}$. We are left to determine what happens to the powers of $t$. It is easy to check that in terms of lattice paths flipping a corner of color $a$ up (down) on one side of the mapping corresponds to flipping a corner of color $k-a+1$ down (up) on the other. As the space of configurations on both sides is connected under such flips, a corner flipping argument (see Lemma 6.7 in \cite{CGKM}) shows the difference in the powers of $t$ does not depend on the configuration. Thus there is some overall power of $t$ difference, call it $\tilde d_p(\bm{\lambda},\bm{\mu})$. We need only to compute the difference in the power of $t$ for a specific choice of configurations to compute $\tilde d_p(\bm{\lambda},\bm{\mu})$. A similar argument to that of the proof of Lemma \ref{lem:dp} shows $\tilde d_p(\bm{\lambda},\bm{\mu}) = 0$.
\end{proof}

Combining all the above lemmas gives the following proposition.
\begin{prop}\label{prop:Lpstar}
\[
\begin{aligned}
(\mathcal{L}^P)^*_{\bm{\lambda}/\bm{\mu}}(X_m;t) = & (x_1\ldots x_m)^{k(l_2-m+1)}(x^{\rho_m})^k t^{d_P(\bm{\lambda},\bm{\mu})} \mathcal{L}^P_{\bm{\lambda}/\bm{\mu}}(X_m;t)
\end{aligned}
\]
where the whole thing is independent of $l_1$, and $d_P(\bm{\lambda},\bm{\mu})$ and $\mathcal{L}^P_{\bm{\lambda}/\bm{\mu}}$ are also independent of $l_2$.
\end{prop}

\subsection{Cauchy Identities}
Using the above we will now prove several Cauchy identities for the $\mathcal{L}$ and $\mathcal{L}^P$.

\begin{prop}\label{cauchy2}
Let $\bm{\mu}$ and $\bm{\nu}$ be tuples of partitions each with infinitely many parts only finitely many of which are nonzero. Then
\[
\sum_{\bm{\lambda}} t^{d(\bm{\mu},\bm{\lambda})}\mc{L}_{\bm{\mu}/\bm{\lambda}}(Y_m;t) \mc{L}^P_{\bm{\nu}/\bm{\lambda}}(X_n;t) = \left(\prod_{i,j}\prod_{l=0}^{k-1}(1+x_iy_jt^l)^{-1} \right) \sum_{\bm{\lambda}}t^{d(\bm{\lambda},\bm{\nu})} \mc{L}^P_{\bm{\lambda}/\bm{\mu}}(X_n;t) \mc{L}_{\bm{\lambda}/\bm{\nu}}(Y_m;t)
\]
\end{prop}
\begin{proof}
Given $\bm{\mu}$ and $\bm{\nu}$, choose positive integer $l_1$ and $l_2$ such that maximum number of nonzero parts of a partition in $\bm{\nu}$ is less than $l_1-m$ and the largest part of any partition in $\bm{\nu}$ is less than $l_2+m$. Note this ensures that $l_1$ is greater than the maximum number of nonzero parts of a partition in $\bm{\mu}$ and that $l_2$ is greater than the largest part of any partition in $\bm{\mu}$. Truncate the partitions so that $\bm{\mu}\in P_{l_1,l_2}^{(k)}$ and $\bm{\nu}\in P_{l_1-m,l_2+m}^{(k)}$. Consider the following partition function.
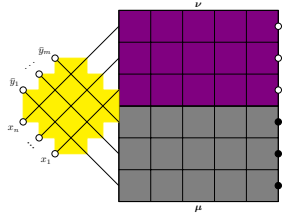
\[
\resizebox{0.25\textwidth}{!}{
$
\begin{tikzpicture}[baseline = (current bounding box).center]
\draw[fill=gray] (0,0) rectangle (5,3);
\draw[step=1.0] (0,0) grid (5,3);
\draw[fill=violet] (0,3) rectangle (5,6);
\draw[step=1.0] (0,3) grid (5,6);
\node[below] at (2.5,0) {$\bm{\mu}$};
\node[above] at (2.5,6) {$\bm{\nu}$};
\draw[fill=black] (5,0.5) circle (0.1);
\draw[fill=black] (5,1.5) circle (0.1);
\draw[fill=black] (5,2.5) circle (0.1);
\draw[fill=white] (5,3.5) circle (0.1);
\draw[fill=white] (5,4.5) circle (0.1);
\draw[fill=white] (5,5.5) circle (0.1);
\draw[yellow,fill=yellow] (-3,2.5)--(-3,3.5)--(-2.5,3.5)--(-2.5,4)--(-2,4)--(-2,4.5)--(-1,4.5)--(-1,4)--(-0.5,4)--(-0.5,3.5)--(0,3.5)--(0,2.5)--(-0.5,2.5)--(-0.5,2)--(-1,2)--(-1,1.5)--(-2,1.5)--(-2,2)--(-2.5,2)--(-2.5,2.5)--cycle;
\draw (-3,3.5)--(0,0.5); \draw[fill=white] (-3,3.5) circle (0.1); \node[above left] at (-3,3.5) {$\bar{y}_1$};
\draw (-2.5,4)--(0,1.5); \draw[fill=white] (-2.5,4) circle (0.1); \node[above left] at (-2.5,4) {$\iddots$};
\draw (-2,4.5)--(0,2.5); \draw[fill=white] (-2,4.5) circle (0.1); \node[above left] at (-2,4.5) {$\bar{y}_m$};
\draw (-2,1.5)--(0,3.5); \draw[fill=white] (-2,1.5) circle (0.1); \node[below left] at (-2,1.5) {$x_1$};
\draw (-2.5,2)--(0,4.5); \draw[fill=white] (-2.5,2) circle (0.1); \node[below left] at (-2.5,2) {$\ddots$};
\draw (-3,2.5)--(0,5.5); \draw[fill=white] (-3,2.5) circle (0.1); \node[below left] at (-3,2.5) {$x_n$};
\end{tikzpicture}
$
}
\]
This can be split into three pieces as follows.
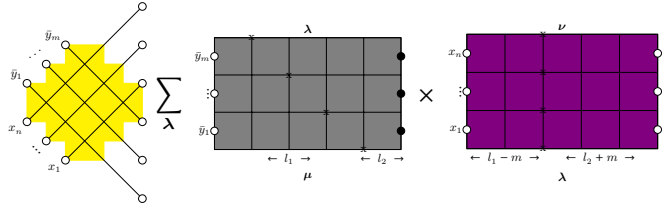
\[
\resizebox{2cm}{!}{
\begin{tikzpicture}[baseline = (current bounding box).center]
\draw[yellow,fill=yellow] (-3,2.5)--(-3,3.5)--(-2.5,3.5)--(-2.5,4)--(-2,4)--(-2,4.5)--(-1,4.5)--(-1,4)--(-0.5,4)--(-0.5,3.5)--(0,3.5)--(0,2.5)--(-0.5,2.5)--(-0.5,2)--(-1,2)--(-1,1.5)--(-2,1.5)--(-2,2)--(-2.5,2)--(-2.5,2.5)--cycle;
\draw (-3,3.5)--(0,0.5); \draw[fill=white] (-3,3.5) circle (0.1); \draw[fill=white] (0,0.5) circle (0.1); \node[above left] at (-3,3.5) {$\bar{y}_1$};
\draw (-2.5,4)--(0,1.5); \draw[fill=white] (-2.5,4) circle (0.1); \draw[fill=white] (0,1.5) circle (0.1); \node[above left] at (-2.5,4) {$\iddots$};
\draw (-2,4.5)--(0,2.5); \draw[fill=white] (-2,4.5) circle (0.1); \draw[fill=white] (0,2.5) circle (0.1); \node[above left] at (-2,4.5) {$\bar{y}_m$};
\draw (-2,1.5)--(0,3.5); \draw[fill=white] (-2,1.5) circle (0.1); \draw[fill=white] (0,3.5) circle (0.1); \node[below left] at (-2,1.5) {$x_1$};
\draw (-2.5,2)--(0,4.5); \draw[fill=white] (-2.5,2) circle (0.1); \draw[fill=white] (0,4.5) circle (0.1); \node[below left] at (-2.5,2) {$\ddots$};
\draw (-3,2.5)--(0,5.5); \draw[fill=white] (-3,2.5) circle (0.1); \draw[fill=white] (0,5.5) circle (0.1); \node[below left] at (-3,2.5) {$x_n$};
\end{tikzpicture}
}
\sum_{\bm{\lambda}} 
\resizebox{3cm}{!}{
\begin{tikzpicture}[baseline=(current bounding box.center)] 
\draw[fill=gray] (0,0) rectangle (5,3);
\draw (0,0) grid (5,3);
\node[left] at (0,1.5) {$\vdots$};
\node[below] at (2.5,-0.5) {$\bm{\mu}$}; \node[above] at (2.5,3) {$\bm{\lambda}$};
\node[] at (4,0) {x}; \node[] at (3,1) {x}; \node[] at (2,2) {x}; \node[] at (1,3) {x};
\node[left] at (0,0.5) {$\bar y_1$}; \node[left] at (0,2.5) {$\bar y_m$}; 
\draw[fill=black] (5,0.5) circle (0.1); \draw[fill=black] (5,1.5) circle (0.1); \draw[fill=black] (5,2.5) circle (0.1);
\draw[fill=white] (0,0.5) circle (0.1); \draw[fill=white] (0,1.5) circle (0.1); \draw[fill=white] (0,2.5) circle (0.1);
\node[below] at (2,0) {$\leftarrow\; l_1\; \rightarrow$};
\node[below] at (4.5,0) {$\leftarrow\; l_2\; \rightarrow$};
\end{tikzpicture} 
}
\times 
\resizebox{3cm}{!}{
\begin{tikzpicture}[baseline=(current bounding box.center)] 
\draw[fill=violet] (0,0) rectangle (5,3);
\draw (0,0) grid (5,3);
\node[left] at (0,1.5) {$\vdots$};
\node[below] at (2.5,-0.5) {$\bm{\lambda}$}; \node[above] at (2.5,3) {$\bm{\nu}$};
\node[] at (2,0) {x}; \node[] at (2,1) {x}; \node[] at (2,2) {x}; \node[] at (2,3) {x};
\node[left] at (0,0.5) {$x_1$}; \node[left] at (0,2.5) {$x_n$}; 
\draw[fill=white] (5,0.5) circle (0.1); \draw[fill=white] (5,1.5) circle (0.1); \draw[fill=white] (5,2.5) circle (0.1);
\draw[fill=white] (0,0.5) circle (0.1); \draw[fill=white] (0,1.5) circle (0.1); \draw[fill=white] (0,2.5) circle (0.1);
\node[below] at (1,0) {$\leftarrow\; l_1-m\; \rightarrow$};
\node[below] at (3.5,0) {$\leftarrow\; l_2+m\; \rightarrow$};
\end{tikzpicture} }
\]
From the previous subsection, we know every piece is independent of $l_2$, so we may take $l_2\to \infty$. Here the crosses have weight one. Using the YBE to move the crosses to the other side gives
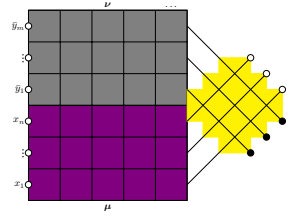
\[
\resizebox{0.25\textwidth}{!}{
$
\begin{tikzpicture}[baseline = (current bounding box).center]
\draw[fill=violet] (0,0) rectangle (5,3);
\draw[step=1.0] (0,0) grid (5,3);
\draw[fill=gray] (0,3) rectangle (5,6);
\draw[step=1.0] (0,3) grid (5,6);
\node[below] at (2.5,0) {$\bm{\mu}$};
\node[above] at (2.5,6) {$\bm{\nu}$};
\draw[fill=white] (0,0.5) circle (0.1); \node[left] at (0,0.5) {$x_1$};
\draw[fill=white] (0,1.5) circle (0.1); \node[left] at (0,1.5) {$\vdots$};
\draw[fill=white] (0,2.5) circle (0.1); \node[left] at (0,2.5) {$x_n$};
\draw[fill=white] (0,3.5) circle (0.1); \node[left] at (0,3.5) {$\bar{y}_1$};
\draw[fill=white] (0,4.5) circle (0.1); \node[left] at (0,4.5) {$\vdots$};
\draw[fill=white] (0,5.5) circle (0.1); \node[left] at (0,5.5) {$\bar{y}_m$};
\node[above] at (4.5,6) {$\ldots$};
\draw[yellow,fill=yellow] (5,2.5)--(5,3.5)--(5.5,3.5)--(5.5,4)--(6,4)--(6,4.5)--(7,4.5)--(7,4)--(7.5,4)--(7.5,3.5)--(8,3.5)--(8,2.5)--(7.5,2.5)--(7.5,2)--(7,2)--(7,1.5)--(6,1.5)--(6,2)--(5.5,2)--(5.5,2.5)--cycle;
\draw (8,3.5)--(5,0.5); \draw[fill=white] (8,3.5) circle (0.1); 
\draw (7.5,4)--(5,1.5); \draw[fill=white] (7.5,4) circle (0.1); 
\draw (7,4.5)--(5,2.5); \draw[fill=white] (7,4.5) circle (0.1); 
\draw (7,1.5)--(5,3.5); \draw[fill=black] (7,1.5) circle (0.1); 
\draw (7.5,2)--(5,4.5); \draw[fill=black] (7.5,2) circle (0.1); 
\draw (8,2.5)--(5,5.5); \draw[fill=black] (8,2.5) circle (0.1); 
\end{tikzpicture}
$
}
\]
Since we have taken $l_2\to\infty$ and paths cannot travel horizontally across a purple face, we know that the paths originating from the bottom boundary must exit from the gray faces at the right boundary.  Splitting this into parts, we get
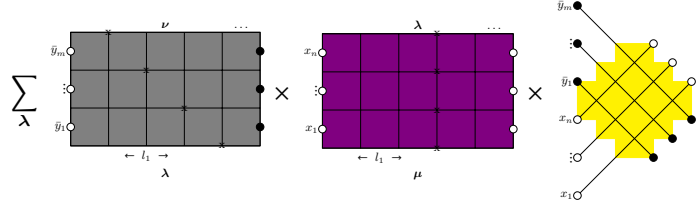
\[
\sum_{\bm{\lambda}} 
\resizebox{3cm}{!}{
\begin{tikzpicture}[baseline=(current bounding box.center)] 
\draw[fill=gray] (0,0) rectangle (5,3);
\draw (0,0) grid (5,3);
\node[left] at (0,1.5) {$\vdots$};
\node[below] at (2.5,-0.5) {$\bm{\lambda}$}; \node[above] at (2.5,3) {$\bm{\nu}$};
\node[] at (4,0) {x}; \node[] at (3,1) {x}; \node[] at (2,2) {x}; \node[] at (1,3) {x};
\node[left] at (0,0.5) {$\bar y_1$}; \node[left] at (0,2.5) {$\bar y_m$}; 
\draw[fill=black] (5,0.5) circle (0.1); \draw[fill=black] (5,1.5) circle (0.1); \draw[fill=black] (5,2.5) circle (0.1);
\draw[fill=white] (0,0.5) circle (0.1); \draw[fill=white] (0,1.5) circle (0.1); \draw[fill=white] (0,2.5) circle (0.1);
\node[below] at (2,0) {$\leftarrow\; l_1\; \rightarrow$};
\node[above] at (4.5,3) {$\ldots$};
\end{tikzpicture} 
}
\times 
\resizebox{3cm}{!}{
\begin{tikzpicture}[baseline=(current bounding box.center)] 
\draw[fill=violet] (0,0) rectangle (5,3);
\draw (0,0) grid (5,3);
\node[left] at (0,1.5) {$\vdots$};
\node[below] at (2.5,-0.5) {$\bm{\mu}$}; \node[above] at (2.5,3) {$\bm{\lambda}$};
\node[] at (3,0) {x}; \node[] at (3,1) {x}; \node[] at (3,2) {x}; \node[] at (3,3) {x};
\node[left] at (0,0.5) {$x_1$}; \node[left] at (0,2.5) {$x_n$}; 
\draw[fill=white] (5,0.5) circle (0.1); \draw[fill=white] (5,1.5) circle (0.1); \draw[fill=white] (5,2.5) circle (0.1);
\draw[fill=white] (0,0.5) circle (0.1); \draw[fill=white] (0,1.5) circle (0.1); \draw[fill=white] (0,2.5) circle (0.1);
\node[below] at (1.5,0) {$\leftarrow\; l_1\; \rightarrow$};
\node[above] at (4.5,3) {$\ldots$};
\end{tikzpicture} }
\times
\resizebox{2cm}{!}{
\begin{tikzpicture}[baseline = (current bounding box).center]
\draw[yellow,fill=yellow] (5,2.5)--(5,3.5)--(5.5,3.5)--(5.5,4)--(6,4)--(6,4.5)--(7,4.5)--(7,4)--(7.5,4)--(7.5,3.5)--(8,3.5)--(8,2.5)--(7.5,2.5)--(7.5,2)--(7,2)--(7,1.5)--(6,1.5)--(6,2)--(5.5,2)--(5.5,2.5)--cycle;
\draw (8,3.5)--(5,0.5); \draw[fill=white] (8,3.5) circle (0.1); \draw[fill=white] (5,0.5) circle (0.1);
\draw (7.5,4)--(5,1.5); \draw[fill=white] (7.5,4) circle (0.1); \draw[fill=white] (5,1.5) circle (0.1);
\draw (7,4.5)--(5,2.5); \draw[fill=white] (7,4.5) circle (0.1); \draw[fill=white] (5,2.5) circle (0.1);
\draw (7,1.5)--(5,3.5); \draw[fill=black] (7,1.5) circle (0.1); \draw[fill=black] (5,3.5) circle (0.1); 
\draw (7.5,2)--(5,4.5); \draw[fill=black] (7.5,2) circle (0.1); \draw[fill=black] (5,4.5) circle (0.1); 
\draw (8,2.5)--(5,5.5); \draw[fill=black] (8,2.5) circle (0.1); \draw[fill=black] (5,5.5) circle (0.1); 
\node[left] at (5,5.5) {$\bar y_m$}; \node[left] at (5,4.5) {$\vdots$}; \node[left] at (5,3.5) {$\bar y_1$}; \node[left] at (5,2.5) {$x_n$}; \node[left] at (5,1.5) {$\vdots$}; \node[left] at (5,0.5) {$x_1$};
\end{tikzpicture}
}
\]
Equating the two sums gives
\[
\begin{aligned}
\sum_{\bm{\lambda}} \mathcal{L}^*_{\bm{\mu}/\bm{\lambda}}(Y_m;t) \mathcal{L}^P_{\bm{\nu}/\bm{\lambda}}(X_n;t) = \left(\prod_{i,j}\prod_{l=0}^{k-1}(1+x_iy_jt^l)^{-1} \right) \sum_{\bm{\lambda}} \mc{L}^P_{\bm{\lambda}/\bm{\mu}}(X_n;t) \mc{L}^*_{\bm{\lambda}/\bm{\nu}}(Y_m;t).
\end{aligned}
\]
where the prefactor on the RHS comes from the weight of the crosses. Using Prop. \ref{prop:Lstar}, we get
\[ \begin{aligned}
&\sum_{\bm{\lambda}} (y_1\ldots y_m)^{(l_1-m)k}(y^{\rho_m})^k t^{(ml_1-\binom{m+1}{2})\binom{k}{2}}t^{d(\bm{\mu},\bm{\lambda})}\mathcal{L}_{\bm{\mu}/\bm{\lambda}}(Y_m;t) \mathcal{L}^P_{\bm{\nu}/\bm{\lambda}}(X_n;t) =\\
& \left(\prod_{i,j}\prod_{l=0}^{k-1}(1+x_iy_jt^l)^{-1} \right) \sum_{\bm{\lambda}} (y)^{(l_1-m)k}(y^{\rho_m})^k t^{(ml_1-\binom{m+1}{2})\binom{k}{2}}t^{d(\bm{\lambda},\bm{\nu})} \mc{L}^P_{\bm{\lambda}/\bm{\mu}}(X_n;t) \mc{L}_{\bm{\lambda}/\bm{\nu}}(Y_m;t).
\end{aligned} \]

\noindent All the terms involving $l_1$ cancel and we can then take $l_1\to\infty$ giving the proposition.
\end{proof}

An analogous proof using white faces rather than purple, and using the appropriate cross weights, gives the following proposition.
\begin{prop}\label{cauchy1}
Let $\bm{\mu}$ and $\bm{\nu}$ be tuples of partitions each with infinitely many parts only finitely many of which are nonzero. Then
\[
\sum_{\bm{\lambda}} t^{d(\bm{\mu},\bm{\lambda})}\mc{L}_{\bm{\nu}/\bm{\lambda}}(X_n;t)\mc{L}_{\bm{\mu}/\bm{\lambda}}(Y_m;t) =\prod_{i,j}\prod_{l=0}^{k-1}(1-x_iy_jt^l)  \sum_{\bm{\lambda}} t^{d(\bm{\lambda},\bm{\nu})} \mc{L}_{\bm{\lambda}/\bm{\mu}}(X_n;t) \mc{L}_{\bm{\lambda}/\bm{\nu}}(Y_m;t).
\]
\end{prop}
\noindent Note that this was previously shown in \cite{CGKM}.

Using the white faces and the light-purple faces, we have the following.
\begin{prop}\label{cauchy3}
Let $\bm{\mu}$ and $\bm{\nu}$ be tuples of partitions each with infinitely many parts only finitely many of which are nonzero. Then
\[
\begin{aligned}
\left(\prod_{i,j}\prod_{l=0}^{k-1}(1+x_iy_jt^l)^{-1}\right) \sum_{\bm{\lambda}} t^{d_p(\bm{\lambda},\bm{\nu})} \mc{L}_{\bm{\lambda}/\bm{\mu}}(Y_m;t) \mc{L}^P_{\bm{\lambda}/\bm{\nu}}(X_n;t) = \sum_{\bm{\lambda}} t^{d_p(\bm{\mu},\bm{\lambda})}\mc{L}^P_{\bm{\mu}/\bm{\lambda}}(X_n;t) \mc{L}_{\bm{\nu}/\bm{\lambda}}(Y_m;t)
\end{aligned}
\]
\end{prop}
\begin{proof}
Given $\bm{\mu}$ and $\bm{\nu}$, choose positive integer $l_1$ and $l_2$ such that $l_1$ is greater than or equal to the number of nonzero parts in $\bm{\mu}$ and $\bm{\nu}$, $l_2$ is greater than or equal to the largest part in $\bm{\mu}$, and $l_2-n$ is greater than or equal to the largest part in $\bm{\nu}$. Consider the following partition function.
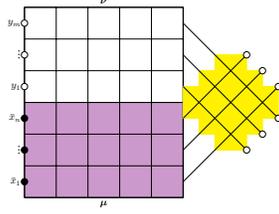
\[
\resizebox{0.25\textwidth}{!}{
$
\begin{tikzpicture}[baseline = (current bounding box).center]
\draw[fill=violet!40!white] (0,0) rectangle (5,3);
\draw[step=1.0] (0,0) grid (5,3);
\draw[] (0,3) rectangle (5,6);
\draw[step=1.0] (0,3) grid (5,6);
\node[below] at (2.5,0) {$\bm{\mu}$};
\node[above] at (2.5,6) {$\bm{\nu}$};
\draw[fill=black] (0,0.5) circle (0.1); \node[left] at (0,0.5) {$\bar{x}_1$};
\draw[fill=black] (0,1.5) circle (0.1); \node[left] at (0,1.5) {$\vdots$};
\draw[fill=black] (0,2.5) circle (0.1); \node[left] at (0,2.5) {$\bar{x}_n$};
\draw[fill=white] (0,3.5) circle (0.1); \node[left] at (0,3.5) {$y_1$};
\draw[fill=white] (0,4.5) circle (0.1); \node[left] at (0,4.5) {$\vdots$};
\draw[fill=white] (0,5.5) circle (0.1); \node[left] at (0,5.5) {$y_m$};
\draw[yellow,fill=yellow] (5,2.5)--(5,3.5)--(5.5,3.5)--(5.5,4)--(6,4)--(6,4.5)--(7,4.5)--(7,4)--(7.5,4)--(7.5,3.5)--(8,3.5)--(8,2.5)--(7.5,2.5)--(7.5,2)--(7,2)--(7,1.5)--(6,1.5)--(6,2)--(5.5,2)--(5.5,2.5)--cycle;
\draw (8,3.5)--(5,0.5); \draw[fill=white] (8,3.5) circle (0.1); 
\draw (7.5,4)--(5,1.5); \draw[fill=white] (7.5,4) circle (0.1); 
\draw (7,4.5)--(5,2.5); \draw[fill=white] (7,4.5) circle (0.1); 
\draw (7,1.5)--(5,3.5); \draw[fill=white] (7,1.5) circle (0.1); 
\draw (7.5,2)--(5,4.5); \draw[fill=white] (7.5,2) circle (0.1); 
\draw (8,2.5)--(5,5.5); \draw[fill=white] (8,2.5) circle (0.1); 
\end{tikzpicture}
$
}
\]
This can be split as
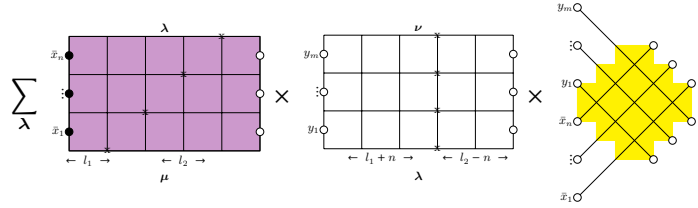
\[
\sum_{\bm{\lambda}} \resizebox{3cm}{!}{
\begin{tikzpicture}[baseline=(current bounding box.center)] 
\draw[fill=violet!40!white] (0,0) rectangle (5,3);
\draw (0,0) grid (5,3);
\node[left] at (0,1.5) {$\vdots$};
\node[below] at (2.5,-0.5) {$\bm{\mu}$}; \node[above] at (2.5,3) {$\bm{\lambda}$};
\node[] at (1,0) {x}; \node[] at (2,1) {x}; \node[] at (3,2) {x}; \node[] at (4,3) {x};
\node[left] at (0,0.5) {$\bar x_1$}; \node[left] at (0,2.5) {$\bar x_n$}; 
\draw[fill=white] (5,0.5) circle (0.1); \draw[fill=white] (5,1.5) circle (0.1); \draw[fill=white] (5,2.5) circle (0.1);
\draw[fill=black] (0,0.5) circle (0.1); \draw[fill=black] (0,1.5) circle (0.1); \draw[fill=black] (0,2.5) circle (0.1);
\node[below] at (0.5,0) {$\leftarrow\; l_1\; \rightarrow$};
\node[below] at (3,0) {$\leftarrow\; l_2\; \rightarrow$};
\end{tikzpicture} 
} 
\times
\resizebox{3cm}{!}{
\begin{tikzpicture}[baseline=(current bounding box.center)] 
\draw (0,0) grid (5,3);
\node[left] at (0,1.5) {$\vdots$};
\node[below] at (2.5,-0.5) {$\bm{\lambda}$}; \node[above] at (2.5,3) {$\bm{\nu}$};
\node[] at (3,0) {x}; \node[] at (3,1) {x}; \node[] at (3,2) {x}; \node[] at (3,3) {x};
\node[left] at (0,0.5) {$y_1$}; \node[left] at (0,2.5) {$y_m$}; 
\draw[fill=white] (5,0.5) circle (0.1); \draw[fill=white] (5,1.5) circle (0.1); \draw[fill=white] (5,2.5) circle (0.1);
\draw[fill=white] (0,0.5) circle (0.1); \draw[fill=white] (0,1.5) circle (0.1); \draw[fill=white] (0,2.5) circle (0.1);
\node[below] at (1.5,0) {$\leftarrow\; l_1+n\; \rightarrow$};
\node[below] at (4,0) {$\leftarrow\; l_2-n\; \rightarrow$};
\end{tikzpicture} }
\times
\resizebox{2cm}{!}{
\begin{tikzpicture}[baseline = (current bounding box).center]
\draw[yellow,fill=yellow] (5,2.5)--(5,3.5)--(5.5,3.5)--(5.5,4)--(6,4)--(6,4.5)--(7,4.5)--(7,4)--(7.5,4)--(7.5,3.5)--(8,3.5)--(8,2.5)--(7.5,2.5)--(7.5,2)--(7,2)--(7,1.5)--(6,1.5)--(6,2)--(5.5,2)--(5.5,2.5)--cycle;
\draw (8,3.5)--(5,0.5); \draw[fill=white] (8,3.5) circle (0.1); \draw[fill=white] (5,0.5) circle (0.1); 
\draw (7.5,4)--(5,1.5); \draw[fill=white] (7.5,4) circle (0.1); \draw[fill=white] (5,1.5) circle (0.1);
\draw (7,4.5)--(5,2.5); \draw[fill=white] (7,4.5) circle (0.1); \draw[fill=white] (5,2.5) circle (0.1);
\draw (7,1.5)--(5,3.5); \draw[fill=white] (7,1.5) circle (0.1); \draw[fill=white] (5,3.5) circle (0.1); 
\draw (7.5,2)--(5,4.5); \draw[fill=white] (7.5,2) circle (0.1); \draw[fill=white] (5,4.5) circle (0.1); 
\draw (8,2.5)--(5,5.5); \draw[fill=white] (8,2.5) circle (0.1); \draw[fill=white] (5,5.5) circle (0.1); 
\node[left] at (5,5.5) {$y_m$}; \node[left] at (5,4.5) {$\vdots$}; \node[left] at (5,3.5) {$y_1$}; \node[left] at (5,2.5) {$\bar x_n$}; \node[left] at (5,1.5) {$\vdots$}; \node[left] at (5,0.5) {$\bar x_1$};
\end{tikzpicture}
}
\]
From the previous section we know that none of the pieces depend on $l_1$, so we may extend $l_1\to \infty$. We can use the YBE to move the crosses to the LHS to get
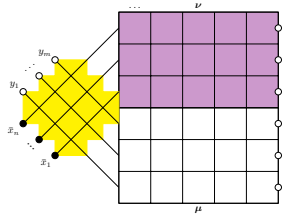
\[
\resizebox{0.25\textwidth}{!}{
$
\begin{tikzpicture}[baseline = (current bounding box).center]
\draw[] (0,0) rectangle (5,3);
\draw[step=1.0] (0,0) grid (5,3);
\draw[fill=violet!40!white] (0,3) rectangle (5,6);
\draw[step=1.0] (0,3) grid (5,6);
\node[below] at (2.5,0) {$\bm{\mu}$};
\node[above] at (2.5,6) {$\bm{\nu}$};
\draw[fill=white] (5,0.5) circle (0.1);
\draw[fill=white] (5,1.5) circle (0.1);
\draw[fill=white] (5,2.5) circle (0.1);
\draw[fill=white] (5,3.5) circle (0.1);
\draw[fill=white] (5,4.5) circle (0.1);
\draw[fill=white] (5,5.5) circle (0.1);
\node[above] at (0.5,6) {$\ldots$};
\draw[yellow,fill=yellow] (-3,2.5)--(-3,3.5)--(-2.5,3.5)--(-2.5,4)--(-2,4)--(-2,4.5)--(-1,4.5)--(-1,4)--(-0.5,4)--(-0.5,3.5)--(0,3.5)--(0,2.5)--(-0.5,2.5)--(-0.5,2)--(-1,2)--(-1,1.5)--(-2,1.5)--(-2,2)--(-2.5,2)--(-2.5,2.5)--cycle;
\draw (-3,3.5)--(0,0.5); \draw[fill=white] (-3,3.5) circle (0.1); \node[above left] at (-3,3.5) {$y_1$};
\draw (-2.5,4)--(0,1.5); \draw[fill=white] (-2.5,4) circle (0.1); \node[above left] at (-2.5,4) {$\iddots$};
\draw (-2,4.5)--(0,2.5); \draw[fill=white] (-2,4.5) circle (0.1); \node[above left] at (-2,4.5) {$y_m$};
\draw (-2,1.5)--(0,3.5); \draw[fill=black] (-2,1.5) circle (0.1); \node[below left] at (-2,1.5) {$\bar{x}_1$};
\draw (-2.5,2)--(0,4.5); \draw[fill=black] (-2.5,2) circle (0.1); \node[below left] at (-2.5,2) {$\ddots$};
\draw (-3,2.5)--(0,5.5); \draw[fill=black] (-3,2.5) circle (0.1); \node[below left] at (-3,2.5) {$\bar{x}_n$};
\end{tikzpicture}
$
}
\]
which we split into
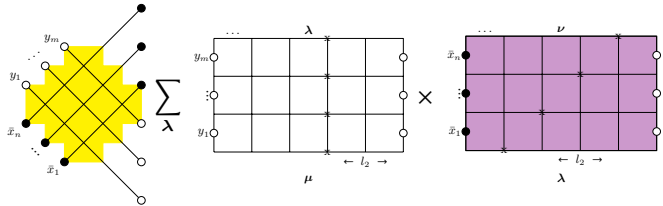
\[
\resizebox{2cm}{!}{
\begin{tikzpicture}[baseline = (current bounding box).center]
\draw[yellow,fill=yellow] (-3,2.5)--(-3,3.5)--(-2.5,3.5)--(-2.5,4)--(-2,4)--(-2,4.5)--(-1,4.5)--(-1,4)--(-0.5,4)--(-0.5,3.5)--(0,3.5)--(0,2.5)--(-0.5,2.5)--(-0.5,2)--(-1,2)--(-1,1.5)--(-2,1.5)--(-2,2)--(-2.5,2)--(-2.5,2.5)--cycle;
\draw (-3,3.5)--(0,0.5); \draw[fill=white] (-3,3.5) circle (0.1); \node[above left] at (-3,3.5) {$y_1$};
\draw (-2.5,4)--(0,1.5); \draw[fill=white] (-2.5,4) circle (0.1); \node[above left] at (-2.5,4) {$\iddots$};
\draw (-2,4.5)--(0,2.5); \draw[fill=white] (-2,4.5) circle (0.1); \node[above left] at (-2,4.5) {$y_m$};
\draw (-2,1.5)--(0,3.5); \draw[fill=black] (-2,1.5) circle (0.1); \node[below left] at (-2,1.5) {$\bar{x}_1$};
\draw (-2.5,2)--(0,4.5); \draw[fill=black] (-2.5,2) circle (0.1); \node[below left] at (-2.5,2) {$\ddots$};
\draw (-3,2.5)--(0,5.5); \draw[fill=black] (-3,2.5) circle (0.1); \node[below left] at (-3,2.5) {$\bar{x}_n$};
\draw[fill=white] (0,0.5) circle (0.1);
\draw[fill=white] (0,1.5) circle (0.1);
\draw[fill=white] (0,2.5) circle (0.1);
\draw[fill=black] (0,3.5) circle (0.1);
\draw[fill=black] (0,4.5) circle (0.1);
\draw[fill=black] (0,5.5) circle (0.1);
\end{tikzpicture}
}
\sum_{\bm{\lambda}} \resizebox{3cm}{!}{
\begin{tikzpicture}[baseline=(current bounding box.center)] 
\draw (0,0) grid (5,3);
\node[left] at (0,1.5) {$\vdots$};
\node[below] at (2.5,-0.5) {$\bm{\mu}$}; \node[above] at (2.5,3) {$\bm{\lambda}$};
\node[] at (3,0) {x}; \node[] at (3,1) {x}; \node[] at (3,2) {x}; \node[] at (3,3) {x};
\node[left] at (0,0.5) {$y_1$}; \node[left] at (0,2.5) {$y_m$}; 
\draw[fill=white] (5,0.5) circle (0.1); \draw[fill=white] (5,1.5) circle (0.1); \draw[fill=white] (5,2.5) circle (0.1);
\draw[fill=white] (0,0.5) circle (0.1); \draw[fill=white] (0,1.5) circle (0.1); \draw[fill=white] (0,2.5) circle (0.1);
\node[below] at (4,0) {$\leftarrow\; l_2\; \rightarrow$};
\node[above] at (0.5,3) {$\ldots$};
\end{tikzpicture} }
\times 
\resizebox{3cm}{!}{
\begin{tikzpicture}[baseline=(current bounding box.center)] 
\draw[fill=violet!40!white] (0,0) rectangle (5,3);
\draw (0,0) grid (5,3);
\node[left] at (0,1.5) {$\vdots$};
\node[below] at (2.5,-0.5) {$\bm{\lambda}$}; \node[above] at (2.5,3) {$\bm{\nu}$};
\node[] at (1,0) {x}; \node[] at (2,1) {x}; \node[] at (3,2) {x}; \node[] at (4,3) {x};
\node[left] at (0,0.5) {$\bar x_1$}; \node[left] at (0,2.5) {$\bar x_n$}; 
\draw[fill=white] (5,0.5) circle (0.1); \draw[fill=white] (5,1.5) circle (0.1); \draw[fill=white] (5,2.5) circle (0.1);
\draw[fill=black] (0,0.5) circle (0.1); \draw[fill=black] (0,1.5) circle (0.1); \draw[fill=black] (0,2.5) circle (0.1);
\node[below] at (3,0) {$\leftarrow\; l_2\; \rightarrow$};
\node[above] at (0.5,3) {$\ldots$};
\end{tikzpicture} 
} 
\]
Note that sufficiently far to the left in the white faces every column is dense with vertical paths, thus the paths from the cross must enter in the light-purple faces. Setting this equal to the other sum
gives
\[
\left(\prod_{i,j}\prod_{l=0}^{k-1}(1+x_iy_jt^{l})^{-1}\right) \sum_{\bm{\lambda}}  \mc{L}_{\bm{\lambda}/\bm{\mu}}(Y_m;t) (\mc{L}^P)^*_{\bm{\lambda}/\bm{\nu}}(X_n;t) = \sum_{\bm{\lambda}} (\mc{L}^P)^*_{\bm{\mu}/\bm{\lambda}}(X_n;t) \mc{L}_{\bm{\nu}/\bm{\lambda}}(Y_m;t)
\]
where the prefactor on the LHS comes from the weight of the crosses. Using Prop. \ref{prop:Lpstar} we have

\[
\begin{aligned}
\left(\prod_{i,j}\prod_{l=0}^{k-1}(1+x_iy_jt^{l})^{-1}\right) \sum_{\bm{\lambda}} (x_1\ldots x_n)^{k(l_2-n+1)}(x^{\rho_n})^k t^{d_p(\bm{\lambda},\bm{\nu})} \mc{L}_{\bm{\lambda}/\bm{\mu}}(Y_m;t) \mc{L}^P_{\bm{\lambda}/\bm{\nu}}(X_n;t) \\
= \sum_{\bm{\lambda}} (x_1\ldots x_n)^{k(l_2-n+1)}(x^{\rho_n})^k t^{d_p(\bm{\mu},\bm{\lambda})}\mc{L}^P_{\bm{\mu}/\bm{\lambda}}(X_n;t) \mc{L}_{\bm{\nu}/\bm{\lambda}}(Y_m;t).
\end{aligned}
\]

\noindent Canceling the terms involving $l_2$ gives the proposition.
\end{proof}

Changing the white faces to purple faces, a similar computation to the above gives
\begin{prop}\label{cauchy4}
Let $\bm{\mu}$ and $\bm{\nu}$ be tuples of partitions each with infinitely many parts only finitely many of which are nonzero. Then
\[
\begin{aligned}
\left(\prod_{i,j}\prod_{l=0}^{k-1}(1-x_iy_jt^l)\right) \sum_{\bm{\lambda}} t^{d_p(\bm{\lambda},\bm{\nu})} \mc{L}^P_{\bm{\lambda}/\bm{\mu}}(Y_m;t) \mc{L}^P_{\bm{\lambda}/\bm{\nu}}(X_n;t) \\= \sum_{\bm{\lambda}} t^{d_p(\bm{\mu},\bm{\lambda})} \mc{L}^P_{\bm{\mu}/\bm{\lambda}}(X_n;t) \mc{L}^P_{\bm{\nu}/\bm{\lambda}}(Y_m;t)
\end{aligned}
\]
\end{prop}
\noindent(One must be careful to only consider terms with finite degree in $y$, this forces the paths to only travel from the SW to the SE on the cross.)

Combining these identities, we now come to the main result of this section: a Cauchy identity for the supersymmetric LLT polynomials.
\begin{thm}\label{thm:cauchyss}
Let $\bm{\mu}$ and $\bm{\nu}$ be tuples of partitions each with infinitely many parts only finitely many of which are nonzero. Fix positive integers $n,m,p,q$. Then
\begin{equation}
\resizebox{0.8\textwidth}{!}{
$
\begin{aligned}
 & \sum_{\bm{\lambda}} t^{d(\bm{\mu},\bm{\lambda})} \mathcal{L}^S_{\bm{\nu}/\bm{\lambda}}(X_n,Y_m;t) \mathcal{L}^S_{\bm{\mu}/\bm{\lambda}}(W_p,Z_q;t) \\ & =  \prod_{l=0}^{k-1} \prod_{i,i'=1}^n\prod_{j,j'=1}^m\prod_{\alpha,\alpha'=1}^p\prod_{\beta,\beta'=1}^q \frac{(1-x_iw_\alpha t^l)(1-y_{j'}z_{\beta'}t^l)}{(1+y_{j}w_{\alpha'}t^l)(1+x_{i'}z_{\beta}t^l)} \sum_{\bm{\lambda}} t^{d(\bm{\lambda},\bm{\nu})} \mathcal{L}^S_{\bm{\lambda}/\bm{\mu}}(X_n,Y_m;t)\mathcal{L}^S_{\bm{\lambda}/\bm{\nu}}(W_p,Z_q;t)
 \end{aligned}
$
 }
\end{equation}
\end{thm}
\begin{proof}
We can rewrite the LHS as
\[
LHS=\sum_{\bm{\lambda},\bm{\sigma},\bm{\rho}}  \mathcal{L}^P_{\bm{\nu}/\bm{\rho}}(Y_m;t)\mathcal{L}_{\bm{\rho}/\bm{\lambda}}(X_n;t) t^{d_P(\bm{\mu},\bm{\sigma})}\mathcal{L}^P_{\bm{\mu}/\bm{\sigma}}(Z_q;t)t^{d(\bm{\sigma},\bm{\lambda})}\mathcal{L}_{\bm{\sigma}/\bm{\lambda}}(W_p;t)
\]
where we use the fact that
\[
d(\bm{\mu},\bm{\lambda}) = d_P(\bm{\mu},\bm{\sigma}) + d(\bm{\sigma},\bm{\lambda}).
\]
Now we may use Prop. \ref{cauchy1} on the sum over $\bm{\lambda}$ with the second and fourth LLT polynomials to get
\[
\begin{aligned}
LHS = & \prod_{l=0}^{k-1}\prod_{i=1}^n\prod_{\alpha=1}^p(1-x_iw_\alpha t^l) \\
& \times \sum_{\bm{\lambda},\bm{\sigma},\bm{\rho}}  \mathcal{L}^P_{\bm{\nu}/\bm{\rho}}(Y_m;t)\mathcal{L}_{\bm{\lambda}/\bm{\sigma}}(X_n;t) t^{d_P(\bm{\mu},\bm{\sigma})}\mathcal{L}^P_{\bm{\mu}/\bm{\sigma}}(Z_q;t)t^{d(\bm{\lambda},\bm{\rho})}\mathcal{L}_{\bm{\lambda}/\bm{\rho}}(W_p;t).
\end{aligned}
\]
Using Prop. \ref{cauchy2} on the first and fourth LLT polynomials and summing over $\bm{\rho}$ gives
\[
\begin{aligned}
LHS = & \prod_{l=0}^{k-1}\prod_{i=1}^n\prod_{j=1}^m\prod_{\alpha,\alpha'=1}^p\frac{1-x_iw_\alpha t^l}{1+y_{j}w_{\alpha'}t^l} \\
& \times \sum_{\bm{\lambda},\bm{\sigma},\bm{\rho}}  \mathcal{L}^P_{\bm{\rho}/\bm{\lambda}}(Y_m;t)\mathcal{L}_{\bm{\lambda}/\bm{\sigma}}(X_n;t) t^{d_P(\bm{\mu},\bm{\sigma})}\mathcal{L}^P_{\bm{\mu}/\bm{\sigma}}(Z_q;t)t^{d(\bm{\rho},\bm{\nu})}\mathcal{L}_{\bm{\rho}/\bm{\nu}}(W_p;t).
\end{aligned}
\]
Next, summing over $\bm{\sigma}$ using Prop. \ref{cauchy3} on the second and third LLT polynomials gives
\[
\begin{aligned}
LHS = & \prod_{l=0}^{k-1}\prod_{i,i'=1}^n\prod_{j=1}^m\prod_{\alpha,\alpha'=1}^p\prod_{\beta=1}^q \frac{1-x_iw_\alpha t^l}{(1+y_{j}w_{\alpha'}t^l)(1+x_{i'}z_{\beta}t^l)} \\
& \times\sum_{\bm{\lambda},\bm{\sigma},\bm{\rho}}  \mathcal{L}^P_{\bm{\rho}/\bm{\lambda}}(Y_m;t)\mathcal{L}_{\bm{\sigma}/\bm{\mu}}(X_n;t) t^{d_P(\bm{\sigma},\bm{\lambda})}\mathcal{L}^P_{\bm{\sigma}/\bm{\lambda}}(Z_q;t)t^{d(\bm{\rho},\bm{\nu})}\mathcal{L}_{\bm{\rho}/\bm{\nu}}(W_p;t).
\end{aligned}
\]
Finally, using Prop. \ref{cauchy4} on the sum over $\bm{\lambda}$ with the first and third LLT polynomials results in
\[
\begin{aligned}
LHS = & \prod_{l=0}^{k-1} \prod_{i,i'=1}^n\prod_{j,j'=1}^m\prod_{\alpha,\alpha'=1}^p\prod_{\beta,\beta'=1}^q \frac{(1-x_iw_\alpha t^l)(1-y_{j'}z_{\beta'}t^l)}{(1+y_{j}w_{\alpha'}t^l)(1+x_{i'}z_{\beta}t^l)} \\
& \times \sum_{\bm{\lambda},\bm{\sigma},\bm{\rho}}  \mathcal{L}^P_{\bm{\lambda}/\bm{\sigma}}(Y_m;t)\mathcal{L}_{\bm{\sigma}/\bm{\mu}}(X_n;t) t^{d_P(\bm{\lambda},\bm{\rho})}\mathcal{L}^P_{\bm{\lambda}/\bm{\rho}}(Z_q;t)t^{d(\bm{\rho},\bm{\nu})}\mathcal{L}_{\bm{\rho}/\bm{\nu}}(W_p;t)
\end{aligned}
\]
which can be combined into
\[
 \prod_{l=0}^{k-1} \prod_{i,i'=1}^n\prod_{j,j'=1}^m\prod_{\alpha,\alpha'=1}^p\prod_{\beta,\beta'=1}^q \frac{(1-x_iw_\alpha t^l)(1-y_{j'}z_{\beta'}t^l)}{(1+y_{j}w_{\alpha'}t^l)(1+x_{i'}z_{\beta}t^l)} \sum_{\bm{\lambda}} t^{d(\bm{\lambda},\bm{\nu})} \mathcal{L}^S_{\bm{\lambda}/\bm{\mu}}(X_n,Y_m;t)\mathcal{L}^S_{\bm{\lambda}/\bm{\nu}}(W_p,Z_q;t)
\]
where we again use that
\[
 d(\bm{\lambda},\bm{\nu}) = d_P(\bm{\lambda},\bm{\rho}) + d(\bm{\rho},\bm{\nu}).
\]
This is the desired RHS.
\end{proof}

\pagebreak

\appendix
\addcontentsline{toc}{section}{Appendices}
\section*{Appendices}

\section{Proof of Lemma \ref{lem:lqm-properties} \label{sec:lqm-properties-proof}}

Throughout this section, whenever we consider a skew shape $\alpha/\beta$, we assume $\alpha$ and $\beta$ have the same number of parts $\ell(\alpha/\beta)$.  Moreover, using Remark \ref{rmk:maya-extend}, we can take the Maya diagrams of $\alpha$ and $\beta$ to have the same length.  Also let $f_k(\alpha/\beta)$ denote the $k$-quotient of $\alpha/\beta$.

Let $\lambda/\mu$ be a skew shape and let $\lm = (\lambda^{(0)}/\mu^{(0)},\ldots,\lambda^{(k-1)}/\mu^{(k-1)})$ be its $k$-quotient.  Let 
\[ T \in \SSRT_k(\lambda/\mu) \leftrightarrow \bm{T} = (T^{(0)},\ldots,T^{(k-1)}) \in \SSSYT_k(\lm) \]
via the Littlewood $k$-quotient map.  We want to prove the following two claims.
\begin{enumerate}
    \item A ribbon in $T$ labelled $i$ corresponds to a cell labelled $i$ in $\bm{T}$, so the number of ribbons in $T$ labelled $i$ equals the number of cells labelled $i$ in $\bm{T}$.
    \item Two ribbons $R,R'$ in $T$ whose tails $u,u'$ have the same content modulo $k$ correspond to two cells $v,v'$ in the same shape in $\bm{T}$.  Moreover, in this case, 
    \[ \frac{c(u)-c(u')}{k} = c(v)-c(v'). \]
\end{enumerate}
We begin by discussing Maya diagrams and content lines.  Let $\alpha/\beta$ be a skew shape, and let $(a_0,\ldots,a_{s-1}),(b_0,\ldots,b_{s-1})$ be the Maya diagrams of $\alpha,\beta$ respectively.  Given a cell $u$ in $\alpha/\beta$, we define its \textbf{adjusted content} to be
\[ ac(u) := c(u) + \ell(\alpha/\beta) - 1, \]
where $c(u)$ is its content.  The following facts are straightforward.
\begin{enumerate}[label=(\Alph*)]
\item The skew shape $\alpha/\beta$ consists of a single cell $u$ iff $a_i = b_{i+1} = E$, $a_{i+1} = b_i = S$, and $a_j = b_j$ for $j \neq i,i+1$, for some $i$.  In this case, if $u$ is the single cell in $\alpha/\beta$, we have $ac(u) = i$. 
\item The skew shape $\alpha/\beta$ consists of a single ribbon iff $a_i = b_{i+k} = E$, $a_{i+k} = b_i = S$, and $a_j = b_j$ for $j \neq i,i+k$, for some $i$.  In this case, if $u$ is the tail of the single ribbon in $\alpha/\beta$, we have $ac(u) = i$. 
\end{enumerate}

The claims will follow from the following lemma.
\begin{lem} \label{lem:a1}
If $\lambda/\mu$ is a $k$-ribbon, then $\lm$ consists of a single cell i.e. $|\lm| = 1$.  Let $u$ be the tail of the ribbon in $\lambda/\mu$ and let $v$ be the cell in $\lm$.  Write $ac(u) = qk+r$ where $0 \leq r < k$.  Then $v$ appears in $\lambda^{(r)}/\mu^{(r)}$ and has adjusted content $ac(v) = q$.
\end{lem}

\begin{proof}[Proof of Lemma \ref{lem:a1}]

Let $u$ be the tail of the ribbon $\lambda/\mu$.  Let $(a_0,\ldots,a_{s-1}),(b_0,\ldots,b_{s-1})$ be the Maya diagrams of $\lambda,\mu$ respectively.  By Remark \ref{rmk:maya-extend}, we may assume $t = s/k$ is an integer.  By Fact B, for some $i$, we have
\[ ac(u) = i; a_i = b_{i+k} = E; a_{i+k} = b_i = S; a_j = b_j \text{ for } j \neq i,i+k. \]
Let $(a^{(j)}_0,\ldots,a^{(j)}_{t-1}),(b^{(j)}_0,\ldots,b^{(j)}_{t-1})$ be the Maya diagrams of $\lambda^{(j)},\mu^{(j)}$ respectively for each $j$.  By the definition of the $k$-quotient map, we have $a^{(j)}_l = a_{lk+j}$ and $b^{(j)}_l = b_{lk+j}$ for each $j$ and $l$.  Since $a_j = b_j$ for $j \neq i,i+k$,
\[ (a^{(j)}_0,\ldots,a^{(j)}_{t-1}) = (b^{(j)}_0,\ldots,b^{(j)}_{t-1}) \text{ for } j \neq r \]
and thus $\lambda^{(j)} = \mu^{(j)}$ for $j \neq r$.  Since $a_i = b_{i+k} = E$, $a_{i+k} = b_i = S$, and $a_j = b_j$ for $j \neq i,i+k$,
\[ a^{(r)}_q = b^{(r)}_{q+1} = E, a^{(r)}_{q+1} = b^{(r)}_q = S, \text{ and } a^{(r)}_j = b^{(r)}_j \text{ for } j \neq q,q+1. \]
Thus, by Fact A, $\lambda^{(r)} / \mu^{(r)}$ consists of a single cell $v$ with adjusted content $ac(v) = q$.
\end{proof}

We now prove Claim 1, using Lemma \ref{lem:a1}.  Suppose there are $m$ ribbons $R_1,\ldots,R_m$ labelled $i$ in $T$.  Then we can construct a series of partitions 
\[ \alpha^{(0)} = \lambda_{\leq i-1}, \ldots, \alpha^{(m)} = \lambda_{\leq i} \]
such that $\alpha^{(j)} / \alpha^{(j-1)} = R_j$ for each $j \in [m]$.  Using the lemma,
\[ |f_k(\lambda_{\leq i}/\lambda_{\leq i-1})| = |f_k(\alpha^{(m)}/\alpha^{(0)})| = \sum_{j=1}^m |f_k(\alpha^{(j)}/\alpha^{(j-1)})| = \sum_{j=1}^m |f_k(R_j)| = \sum_{j=1}^m 1 = m. \]
However, by the definition of the Littlewood $k$-quotient map, $f_k(\lambda_{\leq i}/\lambda_{\leq i-1})$ (lying inside $\lm$) consists of exactly the cells labelled $i$ in $\bm{T}$, so there are $m$ cells labelled $i$ in $\bm{T}$.

We now prove Claim 2, again using Lemma \ref{lem:a1}.  Write $ac(u) = qk+r$ and $ac(u') = q'k+r'$, where $0 \leq r,r' < k$.  Since $u$ and $u'$ have the same content modulo $k$, they have the same adjusted content modulo $k$, hence $r = r'$.  Thus, by the lemma, both $v$ and $v'$ appear in the same shape in $\bm{T}$, namely $T^{(r)} = T^{(r')}$.  Moreover, again using the lemma, 
\[ \frac{c(u)-c(u')}{k} = \frac{ac(u)-ac(u')}{k} = q-q' = ac(v)-ac(v') = c(v)-c(v'). \]

\section{Other Formulations of LLT Polynomials \label{sec:other-LLT}}

We describe the relationship between coinversion LLT polynomials (Definition \ref{def:coinv-LLT}) and other formulations of LLT polynomials that appear in the literature.

\begin{remark} \label{rmk:hhl-llt}
The inversion LLT polynomial, as in \cite{HHL}, is given by
\[ \mc{L}^{HHL}_{\lm}(X;t) = \sum_{T \in \SSYT(\lm)} t^{\inv(T)}x^T. \] 
We have the relationship
\[ \mc{L}_\lm(X;t) = t^m \mc{L}^{HHL}_\lm(X;t^{-1}) \]
where
\[ \begin{aligned}
m &= \# \text{ triples in } \lm \\
&= \max_{U \in \SSYT(\lm)} \coinv(U) + \min_{U \in \SSYT(\lm)} \inv(U) \\
&= \min_{U \in \SSYT(\lm)} \coinv(U) + \max_{T \in \SSYT(\lm)} \inv(U). 
\end{aligned} \]
More explicit formulae for $m$ are given in \cite[Section 5]{CGKM}.  
\end{remark}

\begin{remark} The sp LLT polynomial is given by
\[ \mc{L}^{sp}_{\lambda/\mu}(X;t) = \sum_{T \in \SSRT_k(\lambda/\mu)} t^{\spp(T)} x^T \]
where 
\[ \spp(T) = \frac{\spin(T)}{2} - \min_{U \in \SSRT_k(\lambda/\mu)} \frac{\spin(U)}{2}. \]
If $\lm$ is the $k$-quotient of $\lambda/\mu$, then
\[ \mc{L}_\lm(X;t) = t^{m-e} \mc{L}^{sp}_{\lambda/\mu}(X;t) \]
where 
\[ e = \max_{U \in \SSYT(\lm)} \inv(U), \hspace{1cm} m-e = \min_{U \in \SSYT(\lm)} \coinv(U). \]
\end{remark}

\begin{remark} \label{rmk:lam-llt}
The spin LLT polynomial, as in \cite{Lam, CYZZ, BBBG}, is given by
\[ \mc{L}^{Lam}_{\lambda/\mu}(X;t) = \sum_{T \in \SSRT_k(\lambda/\mu)} t^{\spin(T)} x^T. \]
If $\lm$ is the $k$-quotient of $\lambda/\mu$, then 
\[ \mc{L}_\lm(X;t) = t^{m-e-*} \mc{L}^{Lam}_{\lambda/\mu}(X;t^{1/2}) \]
where 
\[ * = \min_{U \in \SSRT_k(\lambda/\mu)} \frac{\spin(U)}{2}. \]
It is clear from the definitions that
\[ \mc{L}^{Lam}_{\lambda/\mu}(X;t) = \mc{G}^{(k)}_{\lambda/\mu}(X;;t). \]
\end{remark}

\begin{remark}
The cosp LLT polynomial is given by
\[ \mc{L}^{LLT}_{\lambda/\mu}(X;t) = \sum_{T \in \SSRT_k(\lambda/\mu)} t^{\cosp(T)} x^T \]
where 
\[ \cosp(T) = \max_{U \in \SSRT_k(\lambda/\mu)} \frac{\spin(U)}{2} - \frac{\spin(T)}{2}. \]
Historically this was the original formulation of LLT polynomials \cite{LLT}.  If $\lm$ is the $k$-quotient of $\lambda/\mu$, then
\[ \mc{L}_\lm(X;t) = t^{m-e+\dagger-*} \mc{L}^{LLT}_{\lambda/\mu}(X;t^{-1}) \]
where 
\[ \dagger = \max_{U \in \SSRT_k(\lambda/\mu)} \frac{\spin(U)}{2}. \]
\end{remark}

\section{Proofs of Propositions \ref{prop:YBEpurplewhite} and \ref{prop:YBEpurple} \label{sec:abw-ybe-limits}}

We adopt the notation of \cite{ABW} throughout this section, except we use $t$ in place of $q$.  In particular, we let $W_z(\A,\B;\C,\D | r,s)$ be the vertex weights from \cite[Definition 5.1.1]{ABW} with $t$ in place of $q$.  

The proofs of Propositions \ref{prop:YBEpurplewhite} and \ref{prop:YBEpurple} utilize \cite[Prop. 5.1.4]{ABW}, which states
\begin{equation} \label{eq:abw-ybe}
\begin{aligned}
&\sum_{\I_2,\J_2,\K_2} W_{\chi/\gamma}(\I_1,\J_1;\I_2,\J_2 | r,s) W_{\chi/z}(\K_1,\J_2;\K_2,\J_3 | r,t) W_{\gamma/z}(\K_2,\I_2;\K_3,\I_3 | s,t) \\ 
&= \sum_{\I_2,\J_2,\K_2} W_{\gamma/z}(\K_1,\I_1;\K_2,\I_2 | s,t) W_{\chi/z}(\K_2,\J_1;\K_3,\J_2 | r,t) W_{\chi/\gamma}(\I_2,\J_2;\I_3,\J_3 | r,s). 
\end{aligned}
\end{equation}
for all $\chi,\gamma,z,r,s,t \in \CC$ for any choice of boundary condition $\I_1,\J_1,\K_1,\I_3,\J_3,\K_3 \in \{0,1\}^k$.

\begin{proof}[Proof of Proposition \ref{prop:YBEpurplewhite}]
 Fix $S,Y,x,y \in \CC$ and let $\alpha = -S^2$ and $\beta = SY^{1/2}$.  Substituting $\chi = x$, $\gamma = \alpha$, $z = \alpha$, $r = (x/\alpha)^{1/2}$, $s = Sy^{-1/2}$, and $t = SY^{1/2}$ into \eqref{eq:abw-ybe} gives 
\[
\resizebox{0.85\textwidth}{!}{
$ \begin{aligned}
&\sum_{\I_2,\J_2,\K_2} W_{x/\alpha}(\I_1,\J_1;\I_2,\J_2 | (x/\alpha)^{1/2},Sy^{-1/2}) W_{x/\alpha}(\K_1,\J_2;\K_2,\J_3 | (x/\alpha)^{1/2},SY^{-1/2}) W_1(\K_2,\I_2;\K_3,\I_3 | Sy^{-1/2},SY^{1/2}) \\ 
&= \sum_{\I_2,\J_2,\K_2} W_1(\K_1,\I_1;\K_2,\I_2 | Sy^{-1/2},SY^{1/2}) W_{x/\alpha}(\K_2,\J_1;\K_3,\J_2 | (x/\alpha)^{1/2},SY^{1/2}) W_{x/\alpha}(\I_2,\J_2;\I_3,\J_3 | (x/\alpha)^{1/2},Sy^{-1/2}). 
\end{aligned} $
}
\]
Multiplying both sides by $(-\alpha)^{|\J_3|}(SY^{1/2})^{-2|\J_3|}Y^{-|\I_3|}$ gives
\[
\resizebox{0.85\textwidth}{!}{
$
\begin{aligned}
&\sum_{\I_2,\J_2,\K_2} W_{x/\alpha}(\I_1,\J_1;\I_2,\J_2 | (x/\alpha)^{1/2},Sy^{-1/2}) (-\alpha)^{|\J_3|}(SY^{1/2})^{-2|\J_3|}W_{x/\alpha}(\K_1,\J_2;\K_2,\J_3 | (x/\alpha)^{1/2},SY^{-1/2}) \\
&\hspace{0.4\textwidth} \cdot Y^{-|\I_3|}W_1(\K_2,\I_2;\K_3,\I_3 | Sy^{-1/2},SY^{1/2}) \\ 
&= \sum_{\I_2,\J_2,\K_2} Y^{-|\I_2|}W_1(\K_1,\I_1;\K_2,\I_2 | Sy^{-1/2},SY^{1/2}) (-\alpha)^{|\J_2|}(SY^{1/2})^{-2|\J_2|}W_{x/\alpha}(\K_2,\J_1;\K_3,\J_2 | (x/\alpha)^{1/2},SY^{1/2}) \\
&\hspace{0.4\textwidth} \cdot (-\alpha)^{|\J_3|-|\J_2|}S^{-2|\J_3|+2|\J_2|}W_{x/\alpha}(\I_2,\J_2;\I_3,\J_3 | (x/\alpha)^{1/2},Sy^{-1/2}). 
\end{aligned} 
$
}
\]
where we have used $|\I_2| + |\J_2| = |\I_3| + |\J_3|$ by path conservation on the right-hand side.  Substituting $\alpha = -S^2$ and $\beta = SY^{1/2}$ gives
\[ 
\resizebox{0.85\textwidth}{!}{
$
\begin{aligned}
&\sum_{\I_2,\J_2,\K_2} W_{x/\alpha}(\I_1,\J_1;\I_2,\J_2 | (x/\alpha)^{1/2},(-y/\alpha)^{-1/2}) (-\alpha)^{|\J_3|}\beta^{-2|\J_3|}W_{x/\alpha}(\K_1,\J_2;\K_2,\J_3 | (x/\alpha)^{1/2},\beta) \\
&\hspace{0.4\textwidth} \cdot Y^{-|\I_3|}W_1(\K_2,\I_2;\K_3,\I_3 | Sy^{-1/2},SY^{1/2}) \\ 
&= \sum_{\I_2,\J_2,\K_2} Y^{-|\I_2|}W_1(\K_1,\I_1;\K_2,\I_2 | Sy^{-1/2},SY^{1/2})  (-\alpha)^{|\J_2|}\beta^{-2|\J_2|}W_{x/\alpha}(\K_2,\J_1;\K_3,\J_2 | (x/\alpha)^{1/2},\beta) \\
&\hspace{0.4\textwidth} \cdot W_{x/\alpha}(\I_2,\J_2;\I_3,\J_3 | (x/\alpha)^{1/2},(-y/\alpha)^{-1/2}). 
\end{aligned}
$
}
\]
Taking $S \rightarrow 0$ and $Y \rightarrow 0$ (hence also $\alpha = -S^2 \rightarrow 0$ and $\beta = SY^{1/2} \rightarrow 0$) and then applying Lemma \ref{lem:degen-abw} gives the desired YBE
\[ \begin{aligned}
&\sum_{\I_2,\J_2,\K_2} R'_{y/x}(\I_1,\J_1;\I_2,\J_2) \tilde{L}_x(\K_1,\J_2;\K_2,\J_3) \tilde{L}'_y(\K_2,\I_2;\K_3,\I_3) \\ 
&= \sum_{\I_2,\J_2,\K_2} L'_y(\K_1,\I_1;\K_2,\I_2) L_x(\K_2,\J_1;\K_3,\J_2) R'_{y/x}(\I_2,\J_2;\I_3,\J_3). 
\end{aligned} \]
\end{proof}

\begin{proof}[Proof of Proposition \ref{prop:YBEpurple}]
Fix $S,Y,x,y \in \CC$ and let $\alpha = -S^2$ and $\beta = SY^{1/2}$.  Substituting $\chi = 1$, $\gamma = 1$, $z = 1$, $r = Sx^{-1/2}$, $s = Sy^{-1/2}$, and $t = SY^{-1/2}$ into \eqref{eq:abw-ybe} gives 
\[ 
\resizebox{0.85\textwidth}{!}{
$
\begin{aligned}
&\sum_{\I_2,\J_2,\K_2} W_1(\I_1,\J_1;\I_2,\J_2 | Sx^{-1/2},Sy^{-1/2}) W_1(\K_1,\J_2;\K_2,\J_3 | Sx^{-1/2},SY^{-1/2}) W_1(\K_2,\I_2;\K_3,\I_3 | Sy^{-1/2},SY^{1/2}) \\ 
&= \sum_{\I_2,\J_2,\K_2} W_1(\K_1,\I_1;\K_2,\I_2 | Sy^{-1/2},SY^{1/2}) W_1(\K_2,\J_1;\K_3,\J_2 | Sx^{-1/2},SY^{1/2}) W_1(\I_2,\J_2;\I_3,\J_3 | Sx^{-1/2},Sy^{-1/2}). 
\end{aligned}$
}
\]
Multiplying both sides by $Y^{-|\I_3|-|\J_3|}$ gives
\[ 
\resizebox{0.85\textwidth}{!}{
$
\begin{aligned}
&\sum_{\I_2,\J_2,\K_2} W_1(\I_1,\J_1;\I_2,\J_2 | Sx^{-1/2},Sy^{-1/2})  Y^{-|\J_3|}W_1(\K_1,\J_2;\K_2,\J_3 | Sx^{-1/2},SY^{-1/2}) \\
&\hspace{0.4\textwidth} \cdot Y^{-|\I_3|}W_1(\K_2,\I_2;\K_3,\I_3 | Sy^{-1/2},SY^{1/2}) \\ 
&= \sum_{\I_2,\J_2,\K_2} Y^{-|\I_2|}W_1(\K_1,\I_1;\K_2,\I_2 | Sy^{-1/2},SY^{1/2})  Y^{-|\J_2|}W_1(\K_2,\J_1;\K_3,\J_2 | Sx^{-1/2},SY^{1/2}) \\
&\hspace{0.4\textwidth} \cdot W_1(\I_2,\J_2;\I_3,\J_3 | Sx^{-1/2},Sy^{-1/2}).  
\end{aligned}
$
}
\]
where we have used $|\I_2| + |\J_2| = |\I_3| + |\J_3|$ by path conservation on the right-hand side.  Taking $S \rightarrow 0$ and $Y \rightarrow 0$ and then applying Lemma \ref{lem:degen-abw} gives the desired YBE
\[ \begin{aligned}
&\sum_{\I_2,\J_2,\K_2} R''_{x/y}(\I_1,\J_1;\I_2,\J_2) L'_x(\K_1,\J_2;\K_2,\J_3) L'_y(\K_2,\I_2;\K_3,\I_3) \\ 
&= \sum_{\I_2,\J_2,\K_2} L'_y(\K_1,\I_1;\K_2,\I_2) L'_x(\K_2,\J_1;\K_3,\J_2) R''_{x/y}(\I_2,\J_2;\I_3,\J_3). 
\end{aligned} \]
\end{proof}

\section{Example Computations of $\mc{L}^S$ and $\mc{G}$}

\subsection{Example 1}
\indent Let $\lambda = (4,2)$ and $\bm{\lambda} = ((1),(2))$, noting that $\bm{\lambda}$ is the 2-quotient of $\lambda$.
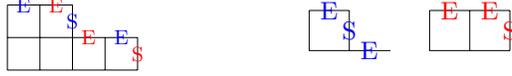
\[ \resizebox{2cm}{!}{
\begin{tikzpicture}[baseline=(current bounding box.center)]
\draw (0,0) grid (2,2); \draw (2,0) grid (4,1);
\node[scale=2,blue] at (0.5,2) {E}; \node[scale=2,red] at (1.5,2) {E}; \node[scale=2,blue] at (2,1.5) {S}; \node[scale=2,red] at (2.5,1) {E}; \node[scale=2,blue] at (3.5,1) {E}; \node[scale=2,red] at (4,0.5) {S};
\end{tikzpicture} }
\hspace{2cm} 
\resizebox{3cm}{!}{
\begin{tikzpicture}[baseline=(current bounding box.center)]
\draw (0,0) grid (1,1); \draw (1,0) grid (2,0);
\node[blue,scale=2] at (0.5,1) {E}; \node[blue,scale=2] at (1,0.5) {S}; \node[blue,scale=2] at (1.5,0) {E};
\draw (3+0,0) grid (3+2,1);
\node[red,scale=2] at (3+0.5,1) {E}; \node[red,scale=2] at (3+1.5,1) {E}; \node[red,scale=2] at (3+2,0.5) {S};
\end{tikzpicture} }
\]
Therefore, by Proposition \ref{prop-L-equals-G}, we must have 
\[ \mc{L}^S_{\bm{\lambda}}(x_1;y_1;t) = t^\square \mc{G}^{(2)}_{\lambda}(x_1;y_1;t^{1/2}) \]
for some half-integer $\square \in \frac{1}{2}\mathbb{Z}$. To compute $\mc{L}^S_{\bm{\lambda}}(x_1;y_1;t)$, we note that there are 4 configurations of the lattice $S_{1,1}(\bm{\lambda})$.
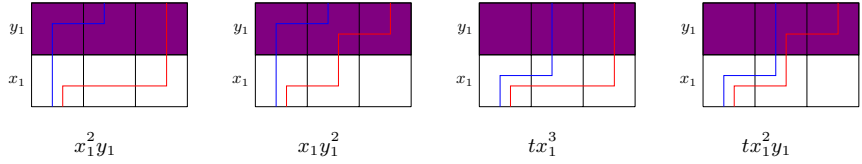
\begin{center}
\resizebox{0.8\textwidth}{!}{
\begin{tabular}{cccc}
	\resizebox{3cm}{!}{
	\begin{tikzpicture}[baseline=(current bounding box.center)]
	\draw[fill=violet] (0,1) rectangle (3,2);
	\draw (0,0) grid (3,2);
	\node [left] at (0,0.5) {$x_1$}; \node [left] at (0,1.5) {$y_1$}; 
	\draw[blue] (0.4,0)--(0.4,1.6)--(1.4,1.6)--(1.4,2);
	\draw[red] (0.6,0)--(0.6,0.4)--(2.6,0.4)--(2.6,2);
	\end{tikzpicture} }
&
	\resizebox{3cm}{!}{
	\begin{tikzpicture}[baseline=(current bounding box.center)]
	\draw[fill=violet] (0,1) rectangle (3,2);
	\draw (0,0) grid (3,2);
	\node [left] at (0,0.5) {$x_1$}; \node [left] at (0,1.5) {$y_1$}; 
	\draw[blue] (0.4,0)--(0.4,1.6)--(1.4,1.6)--(1.4,2);
	\draw[red] (0.6,0)--(0.6,0.4)--(1.6,0.4)--(1.6,1.4)--(2.6,1.4)--(2.6,2);
	\end{tikzpicture} }
&
	\resizebox{3cm}{!}{
	\begin{tikzpicture}[baseline=(current bounding box.center)]
	\draw[fill=violet] (0,1) rectangle (3,2);
	\draw (0,0) grid (3,2);
	\node [left] at (0,0.5) {$x_1$}; \node [left] at (0,1.5) {$y_1$}; 
	\draw[blue] (0.4,0)--(0.4,0.6)--(1.4,0.6)--(1.4,2);
	\draw[red] (0.6,0)--(0.6,0.4)--(2.6,0.4)--(2.6,2);
	\end{tikzpicture} }
&
	\resizebox{3cm}{!}{
	\begin{tikzpicture}[baseline=(current bounding box.center)]
	\draw[fill=violet] (0,1) rectangle (3,2);
	\draw (0,0) grid (3,2);
	\node [left] at (0,0.5) {$x_1$}; \node [left] at (0,1.5) {$y_1$}; 
	\draw[blue] (0.4,0)--(0.4,0.6)--(1.4,0.6)--(1.4,2);
	\draw[red] (0.6,0)--(0.6,0.4)--(1.6,0.4)--(1.6,1.4)--(2.6,1.4)--(2.6,2);
	\end{tikzpicture} }
\\ \\
$x_1^2y_1$ & $x_1y_1^2$ & $tx_1^3$ & $tx_1^2y_1$
\end{tabular}
}
\end{center}
Therefore $\mc{L}^S_{\bm{\lambda}}(x_1;y_1;t) = x_1^2y_1 + x_1y_1^2 + tx_1^3 + tx_1^2y_1$.  To compute $\mc{G}^{(2)}_{\lambda}(x_1;y_1;t^{1/2})$, we note that there are 4 super 2-ribbon tableaux of shape $\lambda$ in the alphabet $\{ 1 < 1' \}$.
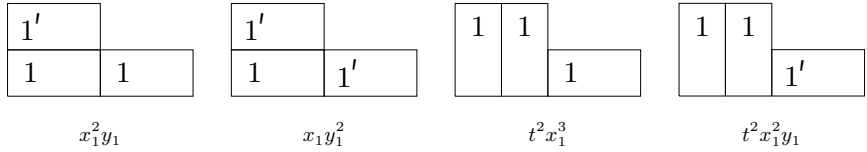
\begin{center}
\resizebox{0.8\textwidth}{!}{
\begin{tabular}{cccc}
	\resizebox{3cm}{!}{
	\begin{tikzpicture}[baseline=(current bounding box.center)]
	\draw (2,0)--(4,0)--(4,1)--(2,1)--(2,0);
	\draw (0,0)--(2,0)--(2,1)--(0,1)--(0,0);
	\draw (0,1)--(2,1)--(2,2)--(0,2)--(0,1);
	\node[scale=2] at (0.5,0.5) {$1$};
	\node[scale=2] at (0.5,1.5) {$1'$};
	\node[scale=2] at (2.5,0.5) {$1$};
	\end{tikzpicture} }
&
	\resizebox{3cm}{!}{
	\begin{tikzpicture}[baseline=(current bounding box.center)]
	\draw (2,0)--(4,0)--(4,1)--(2,1)--(2,0);
	\draw (0,0)--(2,0)--(2,1)--(0,1)--(0,0);
	\draw (0,1)--(2,1)--(2,2)--(0,2)--(0,1);
	\node[scale=2] at (0.5,0.5) {$1$};
	\node[scale=2] at (0.5,1.5) {$1'$};
	\node[scale=2] at (2.5,0.5) {$1'$};
	\end{tikzpicture} }
&
	\resizebox{3cm}{!}{
	\begin{tikzpicture}[baseline=(current bounding box.center)]
	\draw (2,0)--(4,0)--(4,1)--(2,1)--(2,0);
	\draw (0,0)--(0,2)--(1,2)--(1,0)--(0,0);
	\draw (1,0)--(1,2)--(2,2)--(2,0)--(1,0);
	\node[scale=2] at (0.5,1.5) {$1$};
	\node[scale=2] at (1.5,1.5) {$1$};
	\node[scale=2] at (2.5,0.5) {$1$};
	\end{tikzpicture} }
&
	\resizebox{3cm}{!}{
	\begin{tikzpicture}[baseline=(current bounding box.center)]
	\draw (2,0)--(4,0)--(4,1)--(2,1)--(2,0);
	\draw (0,0)--(0,2)--(1,2)--(1,0)--(0,0);
	\draw (1,0)--(1,2)--(2,2)--(2,0)--(1,0);
	\node[scale=2] at (0.5,1.5) {$1$};
	\node[scale=2] at (1.5,1.5) {$1$};
	\node[scale=2] at (2.5,0.5) {$1'$};
	\end{tikzpicture} }
\\ \\
$x_1^2y_1$ & $x_1y_1^2$ & $t^2x_1^3$ & $t^2x_1^2y_1$
\end{tabular}
}
\end{center}
Therefore $\mc{G}^{(2)}_{\lambda}(x_1;y_1;t^{1/2})= x_1^2y_1 + x_1y_1^2 + t^2x_1^3 + t^2x_1^2y_1$.  (The way we've ordered the lattice configurations and the super ribbon tableaux agrees with the bijection $\theta$ from Section \ref{sec:relate-ls-to-g}, so that the $i$-th lattice configuration corresponds to the $i$-th super ribbon tableaux via $\theta$.)

We see that $\mc{L}^S_{\bm{\lambda}}(x_1;y_1;t) = \mc{G}^{(2)}_{\lambda}(x_1;y_1;t^{1/2})$, so in fact $\square = 0$ in this case, which agrees with computing $\square$ using its definition.

\subsection{Example 2}
Let $\lambda = (8,7,3)$ and $\bm{\lambda} = ((1),(3),(2))$, noting $\bm{\lambda}$ is the 3-quotient of $\lambda$.
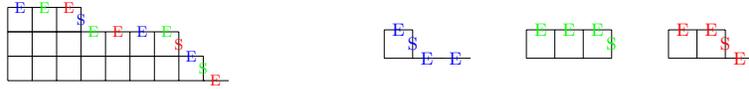
\[ \resizebox{3cm}{!}{
\begin{tikzpicture}[baseline=(current bounding box.center)]
\draw (0,0) grid (8,1); \draw (0,1) grid (7,2); \draw (0,2) grid (3,3); \draw (8,0) grid (9,0);
\node[scale=2,blue] at (0.5,3) {E}; \node[scale=2,green] at (1.5,3) {E};
\node[scale=2,red] at (2.5,3) {E};
\node[scale=2,blue] at (3,2.5) {S};
\node[scale=2,green] at (3.5,2) {E};
\node[scale=2,red] at (4.5,2) {E};
\node[scale=2,blue] at (5.5,2) {E}; \node[scale=2,green] at (6.5,2) {E};
\node[scale=2,red] at (7,1.5) {S};
\node[scale=2,blue] at (7.5,1) {E};
\node[scale=2,green] at (8,0.5) {S};
\node[scale=2,red] at (8.5,0) {E};
\end{tikzpicture} }
\hspace{2cm} 
\resizebox{5cm}{!}{
\begin{tikzpicture}[baseline=(current bounding box.center)]
\draw (0,0) grid (1,1); \draw (1,0) grid (3,0);
\node[blue,scale=2] at (0.5,1) {E}; \node[blue,scale=2] at (1,0.5) {S}; \node[blue,scale=2] at (1.5,0) {E};
\node[blue,scale=2] at (2.5,0) {E};
\draw (5+0,0) grid (5+3,1);
\node[green,scale=2] at (5+0.5,1) {E}; \node[green,scale=2] at (5+1.5,1) {E};
\node[green,scale=2] at (5+2.5,1) {E};
\node[green,scale=2] at (5+3,0.5) {S};
\draw (10+0,0) grid (10+2,1); \draw (10+2,0) grid (10+3,0);
\node[red,scale=2] at (10+0.5,1) {E}; \node[red,scale=2] at (10+1.5,1) {E}; \node[red,scale=2] at (10+2,0.5) {S};
\node[red,scale=2] at (10+2.5,0) {E};
\end{tikzpicture} }
\]
Therefore, by Proposition \ref{prop-L-equals-G}, we must have 
\[ \mc{L}^S_{\bm{\lambda}}(x_1;y_1;t) = t^\square \mc{G}^{(3)}_{\lambda}(x_1;y_1;t^{1/2}) \]
for some half-integer $\square \in \frac{1}{2}\mathbb{Z}$.  We can thus compute $\mc{G}^{(3)}_{\lambda}(x_1;y_1;t^{1/2})$ by computing $\mc{L}^S_{\bm{\lambda}}(x_1;y_1;t)$ and $\square$.

To compute $\mc{L}^S_{\bm{\lambda}}(x_1;y_1;t)$, we note that there are 8 configurations of the lattice $S_{1,1}(\bm{\lambda})$.
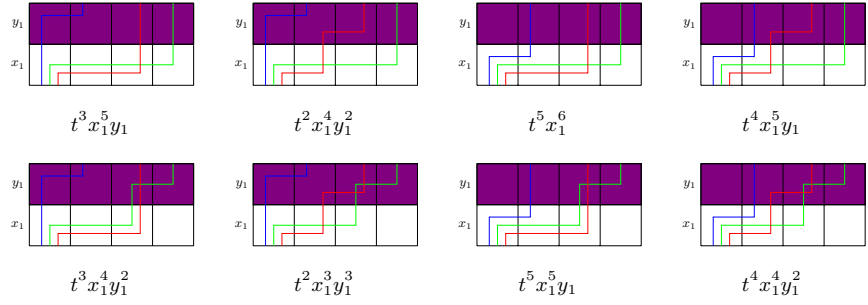
\begin{center}
\resizebox{0.8\textwidth}{!}{
\begin{tabular}{cccc}
	\resizebox{3cm}{!}{
	\begin{tikzpicture}[baseline=(current bounding box.center)]
	\draw[fill=violet] (0,1) rectangle (4,2);
	\draw (0,0) grid (4,2);
	\node [left] at (0,0.5) {$x_1$}; \node [left] at (0,1.5) {$y_1$}; 
	\draw[blue] (0.3,0)--(0.3,1.7)--(1.3,1.7)--(1.3,2);
	\draw[green] (0.5,0)--(0.5,0.5)--(3.5,0.5)--(3.5,2);
	\draw[red] (0.7,0)--(0.7,0.3)--(2.7,0.3)--(2.7,2);
	\end{tikzpicture} }
&
	\resizebox{3cm}{!}{
	\begin{tikzpicture}[baseline=(current bounding box.center)]
	\draw[fill=violet] (0,1) rectangle (4,2);
	\draw (0,0) grid (4,2);
	\node [left] at (0,0.5) {$x_1$}; \node [left] at (0,1.5) {$y_1$}; 
	\draw[blue] (0.3,0)--(0.3,1.7)--(1.3,1.7)--(1.3,2);
	\draw[green] (0.5,0)--(0.5,0.5)--(3.5,0.5)--(3.5,2);
	\draw[red] (0.7,0)--(0.7,0.3)--(1.7,0.3)--(1.7,1.3)--(2.7,1.3)--(2.7,2);
	\end{tikzpicture} }
&
	\resizebox{3cm}{!}{
	\begin{tikzpicture}[baseline=(current bounding box.center)]
	\draw[fill=violet] (0,1) rectangle (4,2);
	\draw (0,0) grid (4,2);
	\node [left] at (0,0.5) {$x_1$}; \node [left] at (0,1.5) {$y_1$}; 
	\draw[blue] (0.3,0)--(0.3,0.7)--(1.3,0.7)--(1.3,2);
	\draw[green] (0.5,0)--(0.5,0.5)--(3.5,0.5)--(3.5,2);
	\draw[red] (0.7,0)--(0.7,0.3)--(2.7,0.3)--(2.7,2);
	\end{tikzpicture} }
&
	\resizebox{3cm}{!}{
	\begin{tikzpicture}[baseline=(current bounding box.center)]
	\draw[fill=violet] (0,1) rectangle (4,2);
	\draw (0,0) grid (4,2);
	\node [left] at (0,0.5) {$x_1$}; \node [left] at (0,1.5) {$y_1$}; 
	\draw[blue] (0.3,0)--(0.3,0.7)--(1.3,0.7)--(1.3,2);
	\draw[green] (0.5,0)--(0.5,0.5)--(3.5,0.5)--(3.5,2);
	\draw[red] (0.7,0)--(0.7,0.3)--(1.7,0.3)--(1.7,1.3)--(2.7,1.3)--(2.7,2);
	\end{tikzpicture} }
\\ \\
$t^3x_1^5y_1$ & $t^2x_1^4y_1^2$ & $t^5x_1^6$ & $t^4x_1^5y_1$
\\ \\
	\resizebox{3cm}{!}{
	\begin{tikzpicture}[baseline=(current bounding box.center)]
	\draw[fill=violet] (0,1) rectangle (4,2);
	\draw (0,0) grid (4,2);
	\node [left] at (0,0.5) {$x_1$}; \node [left] at (0,1.5) {$y_1$}; 
	\draw[blue] (0.3,0)--(0.3,1.7)--(1.3,1.7)--(1.3,2);
	\draw[green] (0.5,0)--(0.5,0.5)--(2.5,0.5)--(2.5,1.5)--(3.5,1.5)--(3.5,2);
	\draw[red] (0.7,0)--(0.7,0.3)--(2.7,0.3)--(2.7,2);
	\end{tikzpicture} }
&
	\resizebox{3cm}{!}{
	\begin{tikzpicture}[baseline=(current bounding box.center)]
	\draw[fill=violet] (0,1) rectangle (4,2);
	\draw (0,0) grid (4,2);
	\node [left] at (0,0.5) {$x_1$}; \node [left] at (0,1.5) {$y_1$}; 
	\draw[blue] (0.3,0)--(0.3,1.7)--(1.3,1.7)--(1.3,2);
	\draw[green] (0.5,0)--(0.5,0.5)--(2.5,0.5)--(2.5,1.5)--(3.5,1.5)--(3.5,2);
	\draw[red] (0.7,0)--(0.7,0.3)--(1.7,0.3)--(1.7,1.3)--(2.7,1.3)--(2.7,2);
	\end{tikzpicture} }
&
	\resizebox{3cm}{!}{
	\begin{tikzpicture}[baseline=(current bounding box.center)]
	\draw[fill=violet] (0,1) rectangle (4,2);
	\draw (0,0) grid (4,2);
	\node [left] at (0,0.5) {$x_1$}; \node [left] at (0,1.5) {$y_1$}; 
	\draw[blue] (0.3,0)--(0.3,0.7)--(1.3,0.7)--(1.3,2);
	\draw[green] (0.5,0)--(0.5,0.5)--(2.5,0.5)--(2.5,1.5)--(3.5,1.5)--(3.5,2);
	\draw[red] (0.7,0)--(0.7,0.3)--(2.7,0.3)--(2.7,2);
	\end{tikzpicture} }
&
	\resizebox{3cm}{!}{
	\begin{tikzpicture}[baseline=(current bounding box.center)]
	\draw[fill=violet] (0,1) rectangle (4,2);
	\draw (0,0) grid (4,2);
	\node [left] at (0,0.5) {$x_1$}; \node [left] at (0,1.5) {$y_1$}; 
	\draw[blue] (0.3,0)--(0.3,0.7)--(1.3,0.7)--(1.3,2);
	\draw[green] (0.5,0)--(0.5,0.5)--(2.5,0.5)--(2.5,1.5)--(3.5,1.5)--(3.5,2);
	\draw[red] (0.7,0)--(0.7,0.3)--(1.7,0.3)--(1.7,1.3)--(2.7,1.3)--(2.7,2);
	\end{tikzpicture} }
\\ \\
$t^3x_1^4y_1^2$ & $t^2x_1^3y_1^3$ & $t^5x_1^5y_1$ & $t^4x_1^4y_1^2$
\end{tabular}
}
\end{center}
Therefore 
\[ \mc{L}^S_{\bm{\lambda}}(x_1;y_1;t) = 
t^3x_1^5y_1 + t^2x_1^4y_1^2 + t^5x_1^6 + t^4x_1^5y_1 +
t^3x_1^4y_1^2 + t^2x_1^3y_1^3 + t^5x_1^5y_1 + t^4x_1^4y_1^2. \]
We can also observe that the quantity
\[
\square = \frac{1}{2}\sum_{\alb} \left(
\# \begin{tikzpicture}[baseline=(current bounding box.center)]
\draw[fill=violet] (0,0) rectangle (1,1);
\draw[blue] (0.4,0)--(0.4,0.6)--(1,0.6);
\draw[red] (0.6,0)--(0.6,1);
\end{tikzpicture}
+
\# \begin{tikzpicture}[baseline=(current bounding box.center)]
\draw[fill=violet] (0,0) rectangle (1,1);
\draw[blue] (0,0.6)--(1,0.6);
\draw[blue] (0.4,0)--(0.4,1);
\draw[red] (0.6,0)--(0.6,1);
\end{tikzpicture}
-
\# \begin{tikzpicture}[baseline=(current bounding box.center)]
\draw[fill=violet] (0,0) rectangle (1,1);
\draw[blue] (0.4,0)--(0.4,1);
\draw[red] (0,0.4)--(0.6,0.4)--(0.6,1);
\end{tikzpicture}
-
\# \begin{tikzpicture}[baseline=(current bounding box.center)]
\draw[fill=violet] (0,0) rectangle (1,1);
\draw[blue] (0.4,0)--(0.4,1);
\draw[red] (0.6,0)--(0.6,1);
\draw[red] (0,0.4)--(1,0.4);
\end{tikzpicture}
\right) + 
\frac{1}{2} \sum_{\alb} \left(
\# \begin{tikzpicture}[baseline=(current bounding box.center)]
\draw (0,0)--(1,0)--(1,1)--(0,1)--(0,0);
\draw[blue] (1,0.6)--(0.5,0.6);
\draw[red] (0.6,0.5)--(0.6,1);
\end{tikzpicture}
-
\# \begin{tikzpicture}[baseline=(current bounding box.center)]
\draw (0,0)--(1,0)--(1,1)--(0,1)--(0,0);
\draw[blue] (0.4,0)--(0.4,0.5);
\draw[red] (0,0.4)--(0.5,0.4);
\end{tikzpicture}
\right) 
\] 
is equal to $\frac{1}{2}$ in each of these 8 configurations, so $\square = \frac{1}{2}$.

We can now compute
\[ \begin{aligned} 
\mc{G}^{(3)}_{\lambda}(x_1;y_1;t) =& t^{-2\square} \mc{L}^S_{\bm{\lambda}}(x_1;y_1;t^2) \\
=& t^5x_1^5y_1 + t^3x_1^4y_1^2 + t^9x_1^6 + t^7x_1^5y_1 +
t^5x_1^4y_1^2 + t^3x_1^3y_1^3 + t^9x_1^5y_1 + t^7x_1^4y_1^2.
\end{aligned} \]
Explicitly computing all the super 3-ribbon tableaux 
\begin{center}
\resizebox{0.8\textwidth}{!}{
\begin{tabular}{cccc}
	\resizebox{3cm}{!}{
	\begin{tikzpicture}[baseline=(current bounding box.center)]
	\draw (0,0)--(8,0)--(8,1)--(7,1)--(7,2)--(3,2)--(3,3)--(0,3)--(0,0);
	\draw (0,0)--(0,2)--(1,2)--(1,1)--(2,1)--(2,0)--(0,0);
	\draw (1,1)--(1,2)--(3,2)--(3,0)--(2,0)--(2,1)--(1,1);
	\draw (3,0)--(3,2)--(4,2)--(4,1)--(5,1)--(5,0)--(3,0);
	\draw (4,1)--(4,2)--(6,2)--(6,0)--(5,0)--(5,1)--(4,1);
	\draw (6,0)--(6,2)--(7,2)--(7,1)--(8,1)--(8,0)--(6,0);
	\draw (0,2)--(3,2)--(3,3)--(0,3)--(0,2);
	\node[scale=2] at (0.5,1.5) {$1$};
	\node[scale=2] at (1.5,1.5) {$1$};
	\node[scale=2] at (3.5,1.5) {$1$};
	\node[scale=2] at (4.5,1.5) {$1$};
	\node[scale=2] at (6.5,1.5) {$1$};
	\node[scale=2] at (0.5,2.5) {$1'$};
	\end{tikzpicture} }
&
	\resizebox{3cm}{!}{
	\begin{tikzpicture}[baseline=(current bounding box.center)]
	\draw (0,0)--(8,0)--(8,1)--(7,1)--(7,2)--(3,2)--(3,3)--(0,3)--(0,0);
	\draw (0,0)--(0,2)--(1,2)--(1,1)--(2,1)--(2,0)--(0,0);
	\draw (1,1)--(1,2)--(3,2)--(3,0)--(2,0)--(2,1)--(1,1);
	\draw (3,0)--(3,2)--(4,2)--(4,1)--(5,1)--(5,0)--(3,0);
	\draw (0,2)--(3,2)--(3,3)--(0,3)--(0,2);
	\draw (4,1)--(4,2)--(7,2)--(7,1)--(4,1);
	\draw (5,0)--(5,1)--(8,1)--(8,0)--(5,0);
	\node[scale=2] at (0.5,1.5) {$1$};
	\node[scale=2] at (1.5,1.5) {$1$};
	\node[scale=2] at (3.5,1.5) {$1$};
	\node[scale=2] at (4.5,1.5) {$1'$};
	\node[scale=2] at (5.5,0.5) {$1$};
	\node[scale=2] at (0.5,2.5) {$1'$};
	\end{tikzpicture} }
&
	\resizebox{3cm}{!}{
	\begin{tikzpicture}[baseline=(current bounding box.center)]
	\draw (0,0)--(8,0)--(8,1)--(7,1)--(7,2)--(3,2)--(3,3)--(0,3)--(0,0);
	\draw (0,0)--(1,0)--(1,3)--(0,3)--(0,0);
	\draw (1,0)--(2,0)--(2,3)--(1,3)--(1,0);
	\draw (2,0)--(3,0)--(3,3)--(2,3)--(2,0);
	\draw (3,0)--(3,2)--(4,2)--(4,1)--(5,1)--(5,0)--(3,0);
	\draw (4,1)--(4,2)--(6,2)--(6,0)--(5,0)--(5,1)--(4,1);
	\draw (6,0)--(6,2)--(7,2)--(7,1)--(8,1)--(8,0)--(6,0);
	\node[scale=2] at (0.5,2.5) {$1$};
	\node[scale=2] at (1.5,2.5) {$1$};
	\node[scale=2] at (2.5,2.5) {$1$};
	\node[scale=2] at (3.5,1.5) {$1$};
	\node[scale=2] at (4.5,1.5) {$1$};
	\node[scale=2] at (6.5,1.5) {$1$};
	\end{tikzpicture} }
&
	\resizebox{3cm}{!}{
	\begin{tikzpicture}[baseline=(current bounding box.center)]
	\draw (0,0)--(8,0)--(8,1)--(7,1)--(7,2)--(3,2)--(3,3)--(0,3)--(0,0);
	\draw (0,0)--(1,0)--(1,3)--(0,3)--(0,0);
	\draw (1,0)--(2,0)--(2,3)--(1,3)--(1,0);
	\draw (2,0)--(3,0)--(3,3)--(2,3)--(2,0);
	\draw (3,0)--(3,2)--(4,2)--(4,1)--(5,1)--(5,0)--(3,0);
	\draw (4,1)--(4,2)--(7,2)--(7,1)--(4,1);
	\draw (5,0)--(5,1)--(8,1)--(8,0)--(5,0);
	\node[scale=2] at (0.5,2.5) {$1$};
	\node[scale=2] at (1.5,2.5) {$1$};
	\node[scale=2] at (2.5,2.5) {$1$};
	\node[scale=2] at (3.5,1.5) {$1$};
	\node[scale=2] at (4.5,1.5) {$1'$};
	\node[scale=2] at (5.5,0.5) {$1$};
	\end{tikzpicture} }
\\ \\
$t^5x_1^5y_1$ & $t^3x_1^4y_1^2$ & $t^9x_1^6$ & $t^7x_1^5y_1$
\\ \\
	\resizebox{3cm}{!}{
	\begin{tikzpicture}[baseline=(current bounding box.center)]
	\draw (0,0)--(8,0)--(8,1)--(7,1)--(7,2)--(3,2)--(3,3)--(0,3)--(0,0);
	\draw (0,0)--(0,2)--(1,2)--(1,1)--(2,1)--(2,0)--(0,0);
	\draw (1,1)--(1,2)--(3,2)--(3,0)--(2,0)--(2,1)--(1,1);
	\draw (3,0)--(3,2)--(4,2)--(4,1)--(5,1)--(5,0)--(3,0);
	\draw (0,2)--(3,2)--(3,3)--(0,3)--(0,2);
	\draw (4,1)--(4,2)--(6,2)--(6,0)--(5,0)--(5,1)--(4,1);
	\draw (6,0)--(6,2)--(7,2)--(7,1)--(8,1)--(8,0)--(6,0);
	\node[scale=2] at (0.5,1.5) {$1$};
	\node[scale=2] at (1.5,1.5) {$1$};
	\node[scale=2] at (3.5,1.5) {$1$};
	\node[scale=2] at (4.5,1.5) {$1$};
	\node[scale=2] at (6.5,1.5) {$1'$};
	\node[scale=2] at (0.5,2.5) {$1'$};
	\end{tikzpicture} }
&
	\resizebox{3cm}{!}{
	\begin{tikzpicture}[baseline=(current bounding box.center)]
	\draw (0,0)--(8,0)--(8,1)--(7,1)--(7,2)--(3,2)--(3,3)--(0,3)--(0,0);
	\draw (0,0)--(0,2)--(1,2)--(1,1)--(2,1)--(2,0)--(0,0);
	\draw (1,1)--(1,2)--(3,2)--(3,0)--(2,0)--(2,1)--(1,1);
	\draw (3,0)--(3,2)--(4,2)--(4,1)--(5,1)--(5,0)--(3,0);
	\draw (0,2)--(3,2)--(3,3)--(0,3)--(0,2);
	\draw (4,1)--(4,2)--(7,2)--(7,1)--(4,1);
	\draw (5,0)--(5,1)--(8,1)--(8,0)--(5,0);
	\node[scale=2] at (0.5,1.5) {$1$};
	\node[scale=2] at (1.5,1.5) {$1$};
	\node[scale=2] at (3.5,1.5) {$1$};
	\node[scale=2] at (4.5,1.5) {$1'$};
	\node[scale=2] at (5.5,0.5) {$1'$};
	\node[scale=2] at (0.5,2.5) {$1'$};
	\end{tikzpicture} }
&
	\resizebox{3cm}{!}{
	\begin{tikzpicture}[baseline=(current bounding box.center)]
	\draw (0,0)--(8,0)--(8,1)--(7,1)--(7,2)--(3,2)--(3,3)--(0,3)--(0,0);
	\draw (0,0)--(1,0)--(1,3)--(0,3)--(0,0);
	\draw (1,0)--(2,0)--(2,3)--(1,3)--(1,0);
	\draw (2,0)--(3,0)--(3,3)--(2,3)--(2,0);
	\draw (3,0)--(3,2)--(4,2)--(4,1)--(5,1)--(5,0)--(3,0);
	\draw (4,1)--(4,2)--(6,2)--(6,0)--(5,0)--(5,1)--(4,1);
	\draw (6,0)--(6,2)--(7,2)--(7,1)--(8,1)--(8,0)--(6,0);
	\node[scale=2] at (0.5,2.5) {$1$};
	\node[scale=2] at (1.5,2.5) {$1$};
	\node[scale=2] at (2.5,2.5) {$1$};
	\node[scale=2] at (3.5,1.5) {$1$};
	\node[scale=2] at (4.5,1.5) {$1$};
	\node[scale=2] at (6.5,1.5) {$1'$};
	\end{tikzpicture} }
&
	\resizebox{3cm}{!}{
	\begin{tikzpicture}[baseline=(current bounding box.center)]
	\draw (0,0)--(8,0)--(8,1)--(7,1)--(7,2)--(3,2)--(3,3)--(0,3)--(0,0);
	\draw (0,0)--(1,0)--(1,3)--(0,3)--(0,0);
	\draw (1,0)--(2,0)--(2,3)--(1,3)--(1,0);
	\draw (2,0)--(3,0)--(3,3)--(2,3)--(2,0);
	\draw (3,0)--(3,2)--(4,2)--(4,1)--(5,1)--(5,0)--(3,0);
	\draw (4,1)--(4,2)--(7,2)--(7,1)--(4,1);
	\draw (5,0)--(5,1)--(8,1)--(8,0)--(5,0);
	\node[scale=2] at (0.5,2.5) {$1$};
	\node[scale=2] at (1.5,2.5) {$1$};
	\node[scale=2] at (2.5,2.5) {$1$};
	\node[scale=2] at (3.5,1.5) {$1$};
	\node[scale=2] at (4.5,1.5) {$1'$};
	\node[scale=2] at (5.5,0.5) {$1'$};
	\end{tikzpicture} }
\\ \\
$t^5x_1^4y_1^2$ & $t^3x_1^3y_1^3$ & $t^9x_1^5y_1$ & $t^7x_1^4y_1^2$
\end{tabular}
}
\end{center}
we see that this agrees.

\pagebreak

\bibliographystyle{abbrv}
\bibliography{SSLLT.bib,LLT.bib}

\end{document}